\documentclass[10pt,a4paper,final]{article}
\pdfoutput=1
\usepackage[latin1]{inputenc}
\usepackage{amsmath}
\usepackage{amssymb}
\usepackage{graphicx}
\usepackage{hyperref}
\usepackage{enumerate}
\usepackage{amsthm}
\usepackage{cite}
\usepackage{mathrsfs}
\usepackage{bbm}
\usepackage{subfigure}
\usepackage{float}
\usepackage{epsfig}
\usepackage[numbers]{natbib}
\usepackage{mathtools}
\usepackage{listings}
\usepackage{pgf}
\usepackage{listings}
\usepackage{epstopdf}
\usepackage[hmargin=3.5cm, vmargin=3.5cm]{geometry}
\usepackage{stmaryrd}
\usepackage{caption}
\usepackage{showkeys}

\theoremstyle{plain} 
\newtheorem{thm}{Theorem}[section] 
\newtheorem{prop}[thm]{Proposition}
\newtheorem{cor}[thm]{Corollary}
\newtheorem{lem}[thm]{Lemma}

\theoremstyle{definition}

\theoremstyle{remark} 
\newtheorem{oss}{Remark}

\renewcommand{\P}{\mathbb{P}}
\newcommand{\CC}{\mathcal{C}}
\newcommand{\Q}{\mathbb{Q}}
\newcommand{\R}{\mathbb{R}}
\newcommand{\E}{\mathbb{E}}
\newcommand{\F}{\mathcal{F}}
\newcommand{\ind}{\mathbbm{1}}
\newcommand{\eps}{\varepsilon}
\newcommand{\N}{\mathbb{N}}
\renewcommand{\d}{\,\mathrm{d}}
\newcommand{\G}{\mathbb{G}}
\newcommand{\calV}{\mathcal{V}}
\newcommand{\calA}{\mathcal{A}}

\newcommand{\sigmaUR}{\sigma_{\text{UR}}}
\newcommand{\sigmaUNR}{\sigma_{\text{UNR}}}
\newcommand{\sigmaNF}{\sigma_{\text{NF}}}
\newcommand{\sigmaA}{\sigma_{\text{A}}}
\newcommand{\bS}{\mathbb{S}}

\DeclareMathOperator{\Bin}{Bin}

\title{The probability of unusually large components for critical percolation on random $d$-regular graphs}
\author{Umberto De Ambroggio\thanks{University of Bath, Department of Mathematical Sciences, Bath BA2 7AY, UK. \texttt{umbidea@gmail.com}}\hspace{1.5mm} and Matthew I.~Roberts\thanks{University of Bath, Department of Mathematical Sciences, Bath BA2 7AY, UK. \texttt{mattiroberts@gmail.com}}}
\begin{document}

\maketitle

\begin{abstract}
Let $d\ge 3$ be a fixed integer, $p\in (0,1)$, and let $n\geq 1$ be a positive integer such that $dn$ is even. Let $\mathbb{G}(n, d, p)$ be a (random) graph on $n$ vertices obtained by drawing uniformly at random a $d$-regular (simple) graph on $[n]$ and then performing independent $p$-bond percolation on it, i.e.~we independently retain
each edge with probability $p$ and delete it with probability $1-p$. Let $|\mathcal{C}_{\text{max}}|$ be the size of the largest component in $\mathbb{G}(n, d, p)$. We show that, when $p$ is of the form $p=(d-1)^{-1}(1+\lambda n^{-1/3})$ for $\lambda\in\R$, and $A$ is large,
\begin{align*}
\mathbb{P}(|\mathcal{C}_{\text{max}}|>An^{2/3})\asymp A^{-3/2}e^{-\frac{A^3(d-1)(d-2)}{8d^2}+\frac{\lambda A^2(d-2)^2}{2d(d-1)}-\frac{\lambda^2 A(d-1)}{2(d-2)}}.
\end{align*}
This improves on a result of Nachmias and Peres. We also give an analogous asymptotic for the probability that a particular vertex is in a component of size larger than $An^{2/3}$.
	\vskip0.3cm
	\textbf{Keywords}: Random regular graph, percolation, component size, exploration process
\end{abstract}

\section{Introduction}
Let $d\ge 3$ be a fixed integer, and let $n\in \mathbb{N}$ be such that $dn$ is even. Let $p\in (0,1)$. We let $\mathbb{G}(n,d)$ be a $d$-regular graph sampled uniformly at random from the set of all $d$-regular graphs on $[n]$, and then denote by $\mathbb{G}(n,d,p)$ the random graph obtained by performing $p$-bond percolation on a realisation of $\mathbb{G}(n,d)$. That is, for each edge $e$ of $\mathbb{G}(n,d)$, we independently keep it with probability $p$ and delete it with probability $1-p$.

Alon, Benjamini and Stacey \cite{alon_benj_stacey} showed that $\mathbb{G}(n,d,\mu/(d-1))$ undergoes a phase transition as $\mu$ passes $1$: specifically, the size of the largest component $|\mathcal C_{\max}|$ is of order $\log(n)$ when $\mu < 1$, and of order $n$ when $\mu > 1$. 

A similar behaviour is shared by the Erd\H{o}s-R\'enyi random graph $\mathcal G(n,p)$. Indeed, it is well known (see e.g. the monographs \cite{bollobas_book}, \cite{remco:random_graphs} or \cite{janson_et_al:random_graphs} for more details) that, if $p=p(n)=\mu/n$, then $\mathcal G(n,p)$ undergoes a phase transition as $\mu$ passes 1. Specifically, if $\mu < 1$ then $|\mathcal C_{\max}|$ is of order $\log n$; if $\mu=1$ (the \textit{critical} case), then $|\mathcal C_{\max}|$ is of order $n^{2/3}$; and if $\mu > 1$, then $|\mathcal C_{\max}|$ is of order $n$.

Nachmias and Peres in \cite{nachmias:critical_perco_rand_regular} analysed the $\mathbb{G}(n,d,p)$ model near criticality. Amongst other results they proved that, if $p=(d-1)^{-1}(1+\lambda n^{-1/3})$ with $\lambda \in \mathbb{R}$ and $d\geq 3$ fixed, then there are positive constants $c(\lambda ,d)$ and $C(\lambda,d)$ such that, for any $A>0$ and all $n$, 
\begin{align}\label{uppernachmiasperes}
\mathbb{P}\left(|\mathcal C_{\max}|>An^{2/3}\right)\leq \frac{C(\lambda,d)}{A}e^{-c(\lambda,d)A^3}.
\end{align}
Furthermore, they also proved that there exists a positive constant $D(\lambda,d)$ such that for $\delta>0$ small enough and all $n$, 
\begin{align}\label{upperboundlowertail}
\mathbb{P}\left(|\mathcal C_{\max}|<\lceil \delta n^{2/3}\rceil\right)\leq  D(\lambda, d) \delta^{1/2},
\end{align}
thus showing that the largest component in this model, within the critical window, has size of order $n^{2/3}$, as for the Erd\H{o}s-R\'enyi random graph.

Partially motivated by studying a dynamical version of $\mathbb{G}(n,d,p)$ along the lines of the dynamical Erd\H{o}s-R\'enyi graph introduced by Roberts and {\c{S}}eng{\"{u}}l \cite{seng_rob:dynamical_ER}, our goal with this paper consists of determining the correct asymptotic order for the probability of observing maximal components containing significantly more than $n^{2/3}$ vertices. That is, we will prove a sharper version of \eqref{uppernachmiasperes} and a matching (up to a constant factor) lower bound.

We do this by adapting the methodology introduced in \cite{de_ambroggio_roberts:near_critical_ER} to study component sizes in the near-critical Erd\H{o}s-R\'enyi random graph, thus showing that the argument used there is robust and adaptable to other models of random graphs at criticality.

The main result of this paper is the following theorem. The reader may wish to begin by thinking of $\lambda$ as a constant in $\R$, or even taking $\lambda=0$. 

\begin{thm}\label{mainthm}
	Let $d\geq 3$ be fixed. There exists $A_0\in\N$ such that if $A=A(n)$ satisfies $A_0\le A = o(n^{1/30})$, and $p=p(n)=(1+\lambda n^{-1/3})(d-1)^{-1}$ where $\lambda=\lambda(n)$ satisfies $|\lambda|\leq A(1-2/d)[3(d-1)]^{-1}$, then for all sufficiently large $n$,
	\[\frac{c_1}{A^{1/2}n^{1/3}}e^{-G_{\lambda}(A,d)} \leq \mathbb{P}(|\mathcal{C}(V_n)|>An^{2/3})\leq \frac{c_2}{A^{1/2}n^{1/3}}e^{-G_{\lambda}(A,d)}\]
	and
	\[\frac{c_1}{A^{3/2}}e^{-G_{\lambda}(A,d)} \leq \mathbb{P}(|\mathcal{C}_{\max}|>An^{2/3})\leq \frac{c_2}{A^{3/2}}e^{-G_{\lambda}(A,d)},\]
	where
	\[G_{\lambda}(A,d)\coloneqq \frac{A^3(d-1)(d-2)}{8d^2}-\frac{\lambda A^2(d-1)}{2d}+\frac{\lambda^2 A(d-1)}{2(d-2)}\]
	and $c_1=c_1(d)>0$ and $c_2=c_2(d)>0$ are two finite constants that depend only on $d$. 
\end{thm}
We remark that our proof of the upper bounds in Theorem \ref{mainthm} is relatively straightforward. A key part of the argument will be a simple ballot-type result, established in \cite{de_ambroggio_roberts:near_critical_ER} and also used in \cite{de_ambroggio:component_sizes_crit_RGs} to provide simple (polynomial) upper bounds for the probability of observing unusually large maximal components in other critical models of random graphs.
 
Our proof of the lower bounds in Theorem \ref{mainthm} will be more complicated than that for the upper bound, although it still relies only on robust tools such as Brownian approximations to random walks, again along the lines of \cite{de_ambroggio_roberts:near_critical_ER}.

\subsection{Related work}

Nachmias and Peres \cite{nachmias:critical_perco_rand_regular}, as well as showing that the largest component within the critical window is of order $n^{2/3}$ as mentioned above, also considered the behaviour of the $\mathbb{G}(n,d,p)$ random graph outside the scaling window; see Theorems 3 and 4 in \cite{nachmias:critical_perco_rand_regular}. Moreover they established general upper bounds on the size of the largest component which are valid for all $d$-regular graphs; see Proposition 1 in in \cite{nachmias:critical_perco_rand_regular}. They also studied diameters and mixing times for this model within the critical window---see Corollary 6 in \cite{nachmias:critical_perco_rand_regular}---and established a distributional convergence for the sizes of all components, Theorem 5 in \cite{nachmias:critical_perco_rand_regular}.


Pittel \cite{pittel:largest_cpt_rg} is perhaps the earliest paper dealing with the problem of determining the probability of observing unusually large maximal components in critical random graphs. Pittel showed---among other results---that in the near-critical Erd\H{o}s-R\'enyi random graph $\mathcal{G}(n,p)$ with $p=1/n+\lambda/n^{4/3}$ and $\lambda \in \mathbb{R}$ fixed, 
\[\mathbb{P}(|\mathcal C_{\max}|> An^{2/3} ) \sim \frac{c}{A^{3/2}} e^{-\frac{A^3}{8}+\frac{\lambda A^2}{2}-\frac{\lambda^2A}{2}}\]
where $c$ is stated to equal $(2\pi)^{-1/2}$ but should be $(8/9\pi)^{1/2}$ due to a small oversight in the proof. More details, and a stronger result that allows $A$ and $\lambda$ to depend on $n$, are available in \cite{roberts:component_ER}. 

More recently, with the purpose of obtaining a simple probabilistic proof of the behaviour of the Erd\H{o}s-R\'enyi random graph near criticality, Nachmias and Peres \cite{nachmias_peres:CRG_mgs} introduced an argument using an exploration process and an associated martingale. With their method they did not obtain the correct asymptotic order of $\mathbb{P}(|\mathcal C_{\max}|> An^{2/3} )$ identified by Pittel, but their argument had the advantage of being very robust and adaptable, and has subsequently been used to analyse other models of random graphs at criticality; see e.g. \cite{nachmias:critical_perco_rand_regular}, \cite{hatami_molloy_conf}, and more recently \cite{de_ambroggio_pachon:upper_bounds_inhom_RGs}.

The current authors attempted to combine the advantages of the precise Pittel asymptotic, with a robust and adaptable probabilistic proof \`a la Nachmias and Peres, in \cite{de_ambroggio_roberts:near_critical_ER}. That paper constitutes the main source of inspiration for the proofs in the present paper.

See also Van der Hofstad, Kliem and Van Leeuwaarden \cite{hofstad_et_al:cluster_tails_power_law_RGs}, where similar results to those established by Pittel \cite{pittel:largest_cpt_rg} are proved in the context of inhomogeneous random graphs whose degrees obey a power law.

\subsection{Open problems}
	Our main result, Theorem \ref{mainthm}, does not identify the \textit{exact} asymptotic expansion for the probability of observing unusually large components, but gives bounds which are optimal up to multiplicative constants. The precise constant factor appearing in the asymptotic expansion is known for the Erd\H{o}s-R\'enyi graph (as mentioned above). One open problem is therefore to derive an exact asymptotic for the model studied in this paper, i.e.~to identify specific constants $\beta$ and $\gamma$ such that in the regime described in Theorem \ref{mainthm},
	\[\mathbb{P}(|\mathcal{C}(V_n)|>An^{2/3}) \sim \frac{\beta}{A^{1/2}n^{1/3}}e^{-G_{\lambda}(A,d)}\]
	and
	\[\mathbb{P}(|\mathcal{C}_{\max}|>An^{2/3}) \sim \frac{\gamma}{A^{3/2}}e^{-G_{\lambda}(A,d)}.\]

	We also remark that, in this paper, the parameter $d$ is considered fixed, and our proofs rely on this fact; but Theorem \ref{mainthm} is consistent with the Erd\H{o}s-R\'enyi case, in that if we formally substitute $d=n-1$ into Theorem \ref{mainthm} then we recover the analogous result for Erd\H{o}s-R\'enyi graphs \cite[Theorem 1.1]{de_ambroggio_roberts:near_critical_ER}. Another open problem is therefore to determine whether Theorem \ref{mainthm} holds when $1\ll d < n-1$.

\subsection{Graph-theoretic terminology and general notation}
Given an arbitrary set $S$, we denote by $|S|$ the number of elements contained in it. Let $G=(V,E)$ be any (undirected) graph. Given two vertices $u,v\in V$, we write $u\sim v$ if $\{u,v\}\in E$ and say that vertices $u$ and $v$ are neighbours. We often write $uv$ as shorthand for the edge $\{u,v\}$. We write $u\leftrightarrow v$ if there exists a path of occupied edges connecting vertices $u$ and $v$, where we adopt the convention that $v\leftrightarrow v$ for every $v\in V$. We denote by $\CC(v)\coloneqq \{u\in V:u\leftrightarrow v\}$ the component containing vertex $v\in V$. We define the largest component $\CC_{\max}$ to be some cluster $\CC(v)$ for which $|\CC(v)|$ is maximal, so that $|\CC_{\max}|=\max_{v\in V}|\CC(v)|$.

Given any $k\in \mathbb{N}=\{1,2,\dots\}$ we write $[k]\coloneqq \{1,\dots,k\}$. 
We denote by $\mathbb{N}_0$ the set of all non-negative integers. 
If $(x_n)_n$ and $(y_n)_n$ are two sequences of real numbers, we write $x_n=O(y_n)$ if there exists a finite constant $C>0$ (independent of $n$) such that $x_n\leq Cy_n$ for all large enough $n$. We write $x_n=\Theta(y_n)$ or $x_n\asymp y_n$ if $x_n=O(y_n)$ and $y_n=O(x_n)$. Sometimes we write $O_d(\cdot)$ and $\Theta_d(\cdot)$ to highlight the fact that the constants involved depend on the parameter $d$. Moreover, we write $x_n=o(y_n)$ or $x_n\ll y_n$ if $x_n/y_n\rightarrow 0$ as $n\rightarrow \infty$, and $x_n\sim y_n$ if $x_n/y_n\rightarrow 1$ as $n\rightarrow \infty$. We write $\Bin_{m,p}$ for a binomial random variable with parameters $m$ and $p$, and $U\sim U([0,1])$ for a random variable having a uniform distribution on $[0,1]$. When talking about random variables, the notation i.i.d.~stands for \textit{independent and identically distributed}. We will often write $c$, and sometimes $C$ or $c'$, to denote a finite, strictly positive constant which depends on the parameter $d$, and use $c$ many times in a single proof even though the constant may change from line to line.

\subsection{The configuration model}\label{config_model_sec}
The configuration model, which we describe below and which is due to Bollob\'as \cite{bollobas_config}, gives us a way of choosing a graph $\G(n,d)$ uniformly at random from the set of all $d$-regular graphs on $n$ vertices, provided that $dn$ is even.

Start with $dn$ stubs, labelled $(v,i)$ for $v\in[n]$ and $i\in[d]$. Choose a stub $(V_0,I_0)$ in some way (the manner of choosing may be deterministic or random) and pair it uniformly at random with another stub $(W_0,J_0)$. Say that these two stubs are \emph{matched} and put $\{V_0,W_0\}\in E$. Then at each subsequent step $k\in\{1,\ldots, nd/2-1\}$, choose a stub $(V_k,I_k)$ in some way from the set of unmatched stubs, and pair it uniformly at random with another unmatched stub $(W_k,J_k)$. Say that these two stubs are matched and put $\{V_k,W_k\}\in E$.

At the end of this process, the resulting object $G=([n],E)$ is uniformly chosen amongst all $d$-regular \emph{multigraphs} on $[n]$, i.e.~it may have multiple edges or self-loops. However, with probability converging to $\exp((1-d^2)/4)$ it is a simple graph, and conditioning on this event, it is uniformly chosen amongst all $d$-regular (simple) graphs on $[n]$.

\subsection{Structure of the paper}

The rest of the paper is organized as follows. In Section \ref{exploration_sec} we provide a constructive description of our model through an exploration process, which is a useful algorithmic procedure for revealing the component structure of the graph. We will then show how to relate the analysis of this exploration process to the size of $\CC(V_n)$, where $V_n$ represents a vertex selected uniformly at random from $[n]$. Then, in Section \ref{UBsec}, we prove the upper bounds in Theorem \ref{mainthm}, while the lower bounds are proved in Section \ref{LB_sec}.

\section{The exploration process}\label{exploration_sec}
We now specify a method for exploring the components of the graph $\mathbb{G}(n,d,p)$, which we recall is the random graph obtained by performing bond percolation with parameter $p$ on a realisation $\mathbb G(n,d)$ of a $d$-regular graph sampled uniformly at random from the set of all $d$-regular (simple) graphs on $[n]$. In fact, our exploration process will use the configuration model (see Section \ref{config_model_sec}) to generate components of $\G'(n,d,p)$, the $p$-percolated version of a uniformly random $d$-regular multigraph $\G'(n,d)$. When we talk about whether an edge of $\G'(n,d)$ is \emph{retained}, we mean whether it is present in $\G'(n,d,p)$. 

During our exploration process, each stub (or half-edge) of $\G'(n,d)$ is either \textit{active}, \textit{unseen} or \textit{explored}, and its status changes during the course of the process. We write $\mathcal{A}_{t}$, $\mathcal{U}_{t}$ and $\mathcal{E}_{t}$ for the sets of active, unseen and explored stubs at the end of the $t$-th step of the exploration process, respectively.

Given a stub $h$ of $\mathbb G'(n,d)$, we denote by $v(h)$ the vertex incident to $h$ (in other words, if $h=(u,i)$ for some $i$ then $v(h) = u$) and we write $\mathcal{S}(h)$ for the set of \textit{all} stubs incident to $v(h)$ in $\G'(n,d)$ (that is, $\mathcal{S}(h) = \{(v(h),i) : i\in[d]\}$; note in particular that $h\in \mathcal{S}(h)$).

The exploration process works as follows. Let $V_n$ be a vertex selected uniformly at random from $[n]$. At step $t=0$ we declare \textit{active} all stubs incident to $V_n$, while all the other $d(n-1)$ stubs are declared \textit{unseen}. Therefore we have that $|\mathcal{A}_0|=d$, $|\mathcal{U}_0|=d(n-1)$ and $|\mathcal{E}_0|=0$. For every $t\geq 1$, we proceed as follows.
\begin{itemize}
	\item [(a)] If $|\mathcal{A}_{t-1}|\geq 1$, we choose (in an arbitrary way) one of the active stubs, say $e_t$, and we pair it with a stub $h_t$ picked uniformly at random from $[dn]\setminus \left(\mathcal{E}_{t-1}\cup \{e_t\}\right)$, i.e. from the set of all unexplored stubs after having removed $e_t$. 
	\begin{itemize}
		\item [(a.1)] If $h_t\in \mathcal{U}_{t-1}$ \textbf{and} the edge $e_t h_t$ is retained in the percolation (the latter event occurs with probability $p$, independently of everything else), then all the unseen stubs in $\mathcal{S}(h_t)\setminus \{h_t\}$ are declared active, while $e_t$ and $h_t$ are declared explored. Formally we update
		\begin{itemize}
			\item $\mathcal{A}_t\coloneqq \left(\mathcal{A}_{t-1}\setminus \{e_t\}\right)\cup\left(\mathcal{U}_{t-1} \cap \mathcal{S}(h_t)\setminus \{h_t\}\right)$;
			\item $\mathcal{U}_t\coloneqq \mathcal{U}_{t-1}\setminus \mathcal{S}(h_t)$;
			\item $\mathcal{E}_t\coloneqq \mathcal{E}_{t-1}\cup\{e_t, h_t\}$.
		\end{itemize}
		\item [(a.2)] If $h_t\in \mathcal{U}_{t-1}$ \textbf{but} the edge $e_th_t$ is not retained in the percolation, then we simply declare $e_t$ and $h_t$ explored while the status of all other stubs remain unchanged. Formally we update 
		\begin{itemize}
			\item $\mathcal{A}_t\coloneqq \mathcal{A}_{t-1}\setminus \{e_t\}$;
			\item $\mathcal{U}_t\coloneqq \mathcal{U}_{t-1}\setminus \{h_t\}$;
			\item $\mathcal{E}_t\coloneqq \mathcal{E}_{t-1}\cup\{e_t,h_t\}$.
		\end{itemize}
	\item [(a.3)] If $h_t\in \mathcal{A}_{t-1}$, then we simply declare $e_t$ and $h_t$ explored while the status of all other stubs remain unchanged. Formally we update 
	\begin{itemize}
		\item $\mathcal{A}_t\coloneqq \mathcal{A}_{t-1}\setminus \{e_t,h_t\}$;
		\item $\mathcal{U}_t\coloneqq \mathcal{U}_{t-1}$;
		\item $\mathcal{E}_t\coloneqq \mathcal{E}_{t-1}\cup\{e_t,h_t\}$.
	\end{itemize}
	\end{itemize}
\item [(b)] If $|\mathcal{A}_{t-1}|=0$ \textbf{and} $|\mathcal{U}_{t-1}|\geq 1$, we pick (in an arbitrary way) an unseen stub $e_t$, we declare active \textit{all} the unseen stubs in $\mathcal{S}(e_t)$ (thus $e_t$ at least is declared active), so that the number of active stubs is non-zero, and then we proceed as in step (a).
\item [(c)] Finally, if $|\mathcal{A}_{t-1}|=0$ \textbf{and} $|\mathcal{U}_{t-1}|=0$, then all the stubs have been paired and we terminate the procedure.
\end{itemize}
For $t\geq 1$ we define the event $R_t\coloneqq \{e_t h_t\in \mathbb{G}'(n,d,p)\}$ that the edge $e_t h_t$ revealed during the $t$-th step of the exploration process is retained in the percolation.

Observe that, if $|\mathcal{A}_{t-1}|\geq 1$, then we can write
\begin{equation}\label{mainquantity}
\eta_t\coloneqq |\mathcal{A}_t|-|\mathcal{A}_{t-1}|=\ind_{\{h_t\in \mathcal{U}_{t-1}\}} \ind_{R_t}\left|\mathcal{S}(h_t)\cap \mathcal{U}_{t-1}\setminus \{h_t\}\right|-\ind_{\{h_t\in \mathcal{A}_{t-1}\}}-1.
\end{equation}
In words, assuming $|\mathcal{A}_{t-1}|\geq 1$, the number of active stubs at the end of step $t$ decreases by two if $h_t$ is an active stub; it decreases by one if $h_t$ is unseen and the edge $e_t h_t$ is not retained in the percolation, or if $h_t$ is the unique unseen stub incident to $v(h_t)$ and the edge $e_t h_t$ is retained in the percolation; and it increases by $m-2\in \{0,1,\dots,d-2\}$ if $v(h_t)$ has $m\in\{2,\ldots,d\}$ unseen stubs at the end of step $t-1$ and the edge $e_t h_t$ is retained in the percolation.

\subsection{Relating the exploration process to component sizes}

In order to describe the relationship between $|\CC(V_n)|$ and the exploration process that we have just illustrated, we need to introduce a few quantities. 

For $t\geq 0$, let us denote by $\mathcal{V}^{(m)}_{t}$ the set of vertices that have $m\in \{0,1,\ldots,d\}$ unseen stubs after the completion of step $t$ in the exploration process. Since vertices with $d$ unseen stubs play an important role in our analysis, we give them a name: we say that a vertex is \textit{fresh} if it possesses $d$ unseen stubs, i.e.~is an element of $\mathcal{V}^{(d)}_{t}$. We then define:
\begin{itemize}
	\item [(i)] $\tau \coloneqq \inf \{t\geq 1: |\mathcal{A}_t|=0\}$, the first time at which the set of active stubs is empty;
	\item [(ii)] $\sigmaUR \coloneqq \left| \left\{t\in [\tau]:h_t\in \mathcal{U}_{t-1},\,e_th_t\in \mathbb{G}'(n,d,p)\right\}\right|$, the number of steps $t\leq \tau$ in which the stub $h_t$ picked uniformly at random (from the set of unexplored stubs) is unseen and the edge $e_th_t$ is retained in the percolation;
	\item [(iii)] $\sigmaUNR \coloneqq \left| \left\{t\in [\tau]:h_t\in \mathcal{U}_{t-1},\,e_th_t\notin \mathbb{G}'(n,d,p)\right\}\right|$, the number of steps $t\leq \tau$ in which the stub $h_t$ is unseen and the edge $e_th_t$ is \emph{not} retained in the percolation;
	\item [(iv)] $\sigmaA\coloneqq \left| \left\{t\in [\tau]:h_t\in \mathcal{A}_{t-1} \right\}\right|$, the number of steps $t\leq \tau$ in which the stub $h_t$ is active;
	\item [(v)] $\sigmaNF\coloneqq \left| \left\{t\in [\tau]:h_t\in \bigcup_{m=0}^{d-1}\mathcal{V}^{(m)}_{t-1}\right\}\right|$, the number of steps $t\leq \tau$ in which the stub $h_t$ is incident to a vertex with $m\le d-1$ unseen stubs, i.e.~is not fresh.
\end{itemize}
The relationship between the size of $\CC(V_n)$ and the random variables we have just defined is illustrated in the following result, which corresponds to Lemma 10 in \cite{nachmias:critical_perco_rand_regular}. Since our terminology is different from that in \cite{nachmias:critical_perco_rand_regular} and because we only use part of their argument, we include a proof here for the reader's convenience. 

\begin{lem}[Nachmias and Peres \cite{nachmias:critical_perco_rand_regular}]\label{relationship}
	We have that $|\CC(V_n)|=\sigmaUR+1$ and, moreover,
	\begin{align*}
	\frac{\tau-d}{d-1}\leq \frac{\tau+\sigmaA-d}{d-1}\leq \sigmaUR\leq \frac{\tau+\sigmaA}{d-1}+\sigmaNF.
	\end{align*}
\end{lem}

\begin{proof}
	First we observe that at each step $t$ in which $h_t$ is unseen and the edge $e_th_t$ is retained in the percolation, we add one vertex to our currently explored component, and this is the only way in which vertices can be added to the current component. Thus $|\CC(V_n)|=\sigmaUR+1$ (the $+1$ comes from counting $V_n$).
	
	Denote by $N_m$ the number of steps $t\leq \tau$ in which $h_t$ is incident to a vertex having $m$ unseen stubs at the end of step $t-1$, and the edge $e_t h_t$ is retained in the percolation; formally,
	\[N_m\coloneqq \left| \left\{t\in [\tau]:h_t\in \mathcal{V}^{(m)}_{t-1},\, e_th_t\in \mathbb{G}'(n,d,p)\right\}\right|.\]
	Since at each step $t\in [\tau]$ in which $h_t$ is active we remove two half-edges from the set of active stubs, whereas at each step $t\in [\tau]$ in which $h_t$ is unseen and $e_th_t$ is not retained in the percolation we remove one half-edge from the set of active stubs, we see that
	\begin{equation}\label{identityforthevariousquantities}
	0=|\mathcal{A}_{\tau}|=d-2\sigmaA-\sigmaUNR+\sum_{m=1}^{d}(m-2)N_m.
	\end{equation}
	
	Next observe that
	\begin{align}\label{sumbound}
	\sum_{m=1}^{d}(m-2)N_m\leq (d-2)\sum_{m=1}^{d}N_m\leq (d-2)\sigmaUR,
	\end{align}
	where the second inequality in (\ref{sumbound}) is due to the fact that, if $h_t$ is in $\mathcal{V}^{(m)}_{t-1}$ for some $m\in [d]$, then $h_t$ must be unseen (as if $h_t$ is active then all the stubs adjacent to $v(h_t)$ must be active or explored). Combining (\ref{identityforthevariousquantities}) and (\ref{sumbound}) together with the identity $\tau=\sigmaUNR+\sigmaUR+\sigmaA$, which holds since at each step $t\in [\tau]$ we have either $h_t\in \mathcal{U}_{t-1}$ or $h_t\in \mathcal{A}_{t-1}$, yields
	\begin{align*}
	0\leq d-2\sigmaA-\sigmaUNR+(d-2)\sigmaUR=d-\sigmaA-\tau+(d-1)\sigmaUR
	\end{align*}
	whence
	\[\frac{\sigmaA+\tau-d}{d-1}\leq \sigmaUR.\]
	Since $\sigmaA\geq 0$ we arrive at $\sigmaUR\geq(\tau-d)/(d-1)$, and we have proved the first two inequalities in the statement of our lemma.
	
	To establish the upper bound for $\sigmaUR$, we start by observing that
	\begin{align}\label{lowerforsum}
	\sum_{m=1}^{d}(m-2)N_m\geq (d-2)N_d-\sum_{m=1}^{d-1}N_m\geq (d-2)\sigmaUR-(d-1)\sigmaNF.
	\end{align}	
	Combining (\ref{identityforthevariousquantities}) and (\ref{lowerforsum}) we obtain
	\begin{align*}
	0\geq d-2\sigmaA-\sigmaUNR+(d-2)\sigmaUR-(d-1)\sigmaNF,
	\end{align*}
	which together with the identity $\tau=\sigmaUNR+\sigmaUR+\sigmaA$ mentioned above gives
	\[	(d-2)\sigmaUR\leq -d+2\sigmaA+\sigmaUNR+(d-1)\sigmaNF
	=-d+\sigmaA+\tau-\sigmaUR+(d-1)\sigmaNF.\]
	Rearranging and ignoring the $-d$ term, we have
	\begin{align*}
	(d-1)\sigmaUR\leq \sigmaA+\tau+(d-1)\sigmaNF,
	\end{align*}
	and dividing both sides by $d-1$ yields the required upper bound on $\sigmaUR$.	
\end{proof}
Let $k=k_n\in \mathbb{N}$. Thanks to Lemma \ref{relationship} we can write 
\begin{equation}\label{relationshipLB}
\mathbb{P}(|\CC(V_n)|>k)=\mathbb{P}(\sigmaUR\geq k)\geq \mathbb{P}(\tau\geq (d-1)k + d) = \mathbb{P}(\tau> (d-1)(k+1)).
\end{equation}
Consequently, in order to bound from below the probability that $\CC(V_n)$ contains more than $k$ vertices, it suffices to provide a lower bound for the probability that the number of active stubs stays positive for all times $t\leq (d-1)(k+1)$. 

To establish un upper bound for the probability that $|\CC(V_n)|$ is larger than $k$ in terms of the stopping time $\tau$, the argument is slightly more involved. Recall that, by Lemma \ref{relationship},
\[|\mathcal C(V_n)| - 1 = \sigmaUR\leq \frac{\tau+\sigmaA}{d-1}+\sigmaNF.\]
The idea is that the random variables $\sigmaNF$ and $\sigmaA$, which appear in the upper bound for $\sigmaUR$, are of much smaller order than $\tau$, so that we expect $\mathbb{P}(\sigmaUR\geq k)$ to be of the same order as the probability that $|\mathcal{A}_t|$ stays positive for roughly $(d-1)k$ steps. Our next result makes this precise.

\begin{lem}\label{mylemmaupper}
	Let $k=k_n\in \mathbb{N}$ and $m=m_n\in \mathbb{N}$ be such that $\max\{k,n^{1/2}\}\ll m\ll n$. Then
	\begin{equation}\label{uppersigma}
	\mathbb{P}(\sigmaUR\geq k)\leq 2\mathbb{P}\left(\tau\geq (d-1)k-\frac{4(d-1)m^2}{n}\right)+\exp\left(-\frac{c m^2}{n}\right)
	\end{equation}
	for some constant $c>0$.
\end{lem}

\begin{oss}
When we apply this lemma, we will choose $m^2/n\ll k$, so that the quantity $(d-1)k-4m^2/n$ which appears within the probability on the right-hand side of (\ref{uppersigma}) is asymptotically equivalent to $(d-1)k$, thus making sense of our previous claim that we expect $\mathbb{P}(\sigmaUR\geq k)\asymp \mathbb{P}(\tau\geq (d-1)k)$.
\end{oss}

\begin{proof}
	Observe that, using the upper bound for $\sigmaUR$ established in Lemma \ref{relationship}, we have
	\begin{align*}
	\mathbb{P}(\sigmaUR\geq k)\leq \mathbb{P}(\tau\geq (d-1)k -[\sigmaA+(d-1)\sigmaNF]).
	\end{align*}
	Recall that, by definition,
	\begin{equation*}
	\sigmaNF= \left| \left\{t\in [\tau]:h_t\in \bigcup_{m=1}^{d-1}\mathcal{V}^{(m)}_{t-1}\right\}\right|\text{ and }\sigmaA= \left| \left\{t\in [\tau]:h_t\in \mathcal{A}_{t-1} \right\}\right|.
	\end{equation*}
	By Lemma 16 in \cite{nachmias:critical_perco_rand_regular} (whose proof only uses elementary bounds) we know that, setting
	\begin{align*}
	X_m\coloneqq \left\{t\in [m]: h_t\in \mathcal{A}_{t-1} \text{ or } h_t\in \bigcup_{m=1}^{d-1}\mathcal{V}^{(m)}_{t-1}\right\},
	\end{align*}
	we have
	\begin{align}\label{expupperxm}
	\mathbb{P}(|X_m|>4m^2/n)\leq e^{-c m^2 / n}
	\end{align}
	for some constant $c>0$. Note that $|X_m|$ is almost equal to $\sigmaNF+\sigmaA$, the only difference being that in the definition of $X_m$ we consider the first $m$ steps of the exploration process, while in the definitions of $\sigmaNF$ and $\sigmaA$ we look at all steps until time $\tau$. With the purpose of replacing the random variables $\sigmaNF$ and $\sigmaA$ with $|X_m|$ (which we know how to control), we write
	\begin{multline*}
	\mathbb{P}(\tau\geq (d-1)k-[\sigmaA+(d-1)\sigmaNF])\\
	\leq \mathbb{P}(\tau\geq (d-1)k-[\sigmaA+(d-1)\sigmaNF],\tau\leq m)+\mathbb{P}(\tau>m).
	\end{multline*}
	Observe that, on the event where $\tau\leq m$, we have that
	\begin{equation*}
	\sigmaNF \leq \left| \left\{t\in [m]:h_t\in \bigcup_{m=1}^{d-1}\mathcal{V}^{(m)}_{t-1}\right\}\right|\eqqcolon \widehat{\sigma}_{\text{NF}} \,\,\,\text{ and }\,\,\, \sigmaA\leq  \left| \left\{t\in [m]:h_t\in \mathcal{A}_{t-1} \right\}\right|\eqqcolon \widehat{\sigma}_{\text{A}}.
	\end{equation*}
	Therefore we can bound
	\begin{align*}
	\mathbb{P}(\tau\geq (d-1)k-[\sigmaA+(d-1)&\sigmaNF],\tau\leq m)\leq \mathbb{P}(\tau\geq (d-1)k-[\widehat{\sigma}_{\text{A}}+(d-1)\widehat{\sigma}_{\text{NF}}])
	\end{align*}
	and hence we obtain
	\begin{align*}
	\mathbb{P}(\sigmaUR\geq k)\leq \mathbb{P}(\tau\geq (d-1)k-[\widehat{\sigma}_{\text{A}}+(d-1)\widehat{\sigma}_{\text{NF}}])+\mathbb{P}(\tau>m).
	\end{align*}
	Note that
	\begin{multline*}
	\mathbb{P}(\tau\geq (d-1)k-[\widehat{\sigma}_{\text{A}}+(d-1)\widehat{\sigma}_{\text{NF}}])\\
	\leq \mathbb{P}(\tau\geq (d-1)k-4(d-1)m^2/n)+\mathbb{P}(\widehat{\sigma}_{\text{A}}+(d-1)\widehat{\sigma}_{\text{NF}}> 4(d-1)m^2/n).
	\end{multline*}
	Now since $\widehat{\sigma}_{\text{A}}+\widehat{\sigma}_{\text{NF}}=|X_m|$ we can use (\ref{expupperxm}) to conclude that
	\begin{align*}
	\mathbb{P}(\widehat{\sigma}_{\text{A}}+(d-1)\widehat{\sigma}_{\text{NF}}> 4(d-1)m^2/n)\leq \mathbb{P}\left(\widehat{\sigma}_{\text{A}}+\widehat{\sigma}_{\text{NF}}> \frac{4m^2}{n}\right)\leq e^{-cm^2/n}
	\end{align*}
	and hence we obtain
	\[\mathbb{P}(\sigmaUR\geq k)\leq \mathbb{P}(\tau\geq (d-1)k-4(d-1)m^2/n) +e^{-cm^2/n}+\mathbb{P}(\tau>m).\]
	The proof of the lemma is completed after noticing that, since $m\gg k$,
	\begin{align*}
	\mathbb{P}(\tau>m)\leq \mathbb{P}(\tau\geq (d-1)k-4(d-1)m^2/n)
	\end{align*}
	for all large enough $n$.
\end{proof}

We now apply this lemma and combine with \eqref{relationshipLB} to obtain our desired relationship between the size of $\CC(V_n)$ and the time $\tau$ at which the number of active stubs in our exploration process hits zero for the first time.

\begin{cor}\label{relationship_cor}
For $n^{1/2}\ll k\ll n^{3/4}$ and sufficiently large $n$, we have
\[\mathbb{P}(\tau> (d-1)(k+1)) \leq \P(|\CC(V_n)|>k) \leq 2\mathbb{P}(\tau\ge (d-1)k - n^{1/2})+ e^{-c n^{1/2}}\]
for some constant $c>0$ depending only on $d$.
\end{cor}

\begin{proof}
As noted in \eqref{relationshipLB}, the first inequality follows from Lemma \ref{relationship}. For the second inequality, applying Lemma \ref{mylemmaupper} with $m=\lfloor n^{3/4}/2\sqrt{d-1}\rfloor$ we obtain that for $k\gg n^{1/2}$,
\begin{equation*}
\P(\sigmaUR\geq k) \le 2\mathbb{P}(\tau\geq (d-1)k-n^{1/2}) + e^{-c n^{1/2}}.
\end{equation*}
Recalling that $\mathbb{P}(|\CC(V_n)|>k)=\mathbb{P}(\sigmaUR\geq k)$, the proof is complete.
\end{proof}

Corollary \ref{relationship_cor} has translated the problem of studying the probability that $\CC(V_n)$ is larger than $k$, for $n^{1/2}\ll k \ll n^{3/4}$, to that of studying the probability that the number of active vertices $|\mathcal{A}_t|=d+\sum_{i=1}^{t}\eta_i$ stays positive for all times $1\leq t \le (d-1)k + o(k)$.

The goal is then to bound from above and below the sequence of $\eta_i$ with sequences of random variables which are sufficiently \textit{close} to the $\eta_i$ and that, at the same time, are easier to analyse.

 We introduce here three events which appear very often in the following sections. Specifically, we denote by $F_i$ the event that vertex $v(h_i)$ is fresh (i.e.~has $d$ unseen stubs) at the end of step $i-1$, while we write $F'_i$ for the event that $v(h_i)$ has $d-1$ unseen stubs at the end of step $i-1$ and $F^-_i$ for the event that $v(h_i)$ has $m\in [d-2]$ unseen stubs at the end of step $i-1$. More formally, for $i\geq 1$ we set
\[F_i = \{v(h_i)\in \calV_{i-1}^{(d)}\}, \,\,\,\, F_i' = \{v(h_i)\in\calV_{i-1}^{(d-1)}\} \,\,\text{ and }\,\, F^-_i = \bigg\{v(h_i)\in\bigcup_{m=1}^{d-2}\calV_{i-1}^{(m)}\bigg\}.\]

\subsection{Proof Ideas}

We concentrate first on establishing our result for $\CC(V_n)$. We will then deduce from this the result for $\mathcal{C}_{\max}$. We begin by applying Corollary \ref{relationship_cor}, thanks to which our problem reduces to establishing upper and lower bounds for the probability that the integer-valued random process $(d+\sum_{i=1}^{t}\eta_i)_t$ stays positive up to time $T \approx (d-1)An^{2/3}$. 

For the upper bound, one of our main tools is a ballot-type estimate introduced in \cite{de_ambroggio_roberts:near_critical_ER}, which allows us to bound from above the probability that a random walk started at $d$ stays positive up to time $T$ and finishes at some level $j\gg T^{1/2}$. Hence our main task for the upper bound consists of approximating the $\eta_i$, which are not independent or identically distributed, with i.i.d.~random variables, and then applying the ballot-type result to the new sequence.

For the lower bound, the analysis is more involved. Since the $\eta_i$ are not i.i.d., we again have to approximate---this time from below---to turn the process $(d+\sum_{i=1}^{t}\eta_i)_{t\in [T]}$ into a random walk over the whole interval $[T]$.

However, this time we split the time interval $[T]$ into two disjoint intervals $[T']$ and $[T]\setminus [T']$, where $T'\ll T$, and then use different techniques to control the process on these intervals. We use a rougher approximation over the first interval $[T']$, and then use known bounds to estimate the probability that the resulting random walk stays positive up to time $T'$ and finishes at distance of order $(T')^{1/2}$ from the origin. We then use a more accurate random walk approximation to $\eta_i$ on the second interval, and estimate the probability that this random walk stays positive by comparing it with a Brownian motion. We must also show that the two approximating random walks are constructed in such a way that the probability that $(d+\sum_{i=1}^{t}\eta_i)_t$ stays positive for all times $t\in [T]$ can be split into a product of two terms, the probability that the first random walk stays positive and finishes in a certain region, multiplied by the probability that the second random walk stays positive starting from within that region.



\vskip0.5cm
\begin{center}
	\textbf{Notation summary}
\end{center}
In order to facilitate reading the rest of the paper, we summarize here the quantities and events that have been introduced in this section.
\begin{enumerate}
	\item $\mathcal{A}_{t}$ represents the set of \textit{active} stubs at the end of the $t$-th step of the exploration process;
	\item $\mathcal{U}_{t}$ represents the set of \textit{unseen} stubs at the end of the $t$-th step of the exploration process;
	\item $\mathcal{E}_{t}=[dn]\setminus \left(\mathcal{A}_t\cup \mathcal{U}_t \right)$ represents the set of \textit{explored} stubs at the end of the $t$-th step of the exploration process;
	\item given a stub $s$, we denote by $v(s)$ the vertex incident to $s$;
	\item given a stub $s$, we denote by $\mathcal{S}(s)$ the set of \textit{all} stubs incident to $v(s)$ (note that $s\in \mathcal{S}(s)$ according to our definition);
	\item $R_t$ is the event that the edge $e_th_t$ is retained in the percolation.
	\item $\mathcal{V}^{(m)}_{t}$ denotes the set of vertices with $m\in [d]$ unseen stubs after the completion of step $t$ in the exploration process;
	\item vertices in $\mathcal{V}^{(d)}_t$ are called fresh;
	\item $F_t$ is the event that $v(h_t)$ is fresh after the completion of step $t-1$ (i.e. $v(h_t)\in \mathcal{V}^{(d)}_{t-1}$);
	\item $F'_t$ is the event that $v(h_i)$ has $d-1$ unseen stubs at the end of step $t-1$;
	\item $F^-_t$ is the event that $v(h_i)$ has $m\in [d-2]$ unseen stubs at the end of step $t-1$.
\end{enumerate}

\section{Proof of Theorem \ref{mainthm}: upper bounds}\label{UBsec}
Let $T=\lfloor(d-1) An^{2/3} \rfloor- \lceil n^{1/2}\rceil - 1$. Throughout this section, the letter $c$ denotes a (positive) numerical constant that might depend on $d$ and can change from line to line.

Recalling Corollary \ref{relationship_cor}, here we want to bound from above the probability 
\begin{align*}
\mathbb{P}(|\CC(V_n)|> An^{2/3})\leq 2\mathbb{P}(\tau > T) + e^{-cn^{1/2}} = 2\mathbb{P}\left(d+\sum_{i=1}^{t}\eta_i>0\hspace{0.15cm}\forall t\in [T]\right) + e^{-cn^{1/2}}.
\end{align*}
As a first step in this direction we introduce a new sequence of random variables, larger than the $\eta_i$ and easier to analyse.

Recalling the definition of the random variable $\eta_i$ given in (\ref{mainquantity}) we see that, if $|\mathcal{A}_{i-1}|\geq 1$, then 
\begin{align}\label{ineqdelta}
\eta_i&\leq \ind_{R_i}(d-2)+\ind_{R_i}\ind_{F_i}-1\eqqcolon \eta'_i.
\end{align}
(This bound does not hold if $|\mathcal{A}_{i-1}|=0$, but since we are evaluating the probability that the number of active stubs remains positive at all times $t\in [T]$, this does not concern us.)
To see why (\ref{ineqdelta}) is true, first of all notice that if $h_i\in \mathcal{A}_{i-1}$ (i.e. $h_i$ is active) then $\eta_i=-2<-1\leq \eta'_i$. If $h_i\in \mathcal{U}_{i-1}$ (i.e. $h_i$ is unseen) but the edge $e_ih_i$ is not retained in the percolation, then $\eta_i=-1=\eta'_i$. If $h_i\in \bigcup_{m=1}^{d-1}\mathcal{V}^{(m)}_{i-1}$ (whence $h_i\in \mathcal{U}_{i-1}$) and $e_ih_i$ is retained in the percolation, then $\eta_i=(m-1)-1\leq (d-2)-1=d-3=\eta'_i$. Finally, if $h_i\in \mathcal{V}^{(d)}_{i-1}$ (whence $h_i\in \mathcal{U}_{i-1}$) and $e_ih_i$ is retained in the percolation, then $\eta_i=d-2=\eta'_i$. 

In practice, working with $\eta_i'$ is like assuming that all non-fresh vertices are in $\mathcal{V}^{(d-1)}_{i-1}$, or equivalently that all vertices have either $d$ or $d-1$ unseen stubs.

We can therefore bound
\begin{equation}\label{secondbound}
\mathbb{P}(|\CC(V_n)|>\lfloor An^{2/3} \rfloor)\leq \mathbb{P}\left(d+\sum_{i=1}^{t}\eta_i>0\hspace{0.15cm}\forall t\in [T]\right)\leq \mathbb{P}\left(d+\sum_{i=1}^{t}\eta'_i>0\hspace{0.15cm}\forall t\in [T]\right).
\end{equation}
Although $\eta'_i$ is simpler than $\eta_i$, in order to bound from above the probability on the right-hand side of (\ref{secondbound}), it would be convenient to turn the $\eta'_i$ into (larger) independent random variables. To achieve this, the idea is to substitute the dependent indicators $\ind_{F_i}$ that appear in the definition of $\eta'_i$ with other, independent $\{0,1\}$-valued random variables. To this end notice that, conditional on everything that occurred up to the end of step $i-1$ in the exploration process, vertex $v(h_i)$ is fresh with probability 
\begin{align*}
\frac{d \big|\mathcal{V}^{(d)}_{i-1}\big|}{dn-2(i-1)-1}.
\end{align*}
This is because $h_i$ is chosen uniformly from amongst all unexplored stubs, of which there are exactly $dn-2(i-1)-1$ at the end of step $i-1$. Thus, in order to substitute the $\ind_{F_i}$ with (larger) independent indicator random variables, we need an upper bound for the number of fresh vertices that we expect to observe at each step $i\in \{0\}\cup [T-1]$ of the exploration process.

Our next result, whose proof is postponed to Section \ref{concentration_sec}, states that it is very unlikely to have more than $n-1-i+i^2/2n$ fresh vertices at the $i$-th step of the exploration process, for all $i\in [T-1]$.
\begin{lem}\label{numberfresh}
	Suppose that $A=o(n^{1/12})$ as $n\rightarrow \infty$, and let $a_n(i)\coloneqq n-1-i+i^2/2n$. Then, for every $m\geq 1$ and all large enough $n$, we have that
	\begin{equation}\label{boundonfresh}
	\mathbb{P}\left(\exists i\in [T-1]: |\mathcal{V}^{(d)}_{i}|> a_n(i)+ m\right)\leq cTe^{-\frac{mn^{1/2}}{T}}
	\end{equation}
	where $c=c(d)$ is some finite constant that depends only on $d$.
\end{lem}\label{ossfornumberfresh}
\begin{oss}
	Since $T\sim (d-1)An^{2/3}$, the exponent in (\ref{boundonfresh}) is of order $ m/An^{1/6}$. Thus, since in the statement of our main Theorem \ref{mainthm} we have assumed that $A\ll n^{1/30}$, taking $m=m(n)=An^{4/15}$ we see  that the quantity on the right-hand side of (\ref{boundonfresh}) is much smaller than the (upper) bounds stated in Theorem \ref{mainthm}, provided $n$ is large enough. We also remark that (\ref{boundonfresh}) is not the best possible upper bound, but for our purpose it suffices. 
\end{oss}

In line with the remark above, the reader should think of $m=m(n)=An^{4/15}$ in what follows. Keeping in mind that, by Lemma \ref{numberfresh}, the number of fresh vertices satisfies $|\mathcal{V}^{(d)}_{i}|\leq a_n(i)+m$ for all $i\in [T-1]$ with high probability, we bound
\begin{align}\label{pp}
\nonumber&\mathbb{P}\left(d+\sum_{i=1}^{t}\eta'_i>0\hspace{0.15cm}\forall t\in [T]\right)\\
& \hspace{0.7cm}\leq\mathbb{P}\left(d+\sum_{i=1}^{t}\eta'_i>0\hspace{0.15cm}\forall t\in [T],\, |\mathcal{V}^{(d)}_{i}|\leq a_n(i)+m\hspace{0.15cm}\forall i\in \{0\}\cup [T-1]\right)\\
\nonumber&\hspace{2.7cm}+\mathbb{P}\left(\exists i\in \{0\}\cup[T-1]: |\mathcal{V}^{(d)}_{i}|> a_n(i)+m\right)
\end{align}
and we can focus on the probability in line (\ref{pp}).

Observe that, if $|\mathcal{V}^{(d)}_{i}|\leq a_n(i)+m$ for all $i\in  \{0\}\cup [T-1]$, then we can write
\begin{equation}\label{theetaprime}
\eta'_i=\ind_{R_i}(d-2)+\ind_{R_i}\ind_{F_i}\ind_{\left\{|\mathcal{V}^{(d)}_{i-1}|\leq a_n(i-1)+m \right\}} -1
\end{equation}
for all $i\in [T]$. Conditional on everything that has occurred up to the end of step $i-1$ in the exploration process, the random variable $\ind_{R_i}\ind_{F_i}\ind_{\left\{|\mathcal{V}^{(d)}_{i-1}|\leq a_n(i-1)+m \right\}}$ which appears in (\ref{theetaprime}) equals $1$ with probability 
\begin{align*}
p\ind_{\left\{|\mathcal{V}^{(d)}_{i-1}|\leq a_n(i-1)+m \right\}}\frac{d \left|\mathcal{V}^{(d)}_{i-1}\right|}{dn-2(i-1)-1}\leq p \frac{d (a_n(i-1)+m)}{dn-2(i-1)-1}.
\end{align*}
Thus, if $(U_i)_{i\geq 1}$ is an i.i.d. sequence of $U([0,1])$ random variables, also independent from all other random quantities involved, then
\begin{equation}\label{zi}
\mu_{i}\coloneqq \ind_{R_i}(d-2)+\ind_{R_i}\ind_{\left\{U_i\leq  \frac{d(a_n(i-1)+m)}{dn-2(i-1)-1}\right\}}-1
\end{equation}
defines a sequence of \textit{independent} random variables that, intuitively at least, should be larger than the $\eta'_i$. Thus, heuristically, the random process $(d+\sum_{i=1}^{t}\mu_i)_{t\in [T]}$ should be more likely to remain positive over the whole interval $[T]$ than the process $(d+\sum_{i=1}^{t}\eta'_i)_{t\in [T]}$. 

Our next result, whose proof is postponed to Section \ref{mainpropforupper_sec}, establishes this rigorously.
\begin{prop}\label{mainpropforupper}
	Let $(U_i)_i$ be a sequence of i.i.d random variables, also independent from all other random variables involved, with $U_1\sim U([0,1])$. For each $i\in [T]$ let $\mu_i$ be as in (\ref{zi}) above.
Then we have that
\begin{multline*}
\mathbb{P}\left(d+\sum_{i=1}^{t}\eta'_i>0\hspace{0.15cm}\forall t\in [T],\, |\mathcal{V}^{(d)}_i|<a_n(i)+m\hspace{0.15cm}\forall i\in \{0\}\cup [T-1]\right)\\
\leq \mathbb{P}\left(d+\sum_{i=1}^{t}\mu_i>0\hspace{0.15cm}\forall t\in [T]\right).
\end{multline*}
\end{prop}
Thanks to Lemma \ref{numberfresh} and Proposition \ref{mainpropforupper}, we can focus on the probability
\begin{equation}\label{newtobound}
\mathbb{P}\left(d+\sum_{i=1}^{t}\mu_i>0\hspace{0.15cm}\forall t\in [T]\right).
\end{equation}
In order to provide an upper bound for the above quantity we would like to turn the (independent but not identically distributed) $\mu_i$ into \textit{i.i.d.}~random variables $\xi_i$. To this end, keeping in mind the definition of $\mu_i$ given in \eqref{zi}, define
\begin{equation}\label{bi}
\mu'_i\coloneqq \ind_{R_i}\ind_{\left\{U_i> \frac{d(a_n(i-1)+m)}{dn-2(i-1)-1}\right\}}
\end{equation}
and set (for all $i\in [T]$)
\[\xi_i\coloneqq \mu_i+\mu'_i=\ind_{R_i}(d-2)+\ind_{R_i} - 1 = \ind_{R_i}(d-1) - 1.\]
(Recall that the $\ind_{R_i}$ are independent Bernoulli random variables with parameter $p$.) By adding the (random) sums $\sum_{i=1}^{t}\mu'_i$ to the $\sum_{i=1}^{t}\mu_i$ we can rewrite the probability in (\ref{newtobound}) as 
\begin{equation}\label{eqq}
\mathbb{P}\left(d+\sum_{i=1}^{t}\xi_i>\sum_{i=1}^{t}\mu'_i\hspace{0.15cm}\forall t\in [T]\right).
\end{equation}
Since each $\mu'_i$ is non-negative, we have that $\sum_{i=1}^{t}\mu'_i\geq 0$ for all $t\in [T]$. Therefore we can bound from above the probability in (\ref{eqq}) by
\begin{equation}\label{mnm}
\mathbb{P}\left(d+\sum_{i=1}^{t}\xi_i>0\hspace{0.15cm}\forall t\in [T], d+\sum_{i=1}^{T}\xi_i>\sum_{i=1}^{T}\mu'_i\right).
\end{equation}
In order to control the (random) sum $\sum_{i=1}^{T}\mu'_i$ (which depends on $m$, see its definition given in (\ref{bi})) we use the following result, whose proof is postponed to Section \ref{concentration_sec}.

\begin{lem}\label{arub}
	Let $h\geq 1$. Suppose that $A=o(n^{1/2})$ and let $m=m_n=O(n^{1/2})$. Define 
	\begin{equation}\label{qt}
	q(T)=q_{n,d}(T)\coloneqq p\left(1-\frac{2}{d}\right)\frac{T(T-1)}{2n}.
	\end{equation}
	Then, for all large enough $n$, we have that 
	\begin{equation}\label{mbm}
	\mathbb{P}\left(\sum_{i=1}^{T}\mu'_i\leq q(T)-h\right) \leq cTe^{-\frac{hn^{1/2}}{T}},
	\end{equation} 
	where $c=c(d)$ is a finite constant that depends on $d$.
\end{lem}

\begin{oss}\label{secondoss}
	Analogously to what we said in Remark \ref{ossfornumberfresh}, the quantity which appears within the exponential term in (\ref{mbm}) is of order $h/An^{1/6}$. Thus, if the reader thinks of $h$ as $h=h_n=An^{4/15}$, then the right-hand side in (\ref{mbm}) is much smaller than our desired bound and $h\ll q(T)$. 
\end{oss}

For $t\in [T]$, we bound from above the probability in (\ref{mnm}) by
\begin{multline}\label{readytoballot}
\eqref{mnm}\le\mathbb{P}\left(d+\sum_{i=1}^{t}\xi_i>0\hspace{0.15cm}\forall t\in [T-1],\, d+\sum_{i=1}^{T}\xi_i>q(T)-h\right)\\
+\mathbb{P}\left(\sum_{i=1}^{T}\mu'_i\leq q(T)-h\right).
\end{multline}
Thanks to Lemma \ref{arub} we know that we do not have to worry about the probability that $\sum_{i=1}^{T}\mu'_i$ is smaller than $q(T)-h$, provided $h$ is sufficiently large. Hence we can focus our attention on finding an upper bound for the first term on the right-hand side above.
Observe that, since $d+\sum_{i=1}^{T}\xi_i$ is at most $d+T(d-2)$ (as $\xi_i\leq d-2$ for every $i$), we can write
\begin{multline}\label{thesum}
\mathbb{P}\left(d+\sum_{i=1}^{t}\xi_i>0\hspace{0.15cm}\forall t\in [T],\, d+\sum_{i=1}^{T}\xi_i>q(T)-h\right)\\
\leq \sum_{k=\lfloor q(T)-h\rfloor+1}^{d+T(d-2)}\mathbb{P}\left(d+\sum_{i=1}^{t}\xi_i>0\hspace{0.15cm}\forall t\in [T],\, d+\sum_{i=1}^{T}\xi_i=k\right).
\end{multline}
To estimate the probabilities within the last sum we use Lemma \ref{hereweapplymodifiedballot} below, whose proof, postponed to Section \ref{hereweapplymodifiedballot_sec}, relies on a ballot-type estimate that was introduced in \cite{de_ambroggio_roberts:near_critical_ER}.

\begin{lem}\label{hereweapplymodifiedballot}
	For any $t,k\in\N$, if $d\ge 4$, then
	\[\mathbb{P}\left(d+\sum_{i=1}^{j}\xi_i>0\hspace{0.15cm}\forall j\in [t], d+\sum_{i=1}^{t}\xi_i=k\right)\leq \frac{k+d-4}{p^2(t+2)}\mathbb{P}\left(\sum_{i=1}^{t+2}\xi_i=k+d-4\right)\]
	and if $d=3$, then
	\[\mathbb{P}\left(d+\sum_{i=1}^{j}\xi_i>0\hspace{0.15cm}\forall j\in [t], d+\sum_{i=1}^{t}\xi_i=k\right)\leq \frac{k}{p^3(t+3)}\mathbb{P}\left(\sum_{i=1}^{t+3}\xi_i=k\right).\]
\end{lem}

Noting that the two inequalities in Lemma (\ref{hereweapplymodifiedballot}) are almost identical (the main difference being an extra factor of $1/p\asymp 1$ in the case $d=3$), we concentrate on the case $d\ge 4$. Applying Lemma \ref{hereweapplymodifiedballot}, we bound from above the sum in (\ref{thesum}) by
\[\eqref{thesum}\le \frac{1}{p^2(T+2)}\sum_{k=\lfloor q(T)-h\rfloor+1}^{d+T(d-2)}(k+d-4)\mathbb{P}\left(\sum_{i=1}^{T+2}\xi_i=k+d-4\right).\]
Note from the definition of $\xi_i$ that
\[\sum_{i=1}^{T+2}\xi_i = (d-1)B_{T+2,p} - (T+2)\]
where $B_{T+2,p}\sim\Bin(T+2,p)$. Thus
\[\eqref{thesum}\le \frac{1}{p^2(T+2)}\sum_{k=\lfloor q(T)-h\rfloor+1}^{d+T(d-2)}(k+d-4)\mathbb{P}\left(B_{T+2,p}=\frac{T+2 + k+d-4}{d-1}\right)\]
and, rewriting the last fraction on the right-hand side to isolate the expected value of $B_{T+2,p}$, we have that \eqref{thesum} is at most
\[\frac{1}{p^2(T+2)}\sum_{k=\lfloor q(T)-h\rfloor+1}^{d+T(d-2)}(k+d-4)\mathbb{P}\left(B_{T+2,p}=(T+2)p + \frac{k+d-4-\lambda(T+2)n^{-1/3}}{d-1}\right).\]
To summarise, if we let
\begin{equation}\label{xdndef}
x_{d,n}(k,\lambda,T) = \frac{k+d-4-\lambda(T+2)n^{-1/3}}{d-1},
\end{equation}
then we have shown that
\begin{multline}\label{summaryofprogress}
\mathbb{P}\left(d+\sum_{i=1}^{t}\xi_i>0\hspace{0.15cm}\forall t\in [T],\, d+\sum_{i=1}^{T}\xi_i>q(T)-h\right)\\
\le \frac{1}{p^2(T+2)}\sum_{k=\lfloor q(T)-h\rfloor+1}^{d+T(d-2)}(k+d-4)\mathbb{P}\big(B_{T+2,p}=(T+2)p + x_{d,n}(k,\lambda,T)\big).
\end{multline}

To analyse the right-hand side of \eqref{summaryofprogress}, we will need to split the sum into two parts, one for $k\le T^{2/3}$ and the other for $k>T^{2/3}$. Recall the definition of $G_\lambda(A,d)$ from Theorem \ref{mainthm}.

\begin{lem}\label{smallksum}
Provided that $h\le An^{4/15}$ and $A=o(n^{1/30})$, there exists a finite constant $c$ depending on $d$ such that
\[\frac{1}{p^2(T+2)}\sum_{k=\lfloor q(T)-h\rfloor+1}^{\lfloor T^{2/3}\rfloor}(k+d-4)\mathbb{P}\big(B_{T+2,p}=(T+2)p + x_{d,n}(k,\lambda,T)\big) \le \frac{c}{A^{1/2}n^{1/3}}e^{-G_\lambda(A,d)}.\]
\end{lem}

\begin{lem}\label{largeksum}
Provided that $|\lambda|=o(n^{1/30})$, there exists a constant $c>0$ depending on $d$ such that
\[\frac{1}{p^2(T+2)}\sum_{k=\lfloor T^{2/3}\rfloor+1}^{d+T(d-2)}(k+d-4)\mathbb{P}\big(B_{T+2,p}=(T+2)p + x_{d,n}(k,\lambda,T)\big) \le \exp(-c A^{1/3}n^{2/9}).\]
\end{lem}

We will prove both Lemmas \ref{smallksum} and \ref{largeksum} in Section \ref{ubappears_sec}. With these in hand, we are now in a position to establish the upper bounds stated in Theorem \ref{mainthm}.

\begin{proof}[Proof of the upper bounds in Theorem \ref{mainthm}]
Note that of the two quantities on the right-hand sides of Lemmas \ref{smallksum} and \ref{largeksum}, the fact that $A=o(1/30)$ ensures that the one from Lemma \ref{smallksum} is the larger. Thus, substituting the bounds from these lemmas into \eqref{summaryofprogress} gives that
\begin{equation}\label{finaleqforUB}
\mathbb{P}\left(d+\sum_{i=1}^{t}\xi_i>0\hspace{0.15cm}\forall t\in [T],\, d+\sum_{i=1}^{T}\xi_i>q(T)-h\right) \le \frac{c}{A^{1/2}n^{1/3}}e^{-G_\lambda(A,d)}.
\end{equation}
To complete the proof, we now recall the main points of the argument laid out so far. By \eqref{secondbound} we have
\begin{equation*}
\mathbb{P}(|\CC(V_n)|>\lfloor An^{2/3} \rfloor) \leq \mathbb{P}\left(d+\sum_{i=1}^{t}\eta'_i>0\hspace{0.15cm}\forall t\in [T]\right),
\end{equation*}
and by Lemma \ref{numberfresh} with $m=An^{4/15}$, plus Proposition \ref{mainpropforupper}, we obtain that
\[\mathbb{P}(|\CC(V_n)|>\lfloor An^{2/3} \rfloor)\le \mathbb{P}\left(d+\sum_{i=1}^{t}\mu_i>0\hspace{0.15cm}\forall t\in [T]\right) + cTe^{-c' n^{1/10}}.\]
Equations \eqref{eqq}, \eqref{mnm} and \eqref{readytoballot} then show that
\begin{multline*}
\mathbb{P}(|\CC(V_n)|>\lfloor An^{2/3} \rfloor)\le \mathbb{P}\left(d+\sum_{i=1}^{t}\xi_i>0\hspace{0.15cm}\forall t\in [T-1],\, d+\sum_{i=1}^{T}\xi_i>q(T)-h\right)\\
+\mathbb{P}\left(\sum_{i=1}^{T}\mu'_i\leq q(T)-h\right) + cTe^{-c' n^{1/10}}.
\end{multline*}
Applying Lemma \ref{arub} with $h=An^{4/15}$ gives that
\[\mathbb{P}\left(\sum_{i=1}^{T}\mu'_i\leq q(T)-h\right) \le cTe^{-c' n^{1/10}},\]
and finally from \eqref{finaleqforUB} together with the fact that $A=o(n^{1/30})$ we see that
\begin{equation}\label{uncondUB}
\mathbb{P}(|\CC(V_n)|>  An^{2/3} )\le \frac{c}{A^{1/2}n^{1/3}}e^{-G_\lambda(A,d)}.
\end{equation}
This is precisely the first upper bound in Theorem \ref{mainthm}, except that we have been working throughout via the exploration process described in Section \ref{exploration_sec}, which (as we mentioned in that section) generates a multigraph $\mathbb{G}'(n,d,p)$, whereas Theorem \ref{mainthm} concerns the (simple) graph $\mathbb{G}(n,d,p)$. Writing $\mathbb S_n$ for the event that the multigraph $\mathbb{G}'(n,d)$ underlying $\mathbb{G}'(n,d,p)$ (that is, the multigraph chosen uniformly at random from all $d$-regular multigraphs on $n$ vertices, before we carry out $p$-bond percolation) is simple, we have
\[\mathbb{P}\big(|\CC(V_n)|>  An^{2/3} \,\big|\, \mathbb S_n \big) = \frac{\mathbb{P}(|\CC(V_n)|>  An^{2/3},\, \mathbb S_n )}{\P(\mathbb S_n)} \le \frac{\mathbb{P}(|\CC(V_n)|>  An^{2/3} )}{\P(\mathbb S_n)}.\]
Since, as mentioned in Section \ref{config_model_sec}, $\P(\mathbb S_n)\to e^{(1-d^2)/4}$, the first upper bound in Theorem \ref{mainthm} follows.

For the second upper bound, concerning the size of the largest component, we proceed in a standard way (see e.g.~\cite{nachmias_peres:CRG_mgs}). 
Given any $k\in \mathbb{N}$, we denote by $N_k\coloneqq \sum_{i=1}^{n}\ind_{\{|\CC(i)|>k\}}$ the number of vertices located in components containing more than $k$ nodes. Then, by Markov's inequality, we obtain 
\begin{align*}
\mathbb{P}(|\mathcal{C}_{\text{max}}|>  An^{2/3}  )=\mathbb{P}(N_{\lfloor An^{2/3}\rfloor} >   An^{2/3} ) &\leq \frac{\mathbb{E}\big[N_{\lfloor An^{2/3}\rfloor}\big]}{  An^{2/3}}\\
&=\frac{n}{An^{2/3}}\mathbb{P}(|\CC(V_n)|> An^{2/3})\\
&\leq \frac{c}{A^{3/2}}e^{-G_\lambda(A,d)},
\end{align*}
completing the proof of the upper bounds in Theorem \ref{mainthm}.	
\end{proof}

The remainder of Section \ref{UBsec} is devoted to the proofs of those auxiliary results that have been used in our proof of the upper bounds in Theorem \ref{mainthm}. Specifically, we start by proving Proposition \ref{mainpropforupper} in Section \ref{mainpropforupper_sec}, and we proceed by establishing Lemma \ref{hereweapplymodifiedballot} in Section \ref{hereweapplymodifiedballot_sec}. Subsequently we prove, in Section \ref{ubappears_sec}, the two lemmas which give us the upper bound stated in Theorem \ref{mainthm}, namely Lemmas \ref{smallksum} and \ref{largeksum}. We finish with Section \ref{concentration_sec} where we prove the concentration bounds stated in Lemmas \ref{numberfresh} and \ref{arub}.

\subsection{Creating independent random variables: proof of Proposition \ref{mainpropforupper}}\label{mainpropforupper_sec}
Recall that  $\eta'_i=\ind_{R_i}(d-2)+\ind_{R_i}\ind_{F_i}$, where $R_i$ is the event that the edge $e_i h_i$ revealed during the $i$-th step of the exploration process is retained in the percolation, while $F_i$ is the event that the vertex $v(h_i)$ is fresh after the completion of step $i-1$. Also recall that $a_n(i) = n-1-i-i^2/2m$. As in Proposition \ref{mainpropforupper}, let $(U_i)_{i\geq 1}$ be a sequence of i.i.d.~random variables, also independent from all other random variables involved, with $U_1\sim U([0,1])$, and let
\[\mu_i = \ind_{R_i}(d-2) + \ind_{R_i}\ind_{\{U_i\le \frac{d(a_n(i-1)+m)}{dn - 2(i-1)-1}\}} - 1.\]

We now define, for $t\ge0$,
\[V_t \coloneqq \bigcap_{i=0}^t \{\calV^{(d)}_i \le a_n(i)+m\},\]
the event that the number of fresh vertices is at most $a_n(i)+m$ for each step $i\le t$. Let $\F_t$ be the $\sigma$-algebra generated by the exploration process up to step $t$. We begin with a simple observation, comparing $\eta'_j$ to $\mu_j$ on the event $V_{j-1}$.

\begin{lem}\label{etamuobs}
For any $j\ge 1$, the random variables $\eta'_j$ and $\mu_j$ take values in $\{-1,d-3,d-2\}$. They satisfy
\[\P(\eta'_j = -1\,|\,\F_{j-1}) = \P(\mu_j = -1\,|\,\F_{j-1}),\]
\[\P(\eta'_j = d-3\,|\,\F_{j-1})\ind_{V_{j-1}} \ge \P(\mu_j = d-3\,|\,\F_{j-1})\ind_{V_{j-1}},\]
and
\[\P(\eta'_j = d-2\,|\,\F_{j-1})\ind_{V_{j-1}} \le \P(\mu_j = d-2\,|\,\F_{j-1})\ind_{V_{j-1}}.\]
In other words, $\mu_j$ stochastically dominates $\eta'_j$ on the event $V_{j-1}$, given $\F_{j-1}$.
\end{lem}

\begin{proof}
The first display is trivial since $R_j$ (the event that the edge $e_jh_j$ is retained) is independent of $\F_{j-1}$; both sides equal $1-p$. Since the random variables can only take three possible values, it suffices to show one of the other two displays. Again since $R_j$ is independent of $\F_{j-1}$,
\[\P(\eta'_j = d-2\,|\,\F_{j-1}) = \P(R_j\cap F_j\,|\,\F_{j-1}) = p\P(F_j|\F_{j-1}).\]
Now, on the event $V_{j-1}$, since $h_j$ is chosen uniformly from the set of all unexplored stubs after step $j-1$, of which there are exactly $dn - 2(i-1)-1$, we have
\[\P(F_j|\F_{j-1})\ind_{V_{j-1}} \le \frac{d(a_n(i-1)+m)}{dn - 2(i-1)-1}\ind_{V_{j-1}} = \P\Big(U_j\le\frac{d(a_n(i-1)+m)}{dn - 2(i-1)-1}\, \Big|\,\F_{j-1}\Big)\ind_{V_{j-1}},\]
where the last equality uses the independence of $U_j$ from $\mathcal{F}_{j-1}$. Combining the last two displays gives the result.
\end{proof}

For $j,t\ge0$, we define
\[S_t^{(j)} = d + \sum_{i=1}^{j\wedge t} \eta'_i + \sum_{i=(j\wedge t)+1}^t \mu_i.\]
The idea is that $S_t^{(j)}$ interpolates between summing $\eta'_i$ and summing $\mu_i$; we can change $j$ by one increment at a time to move gradually from one sum to the other. The key ingredient in proving Proposition \ref{mainpropforupper} is the following lemma.

\begin{lem}\label{etamumain}
For any $j\in[T]$,
\[\P\big(S_t^{(j)} > 0 \,\,\forall t\in[T]\,\big|\,\F_{j-1}\big)\ind_{V_{j-1}} \le \P\big(S_t^{(j-1)} > 0 \,\,\forall t\in[T]\,\big|\,\F_{j-1}\big)\ind_{V_{j-1}}.\]
\end{lem}

\begin{proof}
Take $j\in[T]$. 
Since $\eta'_i$ is $\F_{j-1}$-measurable for every $i\le j-1$, and therefore $S^{(j)}_{j-1}$ is $\F_{j-1}$-measurable, we can split the probability of interest over the possible values of $S^{(j)}_{j-1}$ to give
\begin{multline*}
\P\big(S_t^{(j)} > 0 \,\,\forall t\in[T]\,\big|\,\F_{j-1}\big)\\
= \sum_{s = 1}^{\infty} \ind_{\{S^{(j)}_{j-1} = s\}}\ind_{\{S^{(j)}_i > 0 \,\,\forall i\in [j-1]\}} \P\bigg(s + \eta'_j + \sum_{i=j+1}^t \mu_i > 0 \,\, \forall t\in[T]\setminus[j-1]\,\bigg|\,\F_{j-1}\bigg).
\end{multline*}
(In fact, $S^{(j)}_{j-1}$ can take a maximum value of $d+(j-1)(d-2)$ so the sum above has only a finite number of positive summands.) Now observe that $\mu_i$ is independent of $\F_{j-1}$ for every $i\ge j+1$. Further splitting the probability on the right-hand side above over the possible values $Q=\{-1,d-2,d-3\}$ of $\eta'_j$, we therefore have
\begin{align}
&\P\big(S_t^{(j)} > 0 \,\,\forall t\in[T]\,\big|\,\F_{j-1}\big)\nonumber\\
&= \sum_{s = 1}^{\infty} \ind_{\{S^{(j)}_{j-1} = s\}}\ind_{\{S^{(j)}_i > 0 \,\,\forall i\in [j-1]\}}\nonumber\\
&\hspace{30mm}\cdot\sum_{q\in Q} \P\bigg(s+q+\sum_{i=j+1}^t \mu_i > 0 \,\,\forall t\in[T]\setminus[j-1]\bigg)\P(\eta'_j = q\,|\,\F_{j-1}).\label{nutomu1}
\end{align}
We further note that, by exactly the same argument, \eqref{nutomu1} holds also for $S_t^{(j-1)}$, provided that $\nu'_j$ is replaced by $\mu_j$. That is,
\begin{align}
&\P\big(S_t^{(j-1)} > 0 \,\,\forall t\in[T]\,\big|\,\F_{j-1}\big)\nonumber\\
&= \sum_{s = 1}^{\infty} \ind_{\{S^{(j-1)}_{j-1} = s\}}\ind_{\{S^{(j-1)}_i > 0 \,\,\forall i\in [j-1]\}}\nonumber\\
&\hspace{30mm}\cdot \sum_{q\in Q}\P\bigg(s+q+\sum_{i=j+1}^t \mu_i > 0 \,\,\forall t\in[T]\setminus[j-1]\bigg)\P(\mu_j = q\,|\,\F_{j-1}).\label{nutomu2}
\end{align}

We now apply Lemma \ref{etamuobs}, which tells us that $\mu_j$ stochastically dominates $\eta'_j$ on the event $V_{j-1}$, given $\F_{j-1}$. Since
\[\P\Big(s+q+\sum_{i=j+1}^t \mu_i > 0 \,\,\forall t\in[T]\setminus[j-1]\Big)\]
is increasing in $q$, we deduce that
\begin{multline*}
\sum_{q\in Q} \P\Big(s+q+\sum_{i=j+1}^t \mu_i > 0 \,\,\forall t\in[T]\setminus[j-1]\Big)\P(\eta'_j = q\,|\,\F_{j-1})\ind_{V_{j-1}}\\
\le \sum_{q\in Q} \P\Big(s+q+\sum_{i=j+1}^t \mu_i > 0 \,\,\forall t\in[T]\setminus[j-1]\Big)\P(\mu_j = q\,|\,\F_{j-1})\ind_{V_{j-1}}.
\end{multline*}
The result follows by combining this with \eqref{nutomu1} and \eqref{nutomu2}.
\end{proof}

The proof of Proposition \ref{mainpropforupper} is now a straightforward application of the above lemma together with the tower property.

\begin{proof}[Proof of Proposition \ref{mainpropforupper}]
For any $j\in[T]$, by Lemma \ref{etamumain},
\begin{align*}
\P\big(S_t^{(j)} > 0 \,\,\forall t\in[T],\,V_{j-1}\big) &= \E\big[\P\big(S_t^{(j)} > 0 \,\,\forall t\in[T]\,\big|\,\F_{j-1}\big)\ind_{V_{j-1}}\big]\\
&\le \E\big[\P\big(S_t^{(j-1)} > 0 \,\,\forall t\in[T]\,\big|\,\F_{j-1}\big)\ind_{V_{j-1}}\big]\\
&= \P\big(S_t^{(j-1)} > 0 \,\,\forall t\in[T],\,V_{j-1}\big).
\end{align*}
Since $V_t = \bigcap_{i=0}^t \{\calV^{(d)}_i \le a_n(i)+m\}$ is decreasing in $t$, and interpreting $V_{-1}$ to be the empty intersection (i.e.~the whole sample space), we have
\[\P\big(S_t^{(j)} > 0 \,\,\forall t\in[T],\,V_{j-1}\big)\le \P\big(S_t^{(j-1)} > 0 \,\,\forall t\in[T],\,V_{j-2}\big).\]
Repeating $T$ times, we have
\[\P\big(S_t^{(T)} > 0 \,\,\forall t\in[T],\,V_{T-1}\big) \le \P\big(S_t^{(0)} > 0 \,\,\forall t\in[T],\,V_{-1}\big) = \P\big(S_t^{(0)} > 0 \,\,\forall t\in[T]).\]
But for any $t\le T$, we have $S_t^{(T)} = d+\sum_{i=1}^t \nu'_i$ and $S_t^{(0)} = d + \sum_{i=1}^t \mu_i$, and therefore the line above is exactly the statement of the proposition.
\end{proof}


\subsection{The probability of staying positive and finishing above $q(T)$: proof of Lemma \ref{hereweapplymodifiedballot}}\label{hereweapplymodifiedballot_sec}
Recall that we are trying to bound from above the probability 
\begin{equation}\label{pk}
\mathbb{P}\left(d+\sum_{i=1}^{j}(\xi_i-1)>0\hspace{0.15cm}\forall j\in [t], d+\sum_{i=1}^{t}(\xi_i-1)=k\right).
\end{equation}
We wish to use the following result, which is Corollary 2.3 in \cite{de_ambroggio_roberts:near_critical_ER}.

\begin{lem}\label{ballotcor}
	Let $(X_i)_{i\ge 1}$ be i.i.d.~random variables taking values in $\mathbb{Z}$, whose distribution may depend on $t$. Let $h\in \mathbb{N}$, and suppose that $\mathbb{P}(X_1=h)>0$. Then for any $t,k\in\N$ we have
	\begin{equation*}
	\mathbb{P}\left(h+\sum_{i=1}^j X_i >0\hspace{0.15cm} \forall j\in [t],\, h+\sum_{i=1}^t X_i =k\right)\leq \frac{1}{\mathbb{P}(X_1=h)}\frac{k}{t+1}\mathbb{P}\left(\sum_{i=1}^{t+1} X_i=k\right).
	\end{equation*}
\end{lem}

Observe that we can't directly apply Lemma \ref{ballotcor} to bound the probabilities in (\ref{pk}), because $\mathbb{P}(\xi_i-1=d)=\mathbb{P}(\xi_i=d+1)=0$ (since $\xi_i\in \{0,d-1\}$). To solve this issue, we use the following simple tactic.



Let $\xi_0$ be a random variable independent from $(\xi_i)_{i\geq 1}$, such that $\xi_0\overset{d}{=}\xi_1$. Note that, by independence,
\begin{align*}
&\mathbb{P}\left(d+\sum_{i=1}^j (\xi_i-1)>0\hspace{0.15cm}\forall j\in [t], d+\sum_{i=1}^t (\xi_i-1) =k\right)\\
&=\frac{1}{\mathbb{P}(\xi_0-1=d-2)}\mathbb{P}\left(d+\sum_{i=1}^j (\xi_i-1)>0\hspace{0.15cm}\forall j\in [t], d+\sum_{i=1}^t (\xi_i-1) =k,\,\xi_0-1=d-2\right)\\
&\le \frac{1}{p}\mathbb{P}\left(2+\sum_{i=0}^j (\xi_i-1)>0\hspace{0.15cm}\forall j\in [t], 2+\sum_{i=0}^t(\xi_i-1)=k\right).
\end{align*}
Now, since $(\xi_i)_{i\ge0}$ has the same distribution as $(\xi_i)_{i\ge1}$, the above equals
\begin{equation}\label{worksfor3too}
\frac{1}{p}\mathbb{P}\left(2+\sum_{i=1}^j (\xi_i-1)>0\hspace{0.15cm}\forall j\in [t+1], 2+\sum_{i=1}^{t+1}(\xi_i-1)=k\right).
\end{equation}
Now suppose that $d\ge 4$. Then we have $2\le d-2$, and therefore the above is at most
\[\frac{1}{p}\mathbb{P}\left(d-2+\sum_{i=1}^j (\xi_i-1)>0\hspace{0.15cm}\forall j\in [t+1], d-2+\sum_{i=1}^{t+1}(\xi_i-1)=k+d-4\right).\]
We can now apply Lemma \ref{ballotcor} to conclude that this is bounded from above by
\[\frac{k+d-4}{p^2(t+2)}\mathbb{P}\left(\sum_{i=1}^{t+2}\xi_i = k+d-4\right),\]
which establishes Lemma \ref{hereweapplymodifiedballot} in the case $d\ge 4$.

When $d=3$ we need to apply the same trick once more: returning to \eqref{worksfor3too}, first using independence of $\xi_0$ and then the equality in distribution we have
\begin{align*}
\eqref{worksfor3too} &= \frac{1}{p^2}\mathbb{P}\left(2+\sum_{i=1}^j (\xi_i-1)>0\hspace{0.15cm}\forall j\in [t+1], 2+\sum_{i=1}^{t+1}(\xi_i-1)=k,\, \xi_0-1=1\right)\\
&\le \frac{1}{p^2}\mathbb{P}\left(1+\sum_{i=0}^j (\xi_i-1)>0\hspace{0.15cm}\forall j\in [t+1], 1+\sum_{i=0}^{t+1}(\xi_i-1)=k\right)\\
&= \frac{1}{p^2}\mathbb{P}\left(1+\sum_{i=1}^j (\xi_i-1)>0\hspace{0.15cm}\forall j\in [t+2], 1+\sum_{i=1}^{t+2}(\xi_i-1)=k\right).
\end{align*}
Applying Lemma \ref{ballotcor} now gives the result in the case $d=3$, completing the proof.

\subsection{The upper bound in Theorem \ref{mainthm} appears: proof of Lemmas \ref{smallksum} and \ref{largeksum}}\label{ubappears_sec}

To approximate the sum in Lemma \ref{smallksum}, we use the following result from \cite{bollobas_book}.

\begin{lem}[Theorem 1.2 of \cite{bollobas_book}]\label{bolBNPlem}
Let $\text{Bin}_{N,P}$ be a binomial random variable of parameters $N$ and $p$. Suppose that $PN\ge1$ and $1\le x(1-P)N/3$. Then if $j\ge PN+x$, we have
\[\P(\text{Bin}_{N,P}=j) < \frac{1}{\sqrt{2\pi P(1-P)N}} \exp\left(-\frac{x^2}{2P(1-P)N} + \frac{x}{(1-P)N} + \frac{x^3}{P^2 N^2}\right).\]
\end{lem}

We can now proceed with our proof of Lemma \ref{smallksum}. Recall that
\[x_{d,n}(k,\lambda,T) = \frac{k+d-4-\lambda(T+2)n^{-1/3}}{d-1}.\]

\begin{proof}[Proof of Lemma \ref{smallksum}]
Applying Lemma \ref{bolBNPlem} with $N=T+2$, $P=p$, $x=x_{d,n}(k,\lambda,T)$ and $j=(T+2)p + x_{d,n}(k,\lambda,T)$, we have
\begin{multline*}
\mathbb{P}\big(B_{T+2,p}=(T+2)p + x_{d,n}(k,\lambda,T)\big)\\
< \frac{c}{T^{1/2}} \exp\left(-\frac{x_{d,n}(k,\lambda,T)^2}{2p(1-p)(T+2)} + \frac{x_{d,n}(k,\lambda,T)}{(1-p)(T+2)} + \frac{x_{d,n}(k,\lambda,T)^3}{p^2(T+2)^2}\right).
\end{multline*}
When $k\le T^{2/3}$, the last two terms in the exponent are $O(1)$, and so
\[\mathbb{P}\big(B_{T+2,p}=(T+2)p + x_{d,n}(k,\lambda,T)\big) < \frac{c}{T^{1/2}} \exp\left(-\frac{x_{d,n}(k,\lambda,T)^2}{2p(1-p)(T+2)}\right).\]
We deduce that
\begin{multline}\label{sumtoint}
\frac{1}{p^2(T+2)}\sum_{k=\lfloor q(T)-h\rfloor+1}^{\lfloor T^{2/3}\rfloor}(k+d-4)\mathbb{P}\big(B_{T+2,p}=(T+2)p + x_{d,n}(k,\lambda,T)\big)\\
\le \frac{c}{T^{3/2}} \sum_{k=\lfloor q(T)-h\rfloor+1}^{\lfloor T^{2/3}\rfloor} (k+d-4)\exp\left(-\frac{x_{d,n}(k,\lambda,T)^2}{2p(1-p)(T+2)}\right),
\end{multline}
and the right-hand side above is easily seen to be at most
\[\frac{c}{T^{3/2}} \int_{\lfloor q(T)-h\rfloor}^\infty (y+d-4)\exp\left(-\frac{x_{d,n}(y,\lambda,T)^2}{2p(1-p)(T+2)}\right)\d y.\]
Recalling from \eqref{xdndef} that  $x_{d,n}(y,\lambda,T) = \frac{y+d-4-\lambda(T+2)n^{-1/3}}{d-1}$, in order to bound the integral by one that can be calculated exactly, it is useful to note that for $y\ge \lfloor q(T)-h\rfloor$,
\[y + d -4 \asymp y + d - 4 - \lambda(T+2)n^{-1/3};\]
this follows from the condition on $|\lambda|$ in Theorem \ref{mainthm}. Thus \eqref{sumtoint} is bounded above by
\begin{align}
&\frac{c}{T^{3/2}} \int_{\lfloor q(T)-h\rfloor}^\infty (y + d - 4 - \lambda(T+2)n^{-1/3})\exp\left(-\frac{(y + d - 4 - \lambda(T+2)n^{-1/3})^2}{2(d-1)^2 p(1-p)(T+2)}\right)\d y\nonumber\\
&= \frac{c}{T^{3/2}}(d-1)^2 p(1-p)(T+2)\exp\left(-\frac{(\lfloor q(T)-h\rfloor + d-4 - \lambda(T+2)n^{-1/3})^2}{2(d-1)^2 p(1-p)(T+2)}\right)\nonumber\\
&\le \frac{c}{A^{1/2}n^{1/3}}\exp\left(-\frac{(\lfloor q(T)-h\rfloor + d-4 - \lambda(T+2)n^{-1/3})^2}{2(d-1)^2 p(1-p)(T+2)}\right).\label{integraldone}
\end{align}

Concentrating on the exponent in \eqref{integraldone}, since $h\le An^{4/15}$ and $A=o(n^{1/30})$ and therefore $hq(T)\ll T$, we have
\begin{align*}
\frac{(\lfloor q(T)-h\rfloor + d-4 - \lambda(T+2)n^{-1/3})^2}{2(d-1)^2 p(1-p)(T+2)} &= \frac{(q(T)-\lambda T n^{-1/3})^2}{2(d-1)(1-\frac{1}{d-1})T} + o(1)\\
&= \frac{q(T)^2 - 2q(T)\lambda T n^{-1/3} + \lambda^2 T^2 n^{-2/3}}{2(d-2)T} + o(1).
\end{align*}
Recalling from Lemma \ref{arub} that $q(T)=p(1-2/d)\frac{T(T-1)}{2n}$, we see that the above is
\[\frac{(d-2)^2 T^4}{4(d-1)^2 d^2 n^2 \cdot 2(d-2)T} - \frac{(d-2)T^2 \cdot \lambda T n^{-1/3}}{(d-1)d n\cdot 2(d-2)T} + \frac{\lambda^2 T^2 n^{-2/3}}{2(d-2)T} + o(1)\]
and since $T=(d-1)An^{2/3} + O(n^{1/2})$, this equals
\[\frac{A^3(d-1)(d-2)}{8d^2} - \frac{A^2\lambda(d-1)}{2d} + \frac{A\lambda^2(d-1)}{2(d-2)} + o(1).\]
Combining the above calculations with \eqref{sumtoint} and \eqref{integraldone}, we have shown that
\begin{multline*}
\frac{1}{p^2(T+2)}\sum_{k=\lfloor q(T)-h\rfloor+1}^{\lfloor T^{2/3}\rfloor}(k+d-4)\mathbb{P}\big(B_{T+2,p}=(T+2)p + x_{d,n}(k,\lambda,T)\big)\\
\le \frac{c}{A^{1/2}n^{1/3}}\exp\left(-\frac{A^3(d-1)(d-2)}{8d^2} + \frac{A^2\lambda(d-1)}{2d} - \frac{A\lambda^2(d-1)}{2(d-2)}\right).
\end{multline*}
The exponent above is exactly $-G_\lambda(A,d)$ and the proof is complete.
\end{proof}

To prove Lemma \ref{largeksum}, we will need a Chernoff bound for the Binomial distribution. The following version comes from \cite{janson_et_al:random_graphs}.

\begin{lem}[{\cite[Theorem 2.1]{janson_et_al:random_graphs}}]\label{Bol2}
	Let $B_{N,P}$ be a binomial random variable of parameters $N$ and $p$. Then for every $x\geq 0$ we have
	\[\mathbb{P}(B_{N,P}\ge NP+x)\leq \exp\left(-\frac{x^2}{2(NP+x/3)}\right).\]
\end{lem}

We now apply this elementary bound to prove Lemma \ref{largeksum}.

\begin{proof}[Proof of Lemma \ref{largeksum}]
We first bound the factor of $k+d-4$ by a constant times $T$, so that
\begin{multline*}
\frac{1}{p^2(T+2)}\sum_{k=\lfloor T^{2/3}\rfloor+1}^{d+T(d-2)}(k+d-4)\mathbb{P}\big(B_{T+2,p}=(T+2)p + x_{d,n}(k,\lambda,T)\big)\\
\le c\P\big(B_{T+2,p}\ge (T+2)p + x_{d,n}(\lfloor T^{2/3}\rfloor+1,\lambda,T)\big).
\end{multline*}
Now note that since $|\lambda|=o(n^{1/30})$, for sufficiently large $n$ we have
\[x_{d,n}(\lfloor T^{2/3}\rfloor+1,\lambda,T) \ge \frac{pT^{2/3}}{2}.\]
Thus, by Lemma \ref{Bol2},
\begin{align*}
\P\big(B_{T+2,p}\ge (T+2)p + x_{d,n}(\lfloor T^{2/3}\rfloor+1,\lambda,T)\big) &\le \P\big(B_{T+2,p}\ge (T+2)p + pT^{2/3}/2\big)\\
&\le \exp\left(-\frac{p^2 T^{4/3}}{8((T+2)p+pT^{2/3}/6)}\right).
\end{align*}
Since $T$ is of order $An^{2/3}$, the above is at most $\exp(-c A^{1/3}n^{2/9})$ for some $c>0$ depending on $d$, and the result follows.
\end{proof}

\subsection{Concentration bounds: proofs of Lemmas \ref{numberfresh} and \ref{arub}}\label{concentration_sec}

\begin{proof}[Proof of Lemma \ref{numberfresh}]
	Recall that $a_n(i)=n-1-i-i^2/2n$. A union bound gives us
	\begin{equation*}
	\mathbb{P}\left(\exists \hspace{0.15cm}i\in \{0\}\cup[T-1]: |\mathcal{V}^{(d)}_i|> a_n(i)+ m\right)\leq \sum_{i=0}^{T-1}\mathbb{P}\left( |\mathcal{V}^{(d)}_i|\geq a_n(i)+ m\right).
	\end{equation*}
	Next we bound the probabilities on the right-hand side. 
	Notice that, since $|\mathcal{V}^{(d)}_0|=n-1$ and $a_n(0)=n-1$, we have that $\mathbb{P}\big( |\mathcal{V}^{(d)}_0|\geq a_n(0)+m\big)=0$ and hence we can assume throughout that $i\geq 1$. Now by definition of $a_n(i)$ and using the identity
	\begin{equation}\label{lowboundnumbfresh}
	|\mathcal{V}^{(d)}_i|=|\mathcal{V}^{(d)}_0|-\sum_{j=1}^{i}\ind_{F_j}=n-1-\sum_{j=1}^{i}(1-\ind_{F^c_j})=n-1-i+\sum_{j=1}^{i}\ind_{F^c_j},
	\end{equation}
	we can write
	\begin{equation}
	\nonumber\mathbb{P}\left( |\mathcal{V}^{(d)}_i|\geq a_n(i)+m\right)= \mathbb{P}\bigg( \sum_{j=1}^{i}\ind_{F^c_j}\geq i^2/2n+m\bigg).
	\end{equation}
	Now, given any $r>0$, by Markov's inequality we have
	\begin{align}\label{mgftobound}
	\mathbb{P}\bigg( \sum_{j=1}^{i}\ind_{F^c_j}\geq i^2/2n+m\bigg)\leq e^{-r(i^2/2n+m)}\mathbb{E}\left[e^{r\sum_{j=1}^{i}\ind_{F^c_j}}\right].
	\end{align}
	Next we bound from above the expectation in (\ref{mgftobound}). We write
	\begin{align}\label{towerexp}
	\mathbb{E}\left[e^{r\sum_{j=1}^{i}\ind_{F^c_j}}\right]=\mathbb{E}\left[e^{r\sum_{j=1}^{i-1}\ind_{F^c_j}}\mathbb{E}\left(\left.e^{r\ind_{F^c_i}}\right|\mathcal{F}_{i-1}\right)\right]
	\end{align}
	and observe that
	\begin{equation*}
	\mathbb{E}\left(\left.e^{r\ind_{F^c_i}}\right|\mathcal{F}_{i-1}\right)=1+\mathbb{P}(F^c_i|\mathcal{F}_{i-1})(e^r-1).
	\end{equation*}
	We claim that the number of unexplored stubs that are incident to vertices which are no longer fresh at the end of step $i-1$ is at most $di$. Indeed, at each step the number of non-fresh vertices can increase by at most one, and this vertex contributes at most $d$ stubs to the count. Since $h_i$ is picked uniformly at random from the set of unexplored stubs at the end of step $i-1$, of which there are exactly $dn-2(i-1)-1$, we can therefore bound
	\begin{align*}
	\mathbb{P}(F^c_i|\mathcal{F}_{i-1})\leq \frac{di}{dn-2(i-1)-1}.
	\end{align*}
	Thus
	\begin{equation*}
	\mathbb{E}\left(\left. e^{r\ind_{F^c_i}}\right|\mathcal{F}_{i-1}\right)\leq 1+\frac{di}{dn-2(i-1)-1}(e^r-1),
	\end{equation*}
	from which it follows that (see (\ref{towerexp}))
	\begin{align*}
	\mathbb{E}\left[e^{r\sum_{j=1}^{i}\ind_{F^c_j}}\right]\leq \left(1+\frac{di}{dn-2(i-1)-1}(e^r-1)\right)\mathbb{E}\left[e^{r\sum_{j=1}^{i-1}\ind_{F^c_j}}\right].
	\end{align*}
	Iterating the above argument and using the standard inequality $1+x\leq e^x$ (valid for all $x\in \mathbb{R}$) we obtain
	\begin{align*}
	\mathbb{E}\left[e^{r\sum_{j=1}^{i}\ind_{F^c_j}}\right]&\leq \prod_{j=1}^{i}\left(1+\frac{dj}{dn-2(j-1)-1}(e^r-1)\right)\\
	&\leq \prod_{j=1}^{i}\exp\left(\frac{dj}{dn-2(j-1)-1}(e^r-1)\right)\\
	&=\exp\bigg((e^r-1)d\sum_{j=1}^{i}\frac{j}{dn-2(j-1)-1}\bigg).
	\end{align*}
	A simple calculation shows that, for $i\leq T\ll n$,
	\begin{align*}
	\sum_{j=1}^{i}\frac{j}{dn-2(j-1)-1}=\sum_{j=1}^{i}\frac{j}{dn}\left(1-\frac{2(j-1)+1}{dn}\right)^{-1}\leq \frac{i^2}{2dn}\left(1+O(T/dn)\right)
	\end{align*}
	and hence we obtain
	\begin{equation*}
	\mathbb{E}\left[e^{r\sum_{j=1}^{i}\ind_{F^c_j}}\right]
	\leq \exp\left((e^r-1)\frac{i^2}{2n}\left(1+O(T/dn)\right)\right).
	\end{equation*}
	Combining this with \eqref{mgftobound}, we have shown that
	\begin{equation}\label{klk}
	\mathbb{P}\bigg( \sum_{j=1}^{i}\ind_{F^c_j}\geq i^2/2n+m\bigg)
	\leq e^{-r(i^2/2n+m)}\exp\left((e^r-1)\frac{i^2}{2n}\left(1+O(T/dn)\right)\right).
	\end{equation}
	Taking $r\leq (d-1)^{-1}$ we can write $e^r-1\leq r+r^2$ and hence for $i\leq T$ the expression in (\ref{klk}) is bounded from above by
	\begin{align*}\label{kpk}
	\exp\left(-rm+r^2 \frac{T^2}{2n} + rO\left(\frac{T^3}{n^2}\right)\right).
	\end{align*}
	Taking $r=n^{1/2}/T$ ($\asymp 1/An^{1/6}$) and using the fact that, as $n\to\infty$,
	\[\frac{n^{1/2}}{T}O\left(\frac{T^3}{n^2}\right)=O\left(\frac{T^2}{n^{3/2}}\right)=O\left(\left(\frac{A}{n^{1/12}}\right)^2\right)=o(1),\]
	we see that 
	\begin{equation*}
	\eqref{klk} \le  c\exp\left(-\frac{mn^{1/2}}{T}\right),
	\end{equation*}
	completing the proof.
\end{proof}

\begin{proof}[Proof of Lemma \ref{arub}]
Recall that, for each $i\in [T]$,
\begin{equation*}
\mu'_i=\ind_{R_i}\ind_{\left\{U_i> \frac{d(a_n(i-1)+m)}{dn-2(i-1)-1}\right\}}.
\end{equation*}
By Markov's inequality we see that, for every $r>0$,
\begin{align*}
\mathbb{P}\left(\sum_{i=1}^{T}\mu'_i\leq q(T)-h\right)\leq e^{-rh+rq(T)}\mathbb{E}\left[e^{-r\sum_{i=1}^{T}\mu'_i}\right].
\end{align*}
Since the $\ind_{R_i}$ are independent, the $U_i$ are independent and $(\ind_{R_i})_i$ is independent of $(U_i)_i$ we see that
\begin{align}\label{mgf}
\nonumber\mathbb{E}\left[e^{-r\sum_{i=1}^{T}\mu'_i}\right]=\prod_{i=1}^{T}\mathbb{E}\left(e^{-r\mu'_i}\right)&=\prod_{i=1}^{T}\left[1-p\left(1-\frac{d(a_n(i-1)+m)}{dn-2(i-1)-1}\right)(1-e^{-r})\right]\\
\nonumber&\leq \prod_{i=1}^{T}\exp\left\{-p(1-e^{-r})\left(1-\frac{d(a_n(i-1)+m)}{dn-2(i-1)-1}\right)\right\}\\
&=\exp\left\{-p(1-e^{-r})\sum_{i=1}^{T}\left(1-\frac{d(a_n(i-1)+m)}{dn-2(i-1)-1}\right)\right\}.
\end{align}
Recalling that $a_n(i-1)=n-1-(i-1)+(i-1)^2/2n$ we obtain
\begin{align*}
\frac{d(a_n(i-1)+m)}{dn-2(i-1)-1}=\left(1-\frac{i-1}{n}+\frac{(i-1)^2}{2n^2}+\frac{m-1}{n}\right)\frac{1}{1-\frac{2(i-1)+1}{dn}}.
\end{align*}
An elementary computation then shows that 
\begin{align*}
\frac{d(a_n(i-1)+m)}{dn-2(i-1)-1}\leq 1-\frac{i-1}{n}\left(1-\frac{2}{d}\right)+O\left(\frac{(i-1)^2}{2n^2}\right)+O\left(\frac{m}{n}\right).
\end{align*}
Therefore, recalling the definition of $q(T)=p(1-2/d)(2n)^{-1}T(T-1)$ and using the bound $1-e^{-x} \leq x$, we see that for $r>0$ the expression in (\ref{mgf}) is at most
\begin{align*}
&\exp\left\{-p(1-e^{-r})\left[(1-2/d)\frac{T(T-1)}{2n}-O\left(\frac{T^3}{n^2}\right)-O\left(\frac{Tm}{n}\right)\right]\right\}\\
\leq &\exp\left\{-(1-e^{-r})q(T)+rO\left(\frac{T^3}{n^2}\right)+rO\left(\frac{Tm}{n}\right)\right\}.
\end{align*}
Hence, using the bound $1-e^{-x}\geq 1-(1-x+x^2/2)=x-x^2/2$ (valid for all $x\geq 0$) together with the fact that $q(T)\leq T^2/2n$, we obtain
\begin{align}\label{nn}
\nonumber\mathbb{P}\left(\sum_{i=1}^{T}\mu'_i\leq q(T)-h\right)&\leq \exp\left\{-rh+r^2\frac{T^2}{4n}+rO\left(\frac{T^3}{n^2}\right)+rO\left(\frac{Tm}{n}\right)\right\}.
\end{align}
Taking $r=n^{1/2}/T$ we see that
\begin{align*}
rO\left(\frac{T^3}{n^2}\right)+rO\left(\frac{Tm}{n}\right)=O\left(\left(\frac{A}{n^{1/12}}\right)^2\right)+O\left(\frac{m}{n^{1/2}}\right)=o(1)
\end{align*}
and hence
\begin{align*}
\exp\left\{-rh+r^2\frac{T^2}{4n}+rO\left(\frac{T^3}{n^2}\right)+rO\left(\frac{Tm}{n}\right)\right\}\leq ce^{-\frac{hn^{1/2}}{T}}
\end{align*}
for some positive constant $c$, completing the proof of the lemma.
\end{proof}

\section{Proof of Theorem \ref{mainthm}: lower bounds}\label{LB_sec}
We begin by recalling that our exploration process potentially creates multiple edges or self-loops, and that to produce the simple graph $\mathbb{G}(n,d,p)$ we condition on the event $\mathbb{S}_n$ that the $d$-regular multigraph $\mathbb{G}'(n,d)$ produced by the exploration process (including both retained and unretained edges) is simple. For the upper bounds in Theorem \ref{mainthm}, we worked for the most part with the multigraph and then deduced the result conditional on $\mathbb{S}_n$ at the very last step; this worked because for any event $\mathcal B$, we have
\[\P(\mathcal B\, |\,\mathbb{S}_n) = \P(\mathcal B\cap\mathbb{S}_n)/\P(\mathbb{S}_n) \le \P(\mathcal B)/\P(\mathbb{S}_n) \le c\P(\mathcal B).\]
For the lower bound this does not work, and we must be aware of the conditioning on $\mathbb{S}_n$ throughout. It turns out that our proof does not depend much on whether we condition on $\mathbb{S}_n$ or not, and a version of Theorem \ref{mainthm} for the multigraph $\mathbb{G}'(n,d,p)$ could be given by following our proof and ignoring any appearance of $\mathbb{S}_n$.

By Lemma \ref{relationship}, we are tasked with bounding from below the probability
\begin{align}\label{probtoboundbelow}
\mathbb{P}(|\CC(V_n)|>An^{2/3}\, |\,\mathbb{S}_n) &= \P(\sigmaUR > An^{2/3}-1 \,|\,\bS_n )\nonumber\\
&\ge \mathbb{P}(\tau>(d-1)An^{2/3}+1\,|\,\bS_n)\nonumber\\
&=\mathbb{P}\bigg(d+\sum_{i=1}^{t}\eta_i>0\hspace{0.15cm}\forall t\in \big[\lfloor(d-1)An^{2/3}\rfloor +1\big]\,\bigg|\,\bS_n\bigg),
\end{align}
where we recall that, if $i\le \tau$, then
\begin{equation}
\eta_i=\ind_{\{h_i\in \mathcal{U}_{i-1}\}} \ind_{R_i}\left|\mathcal{S}(h_i)\cap \mathcal{U}_{i-1}\setminus \{h_i\}\right|-\ind_{\{h_i\in \mathcal{A}_{i-1}\}}-1.
\end{equation}
To simplify the notation, we set $T\coloneqq \lfloor(d-1) An^{2/3}\rfloor + 1$, noting that this new definition of $T$ is not quite the same as the one of Section 3.

Recall that $F_i$ is the event that vertex $v(h_i)$ is fresh, i.e.~that $v(h_i)$ has $d$ unseen stubs, at the end of step $i-1$; we also defined $F'_i$ to be the event that $v(h_i)$ has $d-1$ unseen stubs at the end of step $i-1$, and $F^-_i$ to be the event that $v(h_i)$ has $m\in [d-2]$ unseen stubs at the end of step $i-1$. In other words, recalling also that $\mathcal V_{i-1}^{m}$ is the set of vertices with $m$ unseen stubs at the end of step $i-1$, we have
\[F_i = \{v(h_i)\in \calV_{i-1}^{(d)}\}, \,\,\,\, F_i' = \{v(h_i)\in\calV_{i-1}^{(d-1)}\} \,\,\text{ and }\,\, F^-_i = \bigg\{v(h_i)\in\bigcup_{m=1}^{d-2}\calV_{i-1}^{(m)}\bigg\}.\]

Since we want to bound the probability in (\ref{probtoboundbelow}) from below, we need to approximate the $\eta_i$ with \textit{smaller} random variables, sufficiently close to the $\eta_i$ but easier to deal with. To this end, define
\begin{equation}\label{deltadef}
\delta_i \coloneqq \ind_{R_i}\ind_{F_i}(d-1) + \ind_{R_i}\ind_{F'_i}(d-2) - \ind_{\{h_i\in \calA_{i-1}\}} - 1
\end{equation}
and
\begin{equation}\label{delta'def}
\delta'_i \coloneqq \ind_{R_i}\ind_{F_i}(d-1) - \ind_{\{h_i\in \calA_{i-1}\}} - 1
\end{equation}
and note that for each $i\le \tau$ we have
\begin{equation}\label{nudelta}
\eta_i \ge \delta_i \ge \delta'_i.
\end{equation}
Indeed, we note that $\eta_i=\delta_i$ unless $R_i\cap F^-_i$ occurs, in which case $\eta_i$ equals the number of unseen stubs in $\mathcal{S}(h_i)\setminus \{h_i\}$ minus one, whereas $\delta_i=-1$. Furthermore, $\delta_i = \delta'_i$ unless $R_i\cap F'_i$ occurs, in which case $\delta_i=d-3$ and $\delta'_i=-1$.

The less precise approximation given by $\delta'_i$ will be useful when $i$ is small, when almost all vertices will have $d$ unseen stubs. However, approximating $\eta_i$ with $\delta'_i$ over the whole time interval $[T]$ turns out to be insufficient to obtain lower bounds that match the upper bounds established in Section 3, and the closer approximation given by $\delta_i$ will be needed when $t$ is larger and a substantial number of vertices have $d-1$ unseen stubs.

Once we have replaced $\eta_i$ with $\delta'_i$ or $\delta_i$ as appropriate, our next step is to replace the process formed by summing the $\delta_i$ with a random walk with increments
\[D_i \coloneqq \ind_{R_i}(d-1)-1.\]
The random variables $\delta_i$ tend to get smaller as $i$ increases, since fewer vertices have $d$ or $d-1$ unseen stubs, whereas the $D_i$ are (independent and) identically distributed. This means that we must substitute the event that the process $\sum_{i=1}^{t}\delta_i$ stays above $-d$ with the event that the process $\sum_{i=1}^t D_i$ stays above the increasing curve
\begin{equation}\label{defqt}
q(t)=q_{A,n,d}(t)\coloneqq p(1-2/d)\frac{t^2}{2n} + An^{4/15}.
\end{equation}
The next result shows that this is the right curve to use. Define
\[\tau_\delta \coloneqq \min\bigg\{t\ge 1 : d+\sum_{i=1}^t \delta_i \le 0\bigg\}\]
and note that $\tau_{\delta}\leq \tau$, where we recall that $\tau$ is the first time at which the set of active stubs becomes empty. To see this observe that, since $\delta_i\leq \eta_i$ for all $i\leq \tau$, then $d+\sum_{i=1}^{\tau}\delta_i\leq d+\sum_{i=1}^{\tau}\eta_i\leq 0$.

\begin{prop}\label{animpprop}
	Suppose that $A\ll n^{1/30}$. Then for all large enough $n$, we have that
	\[\P\bigg(\exists t\in[T\wedge\tau_\delta] : \sum_{i=1}^t (D_i - \delta_i) > q(t)\,\bigg|\,\bS_n\bigg) \leq CTe^{-cn^{1/10}}.\]
	where $C$ is a finite constant and $c=c(d)>0$ is a constant that depends on $d$.
%
\end{prop}

This result will be proved in Section \ref{animppropsec}. We now apply this result to show that the right-hand side of \eqref{probtoboundbelow} 
can be roughly split into the product of two terms, each of which is easier to analyse.

\begin{lem}\label{splitprod_lem}
For any $0<T'<T$ and $\eps>0$ we have
\begin{align*}
&\P\bigg(d+\sum_{i=1}^t \eta_i > 0\, \forall t\in[T]\,\bigg|\,\bS_n\bigg)\\
&\ge \mathbb{P}\bigg(\sum_{i=1}^{t}\delta'_i > -d\hspace{0.15cm}\forall t\in [T'],\, \sum_{i=1}^{T'}\delta'_i \geq \eps\sqrt{T'}\,\bigg|\,\bS_n\bigg)\\
&\hspace{30mm}\cdot\mathbb{P}\bigg(\sum_{i=T'+1}^t D_i >q(t)-\eps\sqrt{T'}\hspace{0.15cm}\forall t\in [T]\setminus[T']\bigg) - CTe^{-cn^{1/10}}
\end{align*}
where $C$ is a finite constant and $c=c(d)>0$ is a constant that depends on $d$.
\end{lem}

\begin{proof}
Define $\tau_{\delta'}\coloneqq \min \left\{t\geq 1:d+\sum_{i=1}^{t}\delta'_i\leq 0 \right\}$ and note that, since $\delta'_i\leq \delta$ for all $i$, we must have $\tau_{\delta'}\leq \tau_{\delta}$. Since we also have $\tau_{\delta}\leq \tau$ we obtain that
	\begin{align}\label{eventincl}
	\nonumber\{\tau\in [T] \}\subset \{\tau_{\delta}\in [T]\}&\subset \{\tau_{\delta'}\in [T']\}\cup \{\tau_{\delta'}\notin [T'], \tau_{\delta}\in [T]\}\\
	&= \{\tau_{\delta'}\in [T']\}\cup \{\tau_{\delta'}\notin [T'], \tau_{\delta}\in [T]\setminus [T'] \}.
	\end{align}
	Now observe that if $\sum_{i=1}^{t}\delta_i \le -d$, then either $\sum_{i=1}^{t}D_i\leq q(t)$, or $\sum_{i=1}^{t}D_i> q(t)$ and $\sum_{i=1}^{t}(D_i-\delta_i)>q(t)+d > q(t)$. Therefore we have the inclusion
	\begin{multline*}
	\big\{\tau_{\delta}\in [T]\setminus[T']\big\}\subset \bigg\{\exists t\in [T]\setminus[T']:\sum_{i=1}^{t}D_i\leq q(t) \bigg\}\\
	\cup \bigg\{\exists t\in [T\wedge \tau_{\delta}]\setminus [T']:\sum_{i=1}^{t}(D_i-\delta_i)>q(t)\bigg\}.
	\end{multline*}
	Substituting this inclusion into (\ref{eventincl}) we obtain
	\begin{multline*}
	\big\{\tau\in [T] \big\}\subset \big\{\tau_{\delta'}\in [T']\big\}\cup \bigg\{\tau_{\delta'}\notin [T'],\, \exists t\in [T]\setminus[T']:\sum_{i=1}^{t}D_i\leq q(t)\bigg\}\\
	\cup\bigg\{\exists t\in [T\wedge \tau_{\delta}]\setminus [T']:\sum_{i=1}^{t}(D_i-\delta_i)>q(t)\bigg\}.
	\end{multline*}
	Consequently we deduce that
	\begin{align*}
	\mathbb{P}\big(\tau\not\in[T]\big|\mathbb{S}_n\big) &\geq \mathbb{P}\big(\tau_{\delta'}\notin [T']\big|\mathbb{S}_n\big)-\mathbb{P}\bigg(\tau_{\delta'}\notin [T'], \exists t\in [T]\setminus[T']:\sum_{i=1}^{t}D_i\leq q(t)\,\bigg|\,\mathbb{S}_n\bigg)\\
	&\hspace{30mm}-\mathbb{P}\bigg(\exists t\in [T\wedge \tau_{\delta}]\setminus [T']:\sum_{i=1}^{t}(D_i-\delta_i)>q(t)+d \,\bigg|\,\mathbb{S}_n\bigg)\\
	&= \mathbb{P}\bigg(\tau_{\delta'}\notin [T'],  \sum_{i=1}^{t}D_i>q(t)\hspace{0.15cm}\forall t\in [T]\setminus[T']\,\bigg|\,\mathbb{S}_n\bigg)\\
	&\hspace{30mm}-\mathbb{P}\bigg(\exists t\in [T\wedge \tau_{\delta}]\setminus [T']:\sum_{i=1}^{t}(D_i-\delta_i)>q(t)+d \,\bigg|\,\mathbb{S}_n\bigg).
	\end{align*}
	Recalling that $\{\tau\not\in[T]\}$ is equivalent to $\{d+\sum_{i=1}^t \eta_i > 0\,\, \forall t\in[T]\}$, and similarly $\{\tau_{\delta'}\notin [T']\}$ is equivalent to $\{\sum_{i=1}^t \delta'_i > -d\,\,\forall t\in[T']\}$, and applying Proposition \ref{animpprop} we obtain that
\begin{multline*}
\P\bigg(d+\sum_{i=1}^t \eta_i > 0\,\, \forall t\in[T]\,\bigg|\,\bS_n\bigg)\\
\ge \P\bigg(\sum_{i=1}^t \delta'_i > -d\,\,\forall t\in[T'],\,\sum_{i=1}^t D_i > q(t)\,\,\forall t\in[T]\setminus[T']\,\bigg|\,\bS_n\bigg) - CTe^{-cn^{1/10}}.
\end{multline*}
Since $\delta'_i\le D_i$ for each $i$, and therefore $\sum_{i=1}^{T'}\delta'_i\leq \sum_{i=1}^{T'}D_i$, we also have
%
\begin{align*}
&\mathbb{P}\bigg(\sum_{i=1}^{t}\delta'_i >-d\hspace{0.15cm}\forall t\in [T'], \sum_{i=1}^t D_i>q(t)\hspace{0.15cm}\forall t\in [T]\setminus[T']\,\bigg|\,\bS_n\bigg)\\
&\geq \mathbb{P}\bigg(\sum_{i=1}^{t}\delta'_i > -d\hspace{0.15cm}\forall t\in [T'],\,\, \sum_{i=1}^{T'}\delta'_i \geq \eps\sqrt{T'},\,\, \sum_{i=1}^t D_i > q(t)\hspace{0.15cm}\forall t\in [T]\setminus[T']\,\bigg|\,\bS_n\bigg)\\
&\geq \mathbb{P}\bigg(\sum_{i=1}^{t}\delta'_i > -d\hspace{0.15cm}\forall t\in [T'],\,\, \sum_{i=1}^{T'}\delta'_i \geq \eps\sqrt{T'},\,\, \sum_{i=T'+1}^t D_i >q(t)-\eps\sqrt{T'}\hspace{0.15cm}\forall t\in [T]\setminus[T']\,\bigg|\,\bS_n\bigg).
\end{align*}
Since $D_j$ is independent of $\sum_{i=1}^t\delta'_i$ whenever $j>t$, and the sequence $(D_j)$ is also independent of $\bS_n$, the last probability equals
\[\mathbb{P}\bigg(\sum_{i=1}^{t}\delta'_i > -d\hspace{0.15cm}\forall t\in [T'],\,\, \sum_{i=1}^{T'}\delta'_i \geq \eps\sqrt{T'}\,\bigg|\,\bS_n\bigg)\mathbb{P}\bigg(\sum_{i=T'+1}^t D_i >q(t)-\eps\sqrt{T'}\hspace{0.15cm}\forall t\in [T]\setminus[T']\bigg),\]
as desired.
\end{proof}

To bound from below the first probability on the right-hand side of Lemma \ref{splitprod_lem}, the idea is to substitute the process $(\sum_{i=1}^{t}\delta'_i)_{t\in [T']}$ with a random walk having (i.i.d.) mean zero increments, and then to use known results about random walks to bound from below the probability that such a random walk stays positive up to time $T'$ and finishes above level $\eps\sqrt{T'}$ at time $T'$. 

On the other hand, to bound from below the second probability on the right-hand side of Lemma \ref{splitprod_lem}, the idea is to  approximate the random walk $(\sum_{i=T'}^t D_i)_{t\in [T]\setminus [T']}$ with (standard) Brownian motion, and then to bound from below the probability that Brownian motion stays above the curve $q(t)$ by the probability that it stays above two straight lines which lie above the curve $q(t)$.

The details are carried out in the following two propositions, whose proofs can be found in Subsections \ref{propfirstbit_sec} and \ref{propsecondbit_sec}, respectively.
\begin{prop}\label{propfirstbit}
	Let $T'\coloneqq \lfloor n^{2/3}/A^2\rfloor$. Then there exists $\eps>0$ such that for all large $n$,
	\begin{equation*}
	\mathbb{P}\bigg(\sum_{i=1}^{t}\delta'_i > -d\hspace{0.15cm}\forall t\in [T'],\, \sum_{i=1}^{T'}\delta'_i \geq \eps\sqrt{T'}\,\bigg|\,\bS_n\bigg)\geq c\frac{A}{n^{1/3}},
	\end{equation*}
	where $c=c(d)>0$ is a finite constant that depends on $d$.
\end{prop}
\begin{prop}\label{propsecondbit}
	Let $T'\coloneqq \lfloor n^{2/3}/A^2\rfloor$ and $\eps>0$. Then, for all large enough $n$, we have that
	\begin{multline*}
	\mathbb{P}\bigg(\sum_{i=T'+1}^t D_i >q(t)-\eps\sqrt{T'}\hspace{0.15cm}\forall t\in [T]\setminus[T']\bigg)\geq \frac{c}{A^{3/2}}e^{-\frac{A^3(d-1)(d-2)}{8d^2}+\frac{\lambda A^2(d-2)^2}{2d(d-1)}-\frac{\lambda^2 A(d-1)}{2(d-2)}},
	\end{multline*}
	where $c=c(d,\eps)>0$ is a finite constant that depends on $d$ and $\eps$.
\end{prop}

We are now in a position to prove the lower bounds stated in Theorem \ref{mainthm}, subject to completing the proofs of Propositions \ref{animpprop}, \ref{propfirstbit} and \ref{propsecondbit} above.
\begin{proof}[Proof of the lower bounds in Theorem \ref{mainthm}]

It follows from \eqref{probtoboundbelow} and Lemma \ref{splitprod_lem} that for any $0<T'<T$,
\begin{multline*}
\mathbb{P}(|\CC(V_n)|>An^{2/3}\,|\,\bS_n) \ge \mathbb{P}\bigg(\sum_{i=1}^{t}\delta'_i > -d\hspace{0.15cm}\forall t\in [T'],\, \sum_{i=1}^{T'}\delta'_i \geq \eps\sqrt{T'}\,\bigg|\,\bS_n\bigg)\\
\cdot\mathbb{P}\bigg(\sum_{i=T'+1}^t D_i >q(t)-\eps\sqrt{T'}\hspace{0.15cm}\forall t\in [T]\setminus[T']\bigg)  - CTe^{-cn^{1/10}}. 
\end{multline*}
Propositions \ref{propfirstbit} and \ref{propsecondbit} then tell us that when $T' = \lfloor n^{2/3}/A^2\rfloor$ this is at least
\[\frac{c'}{A^{1/2}n^{1/3}}\exp\Big(-\frac{A^3(d-1)(d-2)}{8d^2}+\frac{\lambda A^2(d-2)^2}{2d(d-1)}-\frac{\lambda^2 A(d-1)}{2(d-2)}\Big) - CTe^{-cn^{1/10}}.\]
Since $A\ll n^{1/30}$ and $|\lambda|=O(A)$, the first term above is dominant when $n$ is large, and the required bound on $\mathbb{P}(|\CC(V_n)|>An^{2/3})$ follows.

To obtain the lower bound for the probability that $|\mathcal{C}_{\text{max}}|$ is larger than $An^{2/3}$, we follow exactly the same argument elaborated in \cite{de_ambroggio_roberts:near_critical_ER}, which we recall here for the reader's convenience. First of all we remark that, for any (non-trivial) $\N_0$-valued random variable $X$,
\begin{equation}\label{simpleineq}
\mathbb{P}(X\geq 1 \,|\,\bS_n)\geq \frac{\mathbb{E}[X|\bS_n]^2}{\mathbb{E}[X^2|\bS_n]};
\end{equation}
this simple fact can be proved by applying the Cauchy-Schwarz inequality to $X\ind_{\{X\ge 1\}}$.

Denote by $X = \sum_{i=1}^{n}\ind_{\{|\CC(i)|\in [T,2T]\}}$ the number of vertices contained in components of size between $T$ and $2T$. Observe that $X\geq 1 $ implies $|\mathcal{C}_{\max}|\geq T$. Therefore, using \eqref{simpleineq} we obtain
\begin{equation}\label{ratiox}
\mathbb{P}\left(|\mathcal{C}_{\max}|\geq T \,|\,\bS_n\right)\geq \P(X\ge 1\,|\,\bS_n) \ge \frac{\mathbb{E}[X|\bS_n]^2}{\mathbb{E}[X^2|\bS_n]}.
\end{equation}
For the numerator, since $V_n$ is a vertex selected uniformly at random from the set of vertices we have
\begin{equation}\label{EXtoP}
\mathbb{E}[X|\bS_n]^2=n^2\mathbb{P}\left(\left.|\CC(V_n)|\in [T,2T]\,\right|\,\bS_n\right)^2.
\end{equation}
Next we bound the denominator from above, ignoring the conditioning on $\bS_n$ for now. Given vertices $i,j\in [n]$, recall that we write $i\leftrightarrow j$ if there exists a path of open edges (that is, edges in $\G(n,d,p)$) between $i$ and $j$. Denote by $V'_n$ a vertex sampled uniformly at random from $[n]$, independently of $V_n$. Then we can write
\begin{align}
\mathbb{E}[X^2] &= n^2\mathbb{P}\left(|\CC(V_n)|\in [T,2T],\, |\CC(V'_n)|\in[T,2T]\right)\nonumber\\
&= n^2\P\left(|\CC(V_n)|\in[T,2T],\, V'_n\in \CC(V_n)\right)\nonumber\\
&\hspace{30mm} + n^2 \P\left(|\CC(V_n)|\in[T,2T],\, |\CC(V'_n)|\in[T,2T],\,V_n\nleftrightarrow V'_n\right).\label{secondmoment}
\end{align}
Since $V'_n$ is uniformly chosen independently of $V_n$, we have
\begin{align}
n^2\P\left(|\CC(V_n)|\in[T,2T],\, V'_n\in \CC(V_n)\right) &\le n^2 \frac{2T}{n}\P\left(|\CC(V_n)|\in[T,2T]\right)\nonumber\\
&= 2Tn\P\left(|\CC(V_n)|\in[T,2T]\right).\label{prodsplit1}
\end{align}

For the second term on the right-hand side of \eqref{secondmoment}, we observe that once we have run the exploration process until step $\tau$ and explored $\CC(V_n)$, if we then observe that $V'_n$ is not in $\CC(V_n)$ then we may choose (in part (b) of the exploration process) one of the stubs incident to $V'_n$ to begin the next phase. We may then repeat the argument in Section \ref{UBsec} for the exploration of this second component, to discover that the probability that it is larger than $T$ is again at most
\[\frac{c}{A^{1/2}n^{1/3}}e^{-G_{\lambda}(A,d)}.\]
Thus we have
\begin{align*}
&\P\left(|\CC(V_n)|\in[T,2T],\, |\CC(V'_n)|\in[T,2T],\,V_n\nleftrightarrow V'_n\right)\\
&\hspace{40mm}\le \E\left[\ind_{\{|\CC(V_n)|\in[T,2T], V_n\nleftrightarrow V'_n\}}\P\left(\left.|\CC(V'_n)\ge T \,\right|\,\F_{\tau}, V'_n\right)\right]\\
&\hspace{40mm}\le \frac{c}{A^{1/2}n^{1/3}}e^{-G_{\lambda}(A,d)}\P\left(|\CC(V_n)|\in[T,2T]\right).
\end{align*}
Substituting this and \eqref{prodsplit1} into \eqref{secondmoment}, we obtain
\begin{align*}
\mathbb{E}[X^2] &\le 2Tn\P\left(|\CC(V_n)|\in[T,2T]\right) + \frac{cn^{5/3}}{A^{1/2}}e^{-G_{\lambda}(A,d)}\P\left(|\CC(V_n)|\in[T,2T]\right)\\
&\le 3Tn\P\left(|\CC(V_n)|\in[T,2T]\right)
\end{align*}
for large $n$, and then by our usual argument,
\[\mathbb{E}[X^2|\bS_n] = \frac{\E[X^2\ind_{\bS_n}]}{\P(\bS_n)} \le \frac{\E[X^2]}{\P(\bS_n)} \le c\E[X^2] \le 3cTn\P\left(|\CC(V_n)|\in[T,2T]\right)\]
for some finite constant $c$. In turn substituting this and \eqref{EXtoP} into \eqref{ratiox}, we have for large $n$
\[\mathbb{P}\left(\left.|\mathcal{C}_{\max}|\geq T\,\right|\,\bS_n\right)\geq \frac{n^2\mathbb{P}\left(\left.|\CC(V_n)|\in [T,2T]\,\right|\,\bS_n\right)^2}{3cTn\P\left(|\CC(V_n)|\in[T,2T]\right)} \ge \frac{c'n^{1/3}}{A}\frac{\mathbb{P}\left(\left.|\CC(V_n)|\in [T,2T]\,\right|\,\bS_n\right)^2}{\P\left(|\CC(V_n)|\in[T,2T]\right)}\]
for some constant $c'>0$.

Now applying the lower bound
\[\P\left(\left.|\CC(V_n)|\ge T \,\right|\,\bS_n\right) \ge \frac{c_1}{A^{1/2}n^{1/3}}e^{-G_{\lambda}(A,d)}\]
obtained in the first part of this proof, together with the upper bounds
\[\P\left(|\CC(V_n)|\in[T,2T]\right) \le \P\left(|\CC(V_n)|\ge T\right) \le \frac{c_2}{A^{1/2}n^{1/3}}e^{-G_{\lambda}(A,d)}\]
which follows from \eqref{uncondUB} and
\[\P\left(\left.|\CC(V_n)|\ge 2T \,\right|\,\bS_n\right) \le \frac{c_3}{(2A)^{1/2}n^{1/3}}e^{-G_{\lambda}(2A,d)}\]
which follows from the upper bound in Theorem \ref{mainthm} proved in Section \ref{UBsec} we obtain
\[\mathbb{P}\left(\left.|\mathcal{C}_{\max}|\geq T \,\right|\,\bS_n\right)\geq \frac{c' n^{1/3}}{A}\bigg(\frac{c_1 e^{-G_{\lambda}(A,d)}}{A^{1/2}n^{1/3}} - \frac{c_3 e^{-G_{\lambda}(2A,d)}}{(2A)^{1/2}n^{1/3}}\bigg)^2\bigg(\frac{c_2 e^{-G_{\lambda}(A,d)}}{A^{1/2}n^{1/3}}\bigg)^{-1}.\]
Provided that $A$ is large enough, the second term inside the first set of parentheses is smaller than $1/2$ times the first, and the result follows.
\end{proof}

The remainder of Section \ref{LB_sec} is devoted to the proofs of the results used above. We start by proving Proposition \ref{animpprop} in Section \ref{animppropsec}, and then we establish our two main tools, namely Propositions \ref{propfirstbit} and \ref{propsecondbit}, in Sections \ref{propfirstbit_sec} and \ref{propsecondbit_sec} respectively. In Section \ref{propsecondbit_sec} we will use two lemmas whose proofs we delay until Section \ref{aux_sec}. The most substantial of these is Lemma \ref{mainlemmasecondbit}, where we use a strong Brownian approximation to estimate the probability that a random walk remains above the curve $q(t)$ seen above.

\subsection{Proof of Proposition \ref{animpprop}}\label{animppropsec}
	Recall that we defined
	\begin{equation}\label{Fdef}
	F_i = \{v(h_i)\in \calV_{i-1}^d\}, \,\,\,\, F_i' = \{v(h_i)\in\calV_{i-1}^{d-1}\} \,\,\text{ and }\,\, F^-_i = \bigg\{v(h_i)\in\bigcup_{m=1}^{d-2}\calV_{i-1}^m\bigg\}
	\end{equation}
	and also
	\begin{equation}\label{deldef}
	\delta_i = \ind_{R_i}\ind_{F_i}(d-1) + \ind_{R_i}\ind_{F'_i}(d-2) - \ind_{\{h_i\in \calA_{i-1}\}} - 1,
	\end{equation}
	whereas
	\begin{equation}\label{Ddef}
	D_i = \ind_{R_i}(d-1)-1.
	\end{equation}
	We want to quantify the difference between $\delta_i$ and $D_i$. There are essentially three ways in which the two objects can differ: if $R_i\cap F_i^-$ occurs, then $\delta_i = -1$ whereas $D_i = d-2$; if $h_i\in\calA_{i-1}$, then $\delta_i = -2$ whereas $D_i$ could either equal $d-2$ or $-1$; and if $R_i\cap F_i'$ occurs, then $\delta_i = d-3$ whereas $D_i = d-2$. The first two of these events occur infrequently, which we show in Section \ref{animpprop_pt1}. The third event, $R_i\cap F_i'$, is then the main contribution to the difference between $\delta_i$ and $D_i$, and we control how often this event occurs in Section \ref{animpprop_pt2}.
	
	\subsubsection{The events $R_i\cap F_i^-$ and $\{h_t\in\calA_{t-1}\}$ rarely occur for $t\in[T]$}\label{animpprop_pt1}
	We begin this section by showing that the number of vertices with at most $d-2$ unseen stubs is unlikely to be too large. The bound provided---using a straightforward Chernoff estimate---is not the best possible, but will suffice for our purposes.
	
	\begin{lem}\label{lemmaforfirstrandomsumtoremove}
		Let $i\in[T-1]$. Then, for all $l>0$ and sufficiently large $n\in\N$, we have that
		\[\mathbb{P}\left(\left|\bigcup_{m=1}^{d-2}\mathcal{V}^{(m)}_i\right|>\frac{i^2}{n}+l\right)\leq Ce^{-\frac{n^{1/2}l}{T}}\]
		for some finite constant $C$.
	\end{lem}
	
	\begin{proof}
	Note that, for every $r>0$ we have
	\begin{equation}\label{Vbd1eq}
	\mathbb{P}\left(\left|\bigcup_{m=1}^{d-2}\mathcal{V}^{(m)}_i\right|>\frac{i^2}{n}+l\right)\leq e^{-ri^2/n-rl}\mathbb{E}\left[e^{r\left|\bigcup_{m=1}^{d-2}\mathcal{V}^{(m)}_i\right|}\right].
	\end{equation}
	Observe that $\left|\bigcup_{m=1}^{d-2}\mathcal{V}^{(m)}_{i-1}\right|$ can only increase during step $i$ if vertex $v(h_i)$ has $d-1$ unseen stubs (and we do not keep the edge $e_ih_i$ in the percolation). Thus we can write
	\[\left|\bigcup_{m=1}^{d-2}\mathcal{V}^{(m)}_i\right|
	\leq \left|\bigcup_{m=1}^{d-2}\mathcal{V}^{(m)}_{i-1}\right|+\ind_{\left\{v(h_i)\in \mathcal{V}^{(d-1)}_{i-1} \right\}}.\]
	Therefore, recalling that $\F_t$ is the $\sigma$-algebra generated by the exploration process up to step $t$, we have that
	\[\mathbb{E}\Big[e^{r\left|\bigcup_{m=1}^{d-2}\mathcal{V}^{(m)}_i\right|}\Big|\F_{i-1}\Big] \leq e^{r\left|\bigcup_{m=1}^{d-2}\mathcal{V}^{(m)}_{i-1}\right|}\left(1+\mathbb{P}\left(\left.v(h_i)\in \mathcal{V}^{(d-1)}_{i-1}\right|\F_{i-1}\right)(e^r-1)\right).\]
	Furthermore, at the end of step $i-1$, we can have at most $i-1$ vertices with $d-1$ unseen stubs; indeed, this can only happen if at each step $j\in [i-1]$ we pick a stub incident to a fresh vertex and we do not keep the edge $e_jh_j$. Thus
	\[\mathbb{P}\left(\left.v(h_i)\in \mathcal{V}^{(d-1)}_{i-1}\right|\F_{i-1}\right) \le \frac{(d-1)(i-1)}{dn-2(i-1)-1},\]
	and combining these two inequalities we obtain
	\begin{align*}
	\mathbb{E}\Big[e^{r\left|\bigcup_{m=1}^{d-2}\mathcal{V}^{(m)}_i\right|}\Big]&=\mathbb{E}\Big[\mathbb{E}\Big(e^{r\left|\bigcup_{m=1}^{d-2}\mathcal{V}^{(m)}_i\right|}\Big| \F_{i-1}\Big)\Big]\\
	&\leq \left(1+\frac{(d-1)(i-1)}{dn-2(i-1)-1}(e^r-1)\right) \mathbb{E}\left[e^{r\left|\bigcup_{m=1}^{d-2}\mathcal{V}^{(m)}_{i-1}\right|}\right].
	\end{align*}
	Iterating and using the inequality $1+x\leq e^x$ valid for all $x\in \mathbb{R}$, we see that
	\begin{align*}
	\mathbb{E}\left[e^{r\left|\bigcup_{m=1}^{d-2}\mathcal{V}^{(m)}_i\right|}\right]\leq \prod_{j=1}^{i-1}\left(1+\frac{(d-1)j}{dn-2j-1}(e^r-1)\right) \le \exp\bigg((e^r-1)\sum_{j=1}^{i-1}\frac{j}{n(1-\frac{2j+1}{dn})}\bigg).
	\end{align*}
	Finally, since for large $n$ and $j\le i\le T$ we have $\frac{2j+1}{dn}\le 1/2$, the above is at most $\exp((e^r-1)i^2/n)$, and using the inequality $e^x\le 1+x+x^2$ valid for $x\in[0,1]$, provided $r\in(0,1]$ we have
	\[\mathbb{E}\left[e^{r\left|\bigcup_{m=1}^{d-2}\mathcal{V}^{(m)}_i\right|}\right] \le \exp\bigg(\frac{ri^2}{n} + \frac{r^2i^2}{n}\bigg).\]
	Substituting this into \eqref{Vbd1eq}, we have for $r\in(0,1]$,
	\begin{align*}
	\mathbb{P}\left(\left|\bigcup_{m=1}^{d-2}\mathcal{V}^{(m)}_i\right|>\frac{i^2}{n}+l\right)\leq e^{-ri^2/n-rl}\exp\left(r\frac{i^2}{n}+r^2\frac{i^2}{n}\right)=e^{-rl+r^2 i^2/n}.
	\end{align*}
	Taking $r=n^{1/2}/T \ll 1$ we thus see that
	\begin{align*}
	\mathbb{P}\left(\left|\bigcup_{m=1}^{d-2}\mathcal{V}^{(m)}_i\right|>\frac{i^2}{n}+l\right)\leq Ce^{-\frac{ln^{1/2}}{T}}
	\end{align*}
	for some finite constant $C$, which completes the proof.
\end{proof}

By a similar method we show that the number of active stubs is unlikely to be too large.

	\begin{lem} \label{lemmaforsecondrandomsumtoremove}
	For any $t\in\N$ and $\omega\le 3t/(d-1)$,
	\[\P(\exists i\in [t\wedge\tau] : |\mathcal{A}_i|>\omega) \le C\exp\left(-\frac{\omega^2}{4p(d-1)t} + \frac{\omega}{2}\left(1-\frac{1}{p(d-1)}\right)\right)\]
	where $C=C(d)$ is a finite constant which depends on $d$.
	\end{lem}

	\begin{proof}
	We note first that for $i\le \tau$,
	\[|\mathcal{A}_{i}|=d+\sum_{j=1}^{i}\eta_j\leq d+\sum_{j=1}^{i}\left(\ind_{R_j}(d-1)-1\right)=d+\sum_{j=1}^{i}D_j\]
	and therefore, for any $r>0$,
	\[\P(\exists i\in [t\wedge\tau] : |\mathcal{A}_i|>\omega) \le \P\left(\max_{i\le t\wedge\tau} e^{r\sum_{j=1}^{i}D_j} > e^{r(\omega-d)}\right) \le \P\left(\max_{i\le t} e^{r\sum_{j=1}^{i}D_j} > e^{r(\omega-d)}\right).\]
	Since the left-hand side above is monotone increasing in $\lambda$, we may without loss of generality suppose that $\lambda\ge 0$, so that $p\ge 1/(d-1)$. Then the process $(\sum_{j=1}^{i}D_j, i\ge 0)$ is a submartingale, and therefore so is $\exp(r\sum_{j=1}^{i}D_j)$ for any $r>0$. From Doob's submartingale inequality we obtain
	\begin{equation}\label{Doobsub}
	\P(\exists i\in [t\wedge\tau] : |\mathcal{A}_i|>\omega) \leq \E\left[e^{r\sum_{j=1}^{t}D_j}\right]e^{-r(\omega-d)}.
	\end{equation}
	Now observe that
	\begin{align*}
	\E[e^{r\sum_{j=1}^{t}D_j}] = \prod_{i=1}^t \left(e^{r(d-2)}p + e^{-r}(1-p)\right) &= e^{-rt}\prod_{i=1}^t (1+p(e^{r(d-1)}-1))\\
	&\le \exp\left(-rt + pt(e^{r(d-1)}-1)\right).
	\end{align*}
	Choosing $r=\frac{\omega}{2p(d-1)t}$, since $\omega\le 3t/(d-1)$ we note that $r(d-1)\le 3/2$ and therefore $e^{r(d-1)}-1\le r(d-1)+r^2(d-1)^2$. Combining this with \eqref{Doobsub}, we have
	\[\P(\exists i\in [t\wedge\tau] : |\mathcal{A}_i|>\omega) \leq \exp\left(-rt + pt(r(d-1)+r^2(d-1)^2) -r\omega + rd\right)\]
	which, after simplifying, gives the desired result.
	\end{proof}

We can then apply the above two lemmas to show our main result for this section, which says that $\sum_{i=1}^t (\ind_{R_i\cap F_i^-} + \ind_{\{h_t\in\calA_{t-1}\}})$ is likely to be small when $t\in [T]$. Again we use a straightforward Chernoff bound.

\begin{cor}\label{firstrandomsumtoremove}
	Suppose that $h=h_n$ satisfies $1\ll h \ll T^{1/2}$. Then for all sufficiently large $n\in\N$, we have that
	\[\mathbb{P}\left(\exists t\in [T\wedge\tau] : \sum_{i=1}^{t}(\ind_{R_i\cap F_i^-} + \ind_{\{h_t\in\calA_{t-1}\}}) > \frac{4t^3}{3dn^2} + \left(\frac{8T^{1/2}t}{n}+1\right)h\right)\leq CTe^{-h} + e^{-ch^2},\]
	for some finite constant $C$ and $c=c_d>0$ which depends only on $d$.
\end{cor}

\begin{proof}
	Define the event
	\[V_{t,h} = \bigcap_{i=1}^{t-1}\left\{\left|\bigcup_{m=1}^{d-2}\mathcal{V}^{(m)}_i\right|\le \frac{i^2}{n}+\frac{Th}{n^{1/2}},\, |\calA_i|\le T^{1/2}h\right\}\]
	and also let $A_t = \{h_t\in\calA_{t-1}\}$. We split the probability that we want to bound depending on whether or not $V_{T\wedge\tau,h}$ occurs. That is, we have
	\begin{align}
	&\mathbb{P}\left(\exists t\in [T\wedge\tau]:\sum_{i=1}^t \ind_{(R_i\cap F_i^-)\cup A_i}>\frac{4t^3}{3dn^2} + \left(\frac{8T^{1/2}t}{n}+1\right)h\right)\nonumber\\
	&\le \sum_{t=1}^T \P\left(t\le\tau,\,\sum_{i=1}^t \ind_{(R_i\cap F_i^-)\cup A_i}>\frac{4t^3}{3dn^2} + \left(\frac{8T^{1/2}t}{n}+1\right)h,\,V_{T\wedge\tau,h}\right) + \P(V_{T\wedge\tau,h}^c)\nonumber\\
	&\le \sum_{t=1}^T \E\left[e^{\sum_{i=1}^t \ind_{(R_i\cap F_i^-)\cup A_i}}\ind_{V_{t,h}}\right]\exp\left(-\frac{4t^3}{3dn^2} - \left(\frac{8T^{1/2}t}{n}+1\right)h\right) + \P(V_{T\wedge\tau,h}^c)\label{Vprobbreak}
	\end{align}
	Applying Lemma \ref{lemmaforfirstrandomsumtoremove} with $l=Th/n^{1/2}$ and taking a union bound over $i=1,\ldots,T-1$, and also applying Lemma \ref{lemmaforsecondrandomsumtoremove} with $\omega = T^{1/2}h$, we see that
	\begin{equation}\label{VThbd}
	\mathbb{P}\left(V_{T\wedge\tau,h}^c\right)\leq C(T-1)e^{-h} + e^{-ch^2}.
	\end{equation}

	For the remaining term on the right-hand side of \eqref{Vprobbreak}, we will apply Markov's inequality to $\exp(\sum_{i=1}^t \ind_{(R_i\cap F_i^-)\cup A_i})$. To this end, recalling that $\F_t$ is the $\sigma$-algebra generated by the exploration process up to step $t$, we now focus on bounding
	\begin{equation}\label{Vtheq}
	\E\left[e^{\sum_{i=1}^t \ind_{(R_i\cap F_i^-)\cup A_i}}\ind_{V_{t,h}}\right] = \E\left[\E\left[\left.e^{\ind_{(R_t\cap F_t^-)\cup A_t}}\right|\F_{t-1}\right] e^{\sum_{i=1}^{t-1} \ind_{(R_i\cap F_i^-)\cup A_i}}\ind_{V_{t,h}}\right].
	\end{equation}

%
%
	Note that
	\begin{align*}
	\E\left[\left.e^{\ind_{(R_t\cap F_t^-)\cup A_t}}\right|\F_{t-1}\right]&=1+\mathbb{P}((R_t\cap F_t^-)\cup A_t |\F_{t-1})(e-1)\\
	&=1+p\mathbb{P}(F^-_t|\F_{t-1})(e-1) + \P(A_t|\F_{t-1})(e-1)\\
	&\le 1+2p\mathbb{P}(F^-_t|\F_{t-1}) + 2\P(A_t|\F_{t-1})
	\end{align*}
	and on the event $V_{t,h}$,
	\begin{align*}
	2p\mathbb{P}(F^-_t|\F_{t-1}) + 2\P(A_t|\F_{t-1}) &\le \frac{2p(d-2)\left|\bigcup_{m=1}^{d-2}\mathcal{V}^{(m)}_{t-1}\right|}{dn-2(t-1)-1} + \frac{2\left|\calA_{t-1}\right|}{dn-2(t-1)-1}\\
	&\le \frac{2p(d-2)\left(\frac{(t-1)^2}{n}+\frac{Th}{n^{1/2}}\right)}{dn-2(t-1)-1} + \frac{2T^{1/2}h}{dn-2(t-1)-1}.
	\end{align*}
	When $n$ is large, we have $2(t-1)+1 \le dn/2$, $p(d-2)\le 1$ and $T\le n$, and therefore
	\begin{align*}
	\E\left[\left.e^{\ind_{(R_t\cap F_t^-)\cup A_t}}\right|\F_{t-1}\right]
	&\le 1 + \frac{4}{dn^2}\big((t-1)^2+Thn^{1/2}\big) + \frac{4T^{1/2}h}{dn}\\
	&\le 1 + \frac{4}{dn^2}\big((t-1)^2+2T^{1/2}hn\big).
	\end{align*}
	Substituting this estimate into \eqref{Vtheq}, we obtain that
	\begin{align*}
	\E\left[e^{\sum_{i=1}^{t}\ind_{(R_i\cap F_i^-)\cup A_i}}\ind_{V_{t,h}}\right] &\le \E\left[\left(1 + \frac{4}{dn^2}\big((t-1)^2+2T^{1/2}hn\big)\right) e^{\sum_{i=1}^{t-1}\ind_{(R_i\cap F_i^-)\cup A_i}}\ind_{V_{t,h}}\right]\\
	&\le \left(1 + \frac{4}{dn^2}\big((t-1)^2+2T^{1/2}hn\big)\hspace{-1mm}\right) \E\left[ e^{\sum_{i=1}^{t-1}\ind_{(R_i\cap F_i^-)\cup A_i}}\ind_{V_{t-1,h}}\right].
	\end{align*}
	Iterating, and then using the inequality $1+x\le e^x$ valid for all $x\in\mathbb{R}$, we have
	\begin{align*}
	\E\left[e^{\sum_{i=1}^{t}\ind_{(R_i\cap F_i^-)\cup A_i}}\ind_{V_{t,h}}\right] &\le \prod_{i=1}^t \left(1 + \frac{4}{dn^2}\big((i-1)^2+2T^{1/2}hn\big)\right)\\
	&\le \exp\left(\frac{4}{dn^2}\left(\frac{t^3}{3} + 2T^{1/2}hnt\right)\right).
	\end{align*}
	Substituting this and \eqref{VThbd} into \eqref{Vprobbreak}, we obtain
	\begin{align*}
	\mathbb{P}\left(\exists t\in [T]:\sum_{i=1}^t \ind_{(R_i\cap F_i^-)\cup A_i}>\frac{4t^3}{3dn^2} + \left(\frac{4}{d}+1\right)h\right)
	&\le \sum_{t=1}^T e^{-h} + C(T-1)e^{-h} + e^{-ch^2}
	\end{align*}
	and the result follows.
\end{proof}

\subsubsection{Controlling how often $R_i\cap F_i'$ occurs}\label{animpprop_pt2}

The purpose of our next result, proved in Subsection 4.6, is to control the (random) sums $\sum_{i=1}^{t}\ind_{R_i}\ind_{F^c_i}$ over the interval $[T]\setminus [T']$. Specifically, in Lemma \ref{rightcurveappears} below we substitute the process $(\sum_{i=1}^{t}\ind_{R_i}\ind_{F^c_i})_{t\in [T]\setminus [T']}$ with a deterministic function of $t$, which will be of great importance in order to obtain the correct exponential term in our lower bounds stated in Theorem \ref{mainthm}.
	\begin{lem}\label{rightcurveappears}
		Let $\theta>0$. Then, for all sufficiently large $n$, we have that
		\begin{equation}\label{upperomegaprime}
		\mathbb{P}\left(\exists t\in [T] :\sum_{i=1}^{t}\ind_{R_i}\ind_{F^c_i} > p(1-2/d)\frac{t^2}{2n}+\frac{2T^3}{n^2}+\theta\right)\leq CTe^{-n^{1/2}\theta/T}
		\end{equation}
		for some finite constant $C$.
	\end{lem}

\begin{proof}
	Note that for any $x>0$ and $r>0$,
	\begin{equation}\label{rcastart}
	\P\left(\sum_{i=1}^{t}\ind_{R_i}\ind_{F^c_i} > x\right) \le \mathbb{E}\left[e^{r\sum_{i=1}^{t}\ind_{R_i}\ind_{F^c_i}}\right]e^{-rx}.
	\end{equation}
	Now since $\mathcal{V}^{(d)}_{0}$ consists of all vertices except $V_n$, and at most one vertex can be removed from $\mathcal{V}^{(d)}_{i}$ at each step $i$ of the exploration process, we have
	\[\left|\mathcal{V}^{(d)}_{i}\right| \geq n-1-i\]
	and therefore, recalling that $\F_t$ is the $\sigma$-algebra generated by the exploration process up to time $t$,
	\begin{align*}
	\mathbb{E}\left[\left.e^{r\ind_{R_t}\ind_{F^c_t}}\right|\mathcal{F}_{t-1}\right]&=1+p\mathbb{P}(F^c_t|\mathcal{F}_{t-1})(e^r-1)\\
	&=1+p\left(1-\mathbb{P}(F_t|\mathcal{F}_{t-1})\right)(e^r-1)\\
	&=1+p\left(1-\frac{d|\mathcal{V}^{(d)}_{t-1}|}{dn-2(t-1)-1}\right)(e^r-1)\\
	&\leq 1+p\left(1-\frac{d(n-1-(t-1))}{dn-2(t-1)-1}\right)(e^r-1).
	\end{align*}
	Thus
	\begin{align*}
	\mathbb{E}\left[e^{r\sum_{i=1}^{t}\ind_{R_i}\ind_{F^c_i}}\right]\leq \left(1+p\left(1-\frac{d(n-1-(t-1))}{dn-2(t-1)-1}\right)(e^r-1)\right) \mathbb{E}\left[e^{r\sum_{i=1}^{t-1}\ind_{R_i}\ind_{F^c_i}}\right]
	\end{align*}
	and iterating we obtain
	\begin{align*}
	\mathbb{E}\left[e^{r\sum_{i=0}^{t}\ind_{R_i}\ind_{F^c_i}}\right]&\leq \prod_{i=1}^{t-1}\left(1+p\left(1-\frac{d(n-1-i)}{dn-2i-1}\right)(e^r-1)\right)\\
	&\leq  \exp\left( p(e^r-1)\sum_{i=0}^{t-1}\left(1-\frac{d(n-1-i)}{dn-2i-1}\right)\right).
	\end{align*}
	Now
	\begin{align*}
	\sum_{i=0}^{t-1}\left(1-\frac{d(n-1-i)}{dn-2i-1}\right)&=\sum_{i=1}^{t}\left(1-\frac{n-1-i}{n}\frac{1}{1-\frac{2i+1}{dn}}\right)\\
	&\le \sum_{i=0}^{t-1}\left(1-\left(1-\frac{i+1}{n}\right)\left(1+\frac{2i+1}{dn}\right)\right)\\
	&= \sum_{i=0}^{t-1} \left(\frac{(d-2)i}{dn} + \frac{d-1}{n} + \frac{(i+1)(2i+1)}{dn^2}\right)\\
	&\le (1-2/d)\frac{t^2}{2n} + \frac{dt}{n} + \frac{2t^3}{dn^2}.
	\end{align*}
	Bounding $dt/n + 2t^3/dn^2 \le T^3/n^2$ for large $n$, and using the inequalities $e^r\le 1+r+r^2$ and $e^r \le 1+2r$ for $r\le 1$, we deduce that for $r\leq 1$,
	\begin{align*}
	\mathbb{E}\left[e^{r\sum_{i=1}^{t}\ind_{R_i}\ind_{F^c_i}}\right]&\leq \exp\left(p(e^r-1)\left( (1-2/d)\frac{t^2}{2n}+\frac{T^3}{n^2}\right)\right)\\
	&\leq \exp\left(rp(1-2/d) \frac{t^2}{2n}+r^2p (1-2/d)\frac{t^2}{2n}+\frac{2rT^3}{n^2}\right).
	\end{align*}
	Therefore, substituting into \eqref{rcastart}, if $n$ is large enough then (provided $r\leq 1$) we obtain 
	\begin{align*}
	\mathbb{P}\left(\sum_{i=1}^{t}\ind_{R_i}\ind_{F^c_i}\geq p (1-2/d)\frac{t^2}{2n}+\frac{2T^3}{n^2}+\theta\right)&\leq e^{r^2p(1-2/d)t^2/2n-r\theta} \leq e^{r^2 T^2/2n-r\theta}.
	\end{align*}
	The desired conclusion follows by taking $r=n^{1/2}/T$ and using a union bound.
\end{proof}

We now have the ingredients to complete the proof of Proposition \ref{animpprop}.

\begin{proof}[Proof of Proposition \ref{animpprop}]
Recalling again the definitions \eqref{deldef} and \eqref{Ddef} of $\delta_i$ and $D_i$ respectively, and also \eqref{Fdef}, we observe that
\begin{align*}
0\le D_i - \delta_i &= \ind_{R_i\cap F_i'} + (d-1)\ind_{R_i\cap F_i^-} + \ind_{\{h_i\in\calA_{i-1}\}}(1+(d-1)\ind_{R_i})\\
&\le \ind_{R_i\cap F_i^c} + d(\ind_{R_i\cap F_i^-} + \ind_{\{h_i\in\calA_{i-1}\}}).
\end{align*}
and therefore, since $\tau_\delta\le \tau$,
\begin{align}
&\P\left(\exists t\in[T\wedge\tau_\delta] : \sum_{i=1}^t (D_i - \delta_i) > p(1-2/d)\frac{t^2}{2n} + An^{4/15}\right)\nonumber\\
&\hspace{15mm} \le \P\left(\exists t\in[T] : \sum_{i=1}^t \ind_{R_i\cap F_i^c} > p(1-2/d)\frac{t^2}{2n} + \frac{2An^{4/15}}{3}\right)\nonumber\\
&\hspace{40mm} + \P\left(\exists t\in[T\wedge\tau] : \sum_{i=1}^t (\ind_{R_i\cap F_i^-} + \ind_{\{h_i\in\calA_{i-1}\}}) > \frac{An^{4/15}}{3d}\right).\label{animeq1}
\end{align}
We now apply Corollary \ref{firstrandomsumtoremove} with $h=An^{1/10}$. This tells us that for all sufficiently large $n$ we have
	\begin{multline*}
	\mathbb{P}\left(\exists t\in [T\wedge\tau] : \sum_{i=1}^{t}(\ind_{R_i\cap F_i^-} + \ind_{\{h_t\in\calA_{t-1}\}}) > \frac{4t^3}{3dn^2} + \left(\frac{8T^{1/2}t}{n}+1\right)An^{1/10}\right)\\
	\leq CTe^{-An^{1/10}} + e^{-cA^2n^{1/5}},
	\end{multline*}
and since for $t\le T$ we have
\[\frac{4t^3}{3dn^2} + \left(\frac{8T^{1/2}t}{n}+1\right)An^{1/10} = o(An^{2/5})\]
we deduce that for large $n$,
\begin{equation}\label{drosseq1}
\mathbb{P}\left(\exists t\in [T\wedge\tau] : \sum_{i=1}^{t}(\ind_{R_i\cap F_i^-} + \ind_{\{h_t\in\calA_{t-1}\}}) > \frac{An^{2/5}}{3d}\right)	\leq CTe^{-An^{1/10}}.
\end{equation}

Next we apply Lemma \ref{rightcurveappears} with $\theta = An^{4/15}/2$, which tells us that for all sufficiently large $n$ we have
\begin{equation*}
\mathbb{P}\left(\exists t\in [T] :\sum_{i=1}^{t}\ind_{R_i}\ind_{F^c_i} > p(1-2/d)\frac{t^2}{2n}+\frac{2T^3}{n^2}+An^{4/15}/2\right)\leq C'Te^{-An^{23/30}/2T}.
\end{equation*}
Since $2T^3/n^2 = o(An^{4/15})$ and $An^{23/30}/2T = \Theta_d(n^{1/10})$, combining this with \eqref{drosseq1} and substituting the estimates into \eqref{animeq1} we obtain that
\[\P\left(\exists t\in[T\wedge\tau_\delta] : \sum_{i=1}^t (D_i - \delta_i) > p(1-2/d)\frac{t^2}{2n} + An^{4/15}\right)\leq CTe^{-An^{1/10}} + C'Te^{-cn^{1/10}}.\]
This completes the proof.
\end{proof}

\subsection{The probability of staying positive and finishing above $\eps\sqrt{T'}$: proof of Proposition \ref{propfirstbit}}\label{propfirstbit_sec}

Here we aim to show that for a sufficiently small constant $\eps>0$,
\begin{equation}\label{kklk}
\mathbb{P}\bigg(\sum_{i=1}^t \delta'_i > -d \hspace{0.15cm}\forall t\in [T'],\, \sum_{i=1}^{T'} \delta'_i \geq \eps \sqrt{T'}\,\bigg|\,\bS_n\bigg) \ge \frac{cA}{n^{1/3}},
\end{equation}
where we recall that
\[\delta'_i = \ind_{R_i}\ind_{F_i}(d-1)-\ind_{\{h_i\in \mathcal{A}_{i-1}\}} - 1.\]
We do this in two parts: we show that, if we replace $\delta'_i$ with the simpler
\[\delta''_i \coloneqq \ind_{R_i}\ind_{F_i}(d-1) - 1,\]
then for any $\gamma\in[0,1/2)$,
\begin{equation}\label{delprimeabovegamma}
\mathbb{P}\bigg(\sum_{i=1}^t \delta''_i \ge t^\gamma \hspace{0.15cm}\forall t\in [T'],\, \sum_{i=1}^{T'} \delta''_i \geq 2\eps\sqrt{T'}\,\bigg|\,\bS_n\bigg) \ge \frac{cA}{n^{1/3}},
\end{equation}
and then we show that for $\gamma\in(0,1/2)$,
\begin{equation}\label{activesumbelowgamma}
\mathbb{P}\bigg(\exists t\in[T'\wedge\tau] : \sum_{i=1}^t \ind_{\{h_i\in \mathcal{A}_{i-1}\}} \ge t^\gamma\,\bigg|\,\bS_n\bigg) \le \frac{c}{n^{1/2}} \ll \frac{A}{n^{1/3}},
\end{equation}
where we recall that
\[\tau = \min\{t\in\N : |\mathcal A_t|=0\} = \min\bigg\{t\in\N : d+\sum_{i=1}^t \eta_i = 0\bigg\}.\]
We will then combine \eqref{delprimeabovegamma} and \eqref{activesumbelowgamma} to obtain \eqref{kklk}, proving Proposition \ref{propfirstbit}. We note that the choice of $\gamma$ is not important above; one may choose, for example, $\gamma=1/4$ in both \eqref{delprimeabovegamma} and \eqref{activesumbelowgamma}. We retain the general $\gamma$ in the proofs since this is no extra work.

To prove \eqref{delprimeabovegamma}, we will use a coupling and a change of measure to replace $(\delta''_i)_{i\ge 1}$ with i.i.d.~Bernoulli random variables whose parameter does not depend on $n$, and then apply a theorem of Ritter \cite{ritter:growth_RW_cond_positive}. We will prove \eqref{activesumbelowgamma} by applying Lemma \ref{lemmaforsecondrandomsumtoremove} to show that the number of active stubs is never too large; then we will break $[T']$ up into two smaller intervals, replace the barrier $(t^\gamma)_{t\ge1}$ with a constant barrier on each of these smaller intervals, and use simple Markov and Chernoff bounds to complete the proof.

\subsubsection{Proof of \eqref{delprimeabovegamma}, step 1: removing the conditioning on $\bS_n$}

We begin by removing the conditioning on $\bS_n$. The idea boils down to the fact that the probability of creating a non-simple edge before step $T'$ of the exploration process is at most $\frac{c}{A^2 n^{1/3}}$, and the probability of creating a non-simple edge after step $T'$ is of the same order as the probability that we create a non-simple edge anywhere in the graph, regardless of what happens in the first $T'$ steps.

\begin{lem}\label{removecondS}
For large $n$,
\begin{multline*}
\mathbb{P}\bigg(\sum_{i=1}^t \delta''_i \ge t^\gamma \hspace{0.15cm}\forall t\in [T'],\, \sum_{i=1}^{T'} \delta''_i \geq 2\eps\sqrt{T'}\,\bigg|\,\bS_n\bigg)\\
\ge \frac{1}{2}\mathbb{P}\bigg(\sum_{i=1}^t \delta''_i \ge t^\gamma \hspace{0.15cm}\forall t\in [T'],\, \sum_{i=1}^{T'} \delta''_i \geq 2\eps\sqrt{T'}\bigg) - \frac{c}{A^2 n^{1/3}}
\end{multline*}
where $c=c(d)$ is a finite constant that depends only on $d$.
\end{lem}

In order to prove this result, we need the following lemma that appears in \cite{nachmias:critical_perco_rand_regular}.

\begin{lem}[{\cite[Lemma 23]{nachmias:critical_perco_rand_regular}}]\label{lemnach}
	Let $d\geq 3$ be fixed and let $\bar{d}_1,\bar{d}_2\in \{1,\dots,d\}^m$ be degree sequences of length $m$ such that each sequence sums to an even number. Let $\mathbb{P}_1$ be the distribution of a uniform perfect matching on $\sum_{i=1}^{m}\bar{d}_1(i)$ vertices, divided into $m$ tuples such that the i-th tuple has $\bar{d}_1(i)$ vertices in it. Let $\mathbb{S}$ be the event that contracting each tuple into a single vertex yields a simple graph. Assume $m\rightarrow \infty$. If $\bar{d}_1=(d,\dots,d)$ and $\bar{d}_2$ has $(1 - o(1))m$ entries with the value $d$ then
	\begin{equation*}
	\mathbb{P}_2(\mathbb S) = (1 + o(1))\mathbb{P}_1(\mathbb S).
	\end{equation*}
\end{lem}

We now return to the proof of our lemma.

\begin{proof}[Proof of Lemma \ref{removecondS}]
Let $\bS_n(t)$ be the event that the multigraph produced by the first $t$ steps of the exploration process (including both retained and non-retained edges) is simple. We first claim that
\begin{equation}\label{NPsimple}
\P(\bS_n|\F_{T'}) = (1-o(1)) \P(\bS_n)\ind_{\bS_n(T')}.
\end{equation}
This follows from Lemma \ref{lemnach} above, since $|\calV^{(d)}_{T'}|\ge n-1-T' = (1-o(1)) n$ (indeed, at most one vertex can be removed from the set of fresh vertices at each step). Next we claim that for large $n$,
\begin{equation}\label{probSsmall}
\P(\bS_n(T')^c)\le \frac{2d}{A^2 n^{1/3}}.
\end{equation}
Indeed, at each step $t$ of the exploration process, to create a non-simple edge, $e_t$ has to pair either with one of its sister stubs (i.e.~those associated to its own vertex), of which there are $d-1$, or with a sister stub of a stub that one of its sisters has already been paired with, of which there are at most $(d-1)^2$. Thus the probability of creating a non-simple edge at step $t$ is at most
\[\frac{(d-1) + (d-1)^2}{dn-2(t-1)-1},\]
which is at most $2d/n$ for $t\le T'$ when $n$ is large. Our claim \eqref{probSsmall} then follows by taking a union bound over all $t\le T'$.

Now, writing $V$ for the event of interest,
\[V \coloneqq \bigg\{ \sum_{i=1}^t \delta''_i \ge t^\gamma \hspace{0.15cm}\forall t\in [T'],\, \sum_{i=1}^{T'} \delta''_i \geq 2\eps\sqrt{T'} \bigg\},\]
we have
\[\P(V\cap \bS_n) = \E\left[ \P(V\cap\bS_n | \F_{T'}) \right] = \E\left[ \ind_V \P(\bS_n | \F_{T'})\right].\]
By \eqref{NPsimple}, this equals
\[(1+o(1))\P(\bS_n)\P(V\cap \bS_n(T')) = (1+o(1))\P(\bS_n)\big(\P(V) - \P(V\cap \bS_n(T')^c)\big),\]
and we deduce that
\[\P(V | \bS_n) = (1+o(1))\big(\P(V) - \P(V\cap \bS_n(T')^c)\big) \ge (1+o(1))\big(\P(V) - \P(\bS_n(T')^c)\big).\]
The result now follows from \eqref{probSsmall}.
\end{proof}

\subsubsection{Proof of \eqref{delprimeabovegamma}, step 2: replacing $\delta''_i$ with i.i.d.~random variables}
Lemma \ref{removecondS} ensures that we do not need to include the conditioning on $\bS_n$ in order to prove \eqref{delprimeabovegamma}. The next step in the proof is to compare $\delta''_i$ with
\[\Delta''_i\coloneqq \ind_{R_i}\ind_{\{U_i\le 1-T'/n\}}(d-1)-1,\]
where $(U_i)_{i\ge 0}$ is a sequence of $U([0,1])$ random variables, independent of everything else. We would like to show the following.

\begin{prop}\label{deltatoDprop}
For any $\gamma\ge 0$ and $\eps>0$ we have
\[\P\bigg(\sum_{i=1}^t \delta''_i \ge t^\gamma \hspace{0.15cm}\forall t\in [T'],\, \sum_{i=1}^{T'} \delta''_i \geq \eps\sqrt{T'}\bigg) \ge \P\bigg(\sum_{i=1}^t \Delta''_i \ge t^\gamma  \hspace{0.15cm}\forall t\in [T'],\, \sum_{i=1}^{T'} \Delta''_i \geq \eps\sqrt{T'}\bigg).\]
\end{prop}

The proof of Proposition \ref{deltatoDprop} is almost identical to that of Proposition \ref{mainpropforupper}. We first note that, since $|\calV^{(d)}_i|\ge n-1-i$ for each $i\ge 1$ (indeed, at most one vertex can be removed from the set of fresh vertices at each step), the probability that $h_j$ is fresh when $j\in[T']$ is
\[\frac{d|\calV^{(d)}_{j-1}|}{dn-2(j-1)-1} \ge \frac{d(n-j)}{dn} = 1 - j/n \ge  1-T'/n.\]
Thus
\[\P(\delta''_j = d-2 \,|\,\F_{j-1}) \ge 1-T'/n = \P(\Delta''_j = d-2 \,|\,\F_{j-1}).\]
Since $\delta''_j$ and $\Delta''_j$ can only take the same two possible values, this is the same as saying that $\delta''_j$ stochastically dominates $\Delta''_j$, given $\F_{j-1}$. This is the equivalent of Lemma \ref{etamuobs}.

We apply this fact to prove the following lemma, which is the equivalent of Lemma \ref{etamumain} and the main ingredient in proving Proposition \ref{deltatoDprop}.

\begin{lem}\label{deltatoDlem}
Let
\[\bar S^{(j)}_t = d + \sum_{i=1}^{j\wedge t}\delta''_i + \sum_{i=(j\wedge t)+1}^t \Delta''_i.\]
For any $\gamma\ge 0$, $\eps>0$ and $j\in[T']$,
\begin{multline*}
\P\left(\left.\bar S_t^{(j)} \ge t^\gamma \,\,\forall t\in[T'],\, \bar S^{(j)}_{T'} \geq \eps\sqrt{T'}\,\right|\,\F_{j-1}\right) \\
\ge \P\left(\left.\bar S_t^{(j-1)} \ge t^\gamma \,\,\forall t\in[T'],\, \bar S^{(j-1)}_{T'} \geq \eps\sqrt{T'}\,\right|\,\F_{j-1}\right).
\end{multline*}
\end{lem}

\begin{proof}
We proceed almost exactly as in the proof of Lemma \ref{etamumain}, and therefore leave out some of the details. By summing over the possible values of $\bar S^{(j)}_{j-1}$ and using the $\F_{j-1}$-measurability of $(\bar S^{(j)}_i)_{i=1}^{j-1}$, we have
\begin{align*}
&\P\big(\bar S_t^{(j)} \ge t^\gamma \,\,\forall t\in[T'],\, \bar S^{(j)}_{T'} \geq \eps\sqrt{T'}\,\big|\,\F_{j-1}\big)\\
&= \sum_{s = 1}^{\infty} \ind_{\{\bar S^{(j)}_{j-1} = s\}}\ind_{\{\bar S^{(j)}_i \ge i^\gamma \,\,\forall i\in [j-1]\}}\\
&\hspace{10mm}\cdot\P\bigg(s + \delta''_j + \sum_{i=j+1}^t \Delta''_i \ge t^\gamma \,\, \forall t\in[T']\setminus[j-1],\, s + \delta''_j + \sum_{i=j+1}^{T'} \Delta''_i \ge \eps\sqrt{T'}\,\bigg|\,\F_{j-1}\bigg).
\end{align*}
Then summing further over the possible values $Q\coloneqq\{-1,d-2\}$ of $\delta''_j$, and using the independence of $(\Delta_i)_{i=j+1}^{T'}$ from $\F_{j-1}$, we have
\begin{align}
&\P\left(\left.\bar S_t^{(j)} \ge t^\gamma \,\,\forall t\in[T'],\, \bar S^{(j)}_{T'} \geq \eps\sqrt{T'}\,\right|\,\F_{j-1}\right)\nonumber\\
&= \sum_{s = 1}^{\infty} \ind_{\{\bar S^{(j)}_{j-1} = s\}}\ind_{\{\bar S^{(j)}_i \ge i^\gamma \,\,\forall i\in [j-1]\}}\nonumber\\
&\hspace{10mm}\cdot\sum_{q\in Q} \P\bigg(s+q+\sum_{i=j+1}^t \Delta''_i \ge t^\gamma \,\,\forall t\in[T']\setminus[j-1],\, s+q+\sum_{i=j+1}^{T'} \Delta''_i\ge \eps \sqrt{T'}\bigg)\nonumber\\
&\hspace{100mm}\cdot \P\big(\delta''_j = q\,\big|\,\F_{j-1}\big).\label{deltatoD1}
\end{align}
We also observe that exactly the same argument holds for $\bar S_t^{(j-1)}$ if we replace $\delta''_j$ with $\Delta''_j$; that is,
\begin{align}
&\P\left(\left.\bar S_t^{(j-1)} \ge t^\gamma \,\,\forall t\in[T'],\, \bar S^{(j-1)}_{T'} \geq \eps\sqrt{T'}\,\right|\,\F_{j-1}\right)\nonumber\\
&= \sum_{s = 1}^{\infty} \ind_{\{\bar S^{(j-1)}_{j-1} = s\}}\ind_{\{\bar S^{(j-1)}_i \ge i^\gamma \,\,\forall i\in [j-1]\}}\nonumber\\
&\hspace{10mm}\cdot\sum_{q\in Q} \P\bigg(s+q+\sum_{i=j+1}^t \Delta''_i \ge t^\gamma \,\,\forall t\in[T']\setminus[j-1],\, s+q+\sum_{i=j+1}^{T'} \Delta''_i\ge \eps \sqrt{T'}\bigg)\nonumber\\
&\hspace{100mm}\cdot \P\big(\Delta''_j = q\,\big|\,\F_{j-1}\big).\label{deltatoD2}
\end{align}
Now since
\[\P\bigg(s+q+\sum_{i=j+1}^t \Delta''_i \ge t^\gamma \,\,\forall t\in[T']\setminus[j-1],\, s+q+\sum_{i=j+1}^{T'} \Delta''_i\ge \eps \sqrt{T'}\bigg)\]
is increasing in $q$, and $\delta''_j$ stochastically dominates $\Delta''_j$ given $\F_{j-1}$ (as shown above), we see that $\eqref{deltatoD1}\ge \eqref{deltatoD2}$, completing the proof.
\end{proof}

\begin{proof}[Proof of Proposition \ref{deltatoDprop}]
Using the notation in Lemma \ref{deltatoDlem} and taking expectations, we have
\[\P\left(\bar S_t^{(j)} \ge t^\gamma \,\,\forall t\in[T'],\, \bar S^{(j)}_{T'} \geq \eps\sqrt{T'}\right) \ge \P\left(\bar S_t^{(j-1)} \ge t^\gamma \,\,\forall t\in[T'],\, \bar S^{(j-1)}_{T'} \geq \eps\sqrt{T'}\right).\]
Iterating, we obtain
\[\P\left(\bar S^{(T')}_t \ge t^\gamma \hspace{0.15cm}\forall t\in [T'],\, \bar S^{(T')}_{T'} \geq \eps\sqrt{T'}\right) \ge \P\left(\bar S^{(0)}_t \ge t^\gamma \hspace{0.15cm}\forall t\in [T'],\, \bar S^{(0)}_{T'} \geq \eps\sqrt{T'}\right).\]
However $S^{(T')}_t = d+\sum_{i=1}^t \delta''_i$ and $S^{(0)}_t = d+\sum_{i=1}^t \Delta''_i$, so the above is exactly the result we are trying to prove.
\end{proof}

\subsubsection{Proof of \eqref{delprimeabovegamma}, step 3: replacing $\Delta''_i$ with mean-zero random variables that do not depend on $n$}
Although the random variables $(\Delta''_i)_{i=1}^{T'}$ are i.i.d.~and therefore easier to work with than $(\delta''_i)_{i=1}^{T'}$, their distribution depends on $n$ and has a small but non-zero drift. We now replace $\Delta''_i$ with mean-zero random variables that do not depend on $n$.

\begin{lem}\label{DeltatoDlem}
Let $(D''_i)_{i=0}^{T'}$ be a sequence of i.i.d.~random variables with $\P(D''_i=d-2) = \frac{1}{d-1} = 1-\P(D''_i=-1)$. Suppose that $|\lambda|=O(A)$. Then for any $\gamma\ge0$ and $\eps>0$,
\begin{multline*}
\P\bigg(\sum_{i=1}^t \Delta''_i \ge t^\gamma \hspace{0.15cm}\forall t\in [T'],\, \sum_{i=1}^{T'} \Delta''_i \geq \eps\sqrt{T'}\bigg)\\
\ge c\P\bigg(\sum_{i=1}^t D''_i \ge t^\gamma \hspace{0.15cm}\forall t\in [T'],\, \sum_{i=1}^{T'} D''_i \in \left[\eps\sqrt{T'},\sqrt{T'}/\eps\right]\bigg).
\end{multline*}
where $c=c(d,\eps)>0$ is a constant depending only on $d$ and $\eps$.
\end{lem}

To prove Lemma \ref{DeltatoDlem} we will use an exponential change of measure with a specific parameter $\nu$. We first calculate some asymptotics for quantities involving $\nu$.

\begin{lem}
Let
\begin{equation}\label{nu}
\nu = \frac{1}{d-1}\log(1-p(1-T'/n)) - \frac{1}{d-1}\log((d-2)p(1-T'/n)).
\end{equation}
Then
\begin{equation}\label{nudef}
\nu = \frac{T'}{n(d-2)}-\frac{\lambda}{n^{1/3}(d-2)}+O\left(\frac{\lambda^2}{n^{2/3}} + \frac{\lambda T'}{n^{4/3}} + \frac{(T')^2}{n^2}\right).
\end{equation}
and
\begin{equation}\label{lboundmgf}
\mathbb{E}\left[e^{\nu \Delta''_1}\right] = \exp\left(O\left(\frac{\lambda^2}{n^{2/3}} + \frac{\lambda T'}{n^{4/3}} + \frac{(T')^2}{n^2}\right)\right).
\end{equation}
\end{lem}

\begin{proof}
Note that, since $p=(1+\lambda n^{-1/3})(d-1)^{-1}$, we have
\begin{align*}
&\frac{1}{d-1}\log(1-p(1-T'/n))\\
&\hspace{10mm}=\frac{1}{d-1}\log\left(1-\frac{1}{d-1}+\frac{T'}{n(d-1)}-\frac{\lambda}{n^{1/3}(d-1)}+\frac{\lambda T'}{n^{4/3}(d-1)}\right)\\
&\hspace{10mm}=\frac{1}{d-1}\log\left(\frac{d-2}{d-1}\right) + \frac{1}{d-1}\log\left(1+\frac{T'}{n(d-2)}-\frac{\lambda}{n^{1/3}(d-2)}+\frac{\lambda T'}{n^{4/3}(d-2)}\right).
\end{align*}
Since $\log(1+x) = x - x^2/2 +O(x^3)$ for $x\in(-1,1)$, we have
\begin{multline*}
\log\left(1+\frac{T'}{n(d-2)}-\frac{\lambda}{n^{1/3}(d-2)}+\frac{\lambda T'}{n^{4/3}(d-2)}\right)\\
=\frac{T'}{n(d-2)}-\frac{\lambda}{n^{1/3}(d-2)}-O\left(\frac{\lambda^2}{n^{2/3}} + \frac{\lambda T'}{n^{4/3}} + \frac{(T')^2}{n^2}\right).
\end{multline*}
Therefore we have shown that
\begin{align*}
&\frac{1}{d-1}\log(1-p(1-T'/n))\\
&=\frac{1}{d-1}\log\left(\frac{d-2}{d-1}\right)+\frac{1}{(d-1)(d-2)}\left(\frac{T'}{n}-\frac{\lambda}{n^{1/3}}\right)-O\left(\frac{\lambda^2}{n^{2/3}} + \frac{\lambda T'}{n^{4/3}} + \frac{(T')^2}{n^2}\right).
\end{align*}
Very similar calculations lead to
\begin{align*}
&\frac{1}{d-1}\log((d-2)p(1-T'/n))\\
&\hspace{0.5cm}=\frac{1}{d-1}\log\left(\frac{d-2}{d-1}\right)-\frac{1}{d-1}\left(\frac{T'}{n}-\frac{\lambda}{n^{1/3}}\right)-O\left(\frac{\lambda^2}{n^{2/3}} + \frac{\lambda T'}{n^{4/3}} + \frac{(T')^2}{n^2}\right).
\end{align*} 
Combining these two estimates leads to \eqref{nudef}.

For \eqref{lboundmgf}, since $\Delta''_1\coloneqq \ind_{R_1}\ind_{\{U_1\le 1-T'/n\}}(d-1)-1$, we see that
\begin{align*}
\mathbb{E}\left[e^{\nu \Delta''_1}\right]&=e^{-\nu}\big(e^{\nu(d-1)}p(1-T'/n) + 1-p(1-T'/n)\big)\\
&= e^{-\nu}\left(1+p(1-T'/n)(e^{\nu(d-1)}-1)\right).
\end{align*}
Using the definition (\ref{nu}) of $\nu$ it is easy to see that
\begin{align*}
1+p(1-T'/n)&(e^{\nu(d-1)}-1)\\
&= 1+\frac{T'}{n(d-2)}-\frac{\lambda}{n^{1/3}(d-2)}+\frac{\lambda T'}{n^{4/3}(d-2)}\\
&= \exp\left(\frac{T'}{n(d-2)}-\frac{\lambda}{n^{1/3}(d-2)} + O\left(\frac{\lambda^2}{n^{2/3}} + \frac{\lambda T'}{n^{4/3}} + \frac{(T')^2}{n^2}\right)\right).
\end{align*}
Combining this with \eqref{nudef} gives \eqref{lboundmgf} and completes the proof.
\end{proof}

Equipped with these estimates, we can now return to the proof of our main result for this section.

\begin{proof}[Proof of Lemma \ref{DeltatoDlem}]
Define a new probability measure $\Q$ by setting
\[\frac{\d\Q}{\d\P}\Big|_{\F_t} \coloneqq \frac{e^{\nu \sum_{i=1}^t \Delta''_i}}{\E[e^{\nu \sum_{i=1}^t \Delta''_i}]}\]
for each $t\ge 0$; since the sequence $(\Delta''_i)_{i\ge1}$ is i.i.d., the right-hand side of the definition forms a martingale and the definition is consistent. One may easily check that under $\Q$, the sequence $(\Delta''_i)_{i=1}^{T'}$ is i.i.d.~with
\[\Q(\Delta''_i = d-2 ) = \frac{1}{d-1} = 1-\Q(\Delta''_i = -1 ).\]

Now, we have
\begin{align*}
&\P\bigg(\sum_{i=1}^t \Delta''_i \ge t^\gamma \hspace{0.15cm}\forall t\in [T'],\, \sum_{i=1}^{T'} \Delta''_i \geq \eps\sqrt{T'}\bigg)\\
&\ge\P\bigg(\sum_{i=1}^t \Delta''_i \ge t^\gamma \hspace{0.15cm}\forall t\in [T'],\, \sum_{i=1}^{T'} \Delta''_i \in \big[\eps\sqrt{T'},\sqrt{T'}/\eps\big]\bigg)\\
&= \E_\Q\Big[e^{-\nu \sum_{i=1}^{T'} \Delta''_i} \ind_{\left\{\sum_{i=1}^t \Delta''_i \ge t^\gamma \hspace{0.15cm}\forall t\in [T'],\, \sum_{i=1}^{T'} \Delta''_i \in [\eps\sqrt{T'},\sqrt{T'}/\eps]\right\}}\Big] \E_\P\big[e^{\nu \sum_{i=1}^{T'} \Delta''_i}\big]\\
&\ge e^{-|\nu|\sqrt{T'}/\eps} \Q\bigg(\sum_{i=1}^t \Delta''_i \ge t^\gamma \hspace{0.15cm}\forall t\in [T'],\, \sum_{i=1}^{T'} \Delta''_i \in \big[\eps\sqrt{T'},\sqrt{T'}/\eps\big]\bigg) \E[e^{\nu \Delta''_1}]^{T'}
\end{align*}
By \eqref{lboundmgf} and the fact that $T' = \lfloor n^{2/3}/A^2\rfloor$ we see that 
\begin{equation*}
\E[e^{\nu \Delta''_1}]^{T'}\geq \exp\left(-O\left(\frac{\lambda^2}{A^2} + \frac{\lambda}{A^4} + \frac{1}{A^6}\right)\right),
\end{equation*}
and since $|\lambda|=O(A)$ the expression on the right-hand side is bounded away from zero for sufficiently large $A$. By \eqref{nudef} we also see that $e^{-|\nu|\sqrt{T'}/\eps}$ is bounded away from zero for any fixed $\eps>0$, and therefore
\begin{multline*}
\P\bigg(\sum_{i=1}^t \Delta''_i \ge t^\gamma \hspace{0.15cm}\forall t\in [T'],\, \sum_{i=1}^{T'} \Delta''_i \geq \eps\sqrt{T'}\bigg)\\
\ge c\Q\bigg(\sum_{i=1}^t \Delta''_i \ge t^\gamma \hspace{0.15cm}\forall t\in [T'],\, \sum_{i=1}^{T'} \Delta''_i \in \big[\eps\sqrt{T'},\sqrt{T'}/\eps\big]\bigg).
\end{multline*}
Since $(\Delta''_i)_{i=1}^{T'}$ under $\Q$ have the same distribution as the random variables $(D''_i)_{i=1}^{T'}$ under $\P$ from the statement of the lemma, this completes the proof.
\end{proof}

\subsubsection{Proof of \eqref{delprimeabovegamma}, step 4: bounding the probability that $\sum_{i=1}^t D''_i$ stays above $t^\gamma$ and finishes above $\eps\sqrt{T'}$.}

We have now reduced our problem to working with the straightforward random walk $\sum_{i=1}^t D''_i$, and can apply known results about random walks to gain our desired bound. The following lemma, when combined with Proposition \ref{deltatoDprop} and Lemma \ref{DeltatoDlem}, completes the proof of \eqref{delprimeabovegamma}.

\begin{lem}\label{Dlblem}
Let $(D''_i)_{i=0}^{T'}$ be a sequence of i.i.d.~random variables with $\P(D_i=d-2) = \frac{1}{d-1} = 1-\P(D_i=-1)$, as in Lemma \ref{DeltatoDlem}. Then for any $\gamma\in[0,1/2)$ there exist $\eps>0$ and $c>0$ depending only on $d$ such that
\[\P\bigg(\sum_{i=1}^t D''_i \ge t^\gamma \hspace{0.15cm}\forall t\in [T'],\, \sum_{i=1}^{T'} D''_i \in \left[\eps\sqrt{T'},\sqrt{T'}/\eps\right]\bigg) \ge \frac{cA}{n^{1/3}}.\]
\end{lem}

\begin{proof}
The probability that a mean-zero random walk with finite variance remains positive for $k$ steps is of order $1/\sqrt k$; see e.g.~Spitzer \cite{spitzer}. Thus, if we define the event
\[\mathcal P_k = \bigg\{\sum_{i=1}^t D''_i > 0 \hspace{0.15cm}\forall t\in [k]\bigg\},\]
we have
\begin{align*}
&\P\bigg(\sum_{i=1}^t D''_i \ge t^\gamma \hspace{0.15cm}\forall t\in [T'],\, \sum_{i=1}^{T'} D''_i \in \left[\eps\sqrt{T'},\sqrt{T'}/\eps\right]\bigg)\\
&\ge \P\bigg(\sum_{i=1}^t D''_i \ge t^\gamma \hspace{0.15cm}\forall t\in [T'],\, \sum_{i=1}^{T'} D''_i \in \left[\eps\sqrt{T'},\sqrt{T'}/\eps\right]\,\bigg|\, \mathcal P_{T'}\bigg)\P(\mathcal P_{T'})\\
&\ge \frac{c}{\sqrt{T'}}\Bigg(\P\bigg(\sum_{i=1}^t D''_i \ge t^\gamma \hspace{0.15cm}\forall t\in [T']\,\bigg|\,\mathcal P_{T'}\bigg) - \P\bigg(\sum_{i=1}^{T'} D''_i \not\in \left[\eps\sqrt{T'},\sqrt{T'}/\eps\right]\,\bigg|\, \mathcal P_{T'}\bigg)\Bigg)
\end{align*}
We now claim that for any $\gamma<1/2$, there exists $\alpha>0$ such that
\[\P\bigg(\sum_{i=1}^t D''_i \ge t^\gamma \hspace{0.15cm}\forall t\in [T']\,\bigg|\,\mathcal P_{T'}\bigg)\ge \alpha;\]
and for any $\alpha>0$, there exists $\eps>0$ such that
\[\P\bigg(\sum_{i=1}^{T'} D''_i \not\in \left[\eps\sqrt{T'},\sqrt{T'}/\eps\right]\,\bigg|\, \mathcal P_{T'}\bigg) < \alpha/2.\]
The first statement follows from a result of Ritter \cite{ritter:growth_RW_cond_positive}; see \cite[Lemma 8]{prigent_roberts} for details. The second statement follows from results of \cite{durrett}, specifically Theorem 3.10 (a general theorem on convergence of conditioned Markov processes) together with the results of Section 4.2 (where it is shown that finite variance, mean-zero random walks satisfy the conditions of the earlier theorem). These statements combine to complete the proof.
\end{proof}

Combining Proposition \ref{deltatoDprop}, Lemma \ref{DeltatoDlem} and Lemma \ref{Dlblem} tells us that
\[\mathbb{P}\bigg(\sum_{i=1}^t \delta''_i \ge t^\gamma \hspace{0.15cm}\forall t\in [T'],\, \sum_{i=1}^{T'} \delta''_i \geq 2\eps\sqrt{T'}\bigg) \ge \frac{cA}{n^{1/3}},\]
which is \eqref{delprimeabovegamma} without the conditioning on $\bS_n$; then applying Lemma \ref{removecondS} we deduce \eqref{delprimeabovegamma}.

\subsubsection{Proof of \eqref{activesumbelowgamma} : $h_i$ is not active too often}

Fix $\gamma\in(0,1/2)$. We first note that, since the number of active stubs can increase by at most $d$ at each step, $|\calA_i|\le di$ for each $i$. Since $h_i$ is chosen uniformly at random from the set of unexplored edges after step $i-1$, of which there are exactly $dn-2(i-1)-1$, a union bound gives
\begin{equation}\label{activegamma0}
\P\left(\exists i\in\big[\lceil n^{1/4}\rceil\wedge\tau\big] : h_i\in\calA_{i-1}\right) \le \sum_{i=1}^{\lceil n^{1/4}\rceil} \frac{d(i-1)}{dn-2(i-1)-1} \le \sum_{i=1}^{\lceil n^{1/4}\rceil} \frac{i}{2n/3} \le \frac{1}{n^{1/2}}.
\end{equation}
This gives us our desired bound up to $\lceil n^{1/4}\rceil$. For $t>n^{1/4}$, we observe that
\begin{multline}\label{activegamma1}
\P\bigg(\sum_{i=1}^{T'\wedge \tau} \ind_{\{h_i\in\calA_{i-1}\}} \ge n^{\gamma/4} \bigg)\\ \le \P\bigg(\sum_{i=1}^{T'\wedge \tau} \ind_{\{h_i\in\calA_{i-1}\}} \ge n^{\gamma/4},\, |\calA_t|\le n^{1/3 + \gamma/4}\,\,\forall t\in[T'\wedge\tau] \bigg)\\
+ \P\left(\exists t\in[T'\wedge\tau] : |\calA_t|> n^{1/3 + \gamma/4}\right).
\end{multline}
By Lemma \ref{lemmaforsecondrandomsumtoremove}, the last term above satisfies, for some finite constant $C$,
\begin{align}
&\P\left(\exists t\in[T'\wedge\tau] : |\calA_t|> n^{1/3 + \gamma/4}\right)\\
&\hspace{1cm}\le C\exp\left(-\frac{n^{2/3+\gamma/2}}{4p(d-1)T'} + \frac{n^{1/3+\gamma/4}}{2}\left(1-\frac{1}{p(d-1)}\right)\right)\nonumber\\
&\hspace{1cm}\le C\exp\left(-\frac{A^2 n^{\gamma/2}}{4(1+\lambda n^{-1/3})} + \frac{n^{1/3+\gamma/4}}{2}\left(1-\frac{1}{1+\frac{\lambda}{n^{1/3}}}\right)\right)\nonumber\\
&\hspace{1cm}= C\exp\left(-A^2 n^{\gamma/2}/4 + O\left(A^2\lambda n^{\gamma/2-1/3} + \lambda n^{\gamma/4}\right)\right)\nonumber\\
&\hspace{1cm}\le c\exp\left(- A^2 n^{\gamma/2}/8\right).\label{activegamma2}
\end{align}
On the other hand, the first term on the right-hand side of \eqref{activegamma1} satisfies
\begin{multline}\label{activegamma3}
\P\bigg(\sum_{i=1}^{T'\wedge \tau} \ind_{\{h_i\in\calA_{i-1}\}} \ge n^{\gamma/4},\, |\calA_t|\le n^{1/3 + \gamma/4}\,\,\forall t\in[T'\wedge\tau] \bigg)\\
\le \E\left[ e^{\sum_{i=1}^{T'\wedge \tau} \ind_{\{h_i\in\calA_{i-1}\}}} \ind_{\{|\calA_t|\le n^{1/3 + \gamma/4}\,\,\forall t\in[T'\wedge\tau]\}}\right] e^{-n^{\gamma/4}}.
\end{multline}
For $i\in[T'\wedge\tau]$, on the event $\{|\calA_{i-1}|\le n^{1/3 + \gamma/4}\}$ we have
\[\P(h_i\in\calA_{i-1}\,|\,\F_{i-1}) \le \frac{n^{1/3 + \gamma/4}}{dn-2(i-1)-1} \le n^{-2/3+\gamma/4}\]
and therefore
\begin{align*}
&\E\left[ e^{\sum_{i=1}^{T'\wedge \tau} \ind_{\{h_i\in\calA_{i-1}\}}} \ind_{\{|\calA_t|\le n^{1/3 + \gamma/4}\,\,\forall t\in[T'\wedge\tau]\}}\right] \\
&\le \E\left[ e^{\sum_{i=1}^{T'\wedge \tau - 1} \ind_{\{h_i\in\calA_{i-1}\}}} \ind_{\{|\calA_t|\le n^{1/3 + \gamma/4}\,\,\forall t\in[T'\wedge\tau-1]\}} \E\left[\left. e^{\ind_{\{h_{T'\wedge\tau}\in\calA_{i-1}\}}} \,\right|\,\F_{T'\wedge\tau}\right]\right] \\
&\le \big(1 + (e-1)n^{-2/3+\gamma/4}\big)\E\left[ e^{\sum_{i=1}^{T'\wedge \tau - 1} \ind_{\{h_i\in\calA_{i-1}\}}} \ind_{\{|\calA_t|\le n^{1/3 + \gamma/4}\,\,\forall t\in[T'\wedge\tau-1]\}}\right]\\
&\le \ldots \le \big(1 + (e-1)n^{-2/3+\gamma/4}\big)^{T'}.
\end{align*}
Using the inequality $1+x\le e^x$ valid for all $x\in\R$, we deduce that 
\[\eqref{activegamma3} \le \exp\left((e-1)n^{-2/3+\gamma/4}T' - n^{\gamma/4}\right) \le c\exp\left(-n^{\gamma/4}/2\right).\]
Substituting this and \eqref{activegamma2} into \eqref{activegamma1}, we have shown that
\[\P\bigg(\sum_{i=1}^{T'\wedge \tau} \ind_{\{h_i\in\calA_{i-1}\}} \ge n^{\gamma/4} \bigg) \le ce^{-A^2 n^{\gamma/2}/8} + ce^{-n^{\gamma/4}/2} \le c e^{-n^{\gamma/4}/2}.\]
Finally, combining this with \eqref{activegamma0}, we obtain
\[\mathbb{P}\bigg(\exists t\in[T'\wedge\tau] : \sum_{i=1}^t \ind_{\{h_i\in \mathcal{A}_{i-1}\}} \ge t^\gamma\bigg) \le \frac{1}{n^{1/2}} + c e^{-n^{\gamma/4}/2}.\]
We also have
\begin{multline*}
\mathbb{P}\bigg(\exists t\in[T'\wedge\tau] : \sum_{i=1}^t \ind_{\{h_i\in \mathcal{A}_{i-1}\}} \ge t^\gamma\,\bigg|\,\bS_n\bigg)\\
 \le \frac{1}{\P(\bS_n)}\mathbb{P}\bigg(\exists t\in[T'\wedge\tau] : \sum_{i=1}^t \ind_{\{h_i\in \mathcal{A}_{i-1}\}} \ge t^\gamma \bigg)
\end{multline*}
and since $\P(\bS_n)\to \exp((1-d^2)/4) > 0$ we deduce \eqref{activesumbelowgamma}.

\subsubsection{Combining \eqref{delprimeabovegamma} and \eqref{activesumbelowgamma} to deduce \eqref{kklk} and prove Proposition \ref{propfirstbit}}

We note that since $\eta_i\ge \delta'_i$ for all $i\le\tau$, if $\sum_{i=1}^t \delta'_i > -d$ for all $t\in[T']$, then $\tau\ge T'$. Thus
\begin{align*}
&\mathbb{P}\bigg(\sum_{i=1}^t \delta'_i > -d \hspace{0.15cm}\forall t\in [T'],\, \sum_{i=1}^{T'} \delta'_i \geq \eps \sqrt{T'}\,\bigg|\,\bS_n\bigg)\\
&= \P\bigg(\sum_{i=1}^t \delta'_i > -d \hspace{0.15cm}\forall t\in [T'\wedge\tau],\, \sum_{i=1}^{T'} \delta'_i \geq \eps \sqrt{T'}\,\bigg|\,\bS_n\bigg)\\
&\ge \P\bigg(\sum_{i=1}^t \delta''_i > t^\gamma \hspace{0.15cm}\forall t\in [T'\wedge\tau],\, \sum_{i=1}^{T'} \delta''_i \geq 2\eps \sqrt{T'},\, \sum_{i=1}^t \ind_{\{h_i\in\calA_{i-1}\}}<t^\gamma\,\,\forall t\in[T'\wedge\tau]\,\bigg|\,\bS_n\bigg)\\
&\ge \P\bigg(\sum_{i=1}^t \delta''_i > t^\gamma \hspace{0.15cm}\forall t\in [T'\wedge\tau],\, \sum_{i=1}^{T'} \delta''_i \geq 2\eps \sqrt{T'}\,\bigg|\,\bS_n\bigg)\\
&\hspace{60mm} - \P\bigg(\exists t\in[T'\wedge\tau] : \sum_{i=1}^t \ind_{\{h_i\in\calA_{i-1}\}}\ge t^\gamma\,\bigg|\,\bS_n\bigg)\\
&\ge \eqref{delprimeabovegamma} - \eqref{activesumbelowgamma} \ge \frac{cA}{n^{1/3}}.
\end{align*}
This establishes \eqref{kklk} and therefore completes the proof of Proposition \ref{propfirstbit}.

\subsection{The probability of staying above a curve: proof of Proposition \ref{propsecondbit}}\label{propsecondbit_sec}
Here we want to bound from below the probability
\begin{equation}\label{wwf}
\mathbb{P}\bigg(\sum_{i=T'+1}^t D_i >q(t)-\eps\sqrt{T'}\hspace{0.15cm}\forall t\in [T]\setminus[T']\bigg)
\end{equation}
where we recall that $q(t)=p(1-2/d)\frac{t^2}{2n} + An^{4/15}$ and $D_i=\ind_{R_i}(d-1)-1$. In order to bound the probability (\ref{wwf}), the idea is to approximate the random walk $(\sum_{i=T'+1}^t D_i)_{t=T'}^T$ with (standard) Brownian motion and then to carry on the analysis using estimates for Brownian motion. 

As a first step in this direction, let us rewrite (\ref{wwf}) in a way that helps simplifying the calculations to come. Since $A\ll n^{1/30}$ we see that $An^{4/15}\ll (T')^{1/2} \sim n^{1/3}/A$. Thus, if $n$ is large enough, we can bound
\begin{align}
&\mathbb{P}\bigg(\sum_{i=T'+1}^t D_i >q(t)-\eps\sqrt{T'}\hspace{0.15cm}\forall t\in [T]\setminus[T']\bigg)\nonumber\\
&\hspace{30mm}\geq \mathbb{P}\bigg(\sum_{i=T'+1}^t D_i > p(1-2/d)\frac{t^2}{2n} - \frac{3\eps n^{1/3}}{4A}\hspace{0.15cm}\forall t\in [T]\setminus[T']\bigg)\nonumber\\
&\hspace{30mm}= \mathbb{P}\bigg(\sum_{i=1}^t D_i > p(1-2/d)\frac{(t+T')^2}{2n} - \frac{3\eps n^{1/3}}{4A}\hspace{0.15cm}\forall t\in [T-T']\bigg),\nonumber
\end{align}
where the last equality follows from the fact that $(D_i)_{i\ge 1}$ are i.i.d.. We also note that for any $t\in[T-T']$ and sufficiently large $n$,
\begin{align*}
p(1-2/d)\frac{(t+T')^2}{2n} &\le \frac{1}{(d-1)}(1-2/d)\left(\frac{t^2 + 2tT'}{2n}\right) + \frac{\lambda n^{-1/3} T^2}{2(d-1)n} + \frac{(T')^2}{2n}\\
&\le \frac{(d-2)}{d(d-1)}\left(\frac{t^2 + 2tT'}{2n}\right) + \frac{\eps n^{1/3}}{4A}.
\end{align*}
We therefore have
\begin{multline}\label{wwfsimp}
\mathbb{P}\bigg(\sum_{i=T'+1}^t D_i >q(t)-\eps\sqrt{T'}\hspace{0.15cm}\forall t\in [T]\setminus[T']\bigg)\\
\ge \mathbb{P}\bigg(\sum_{i=1}^t D_i > \frac{(d-2)}{d(d-1)}\left(\frac{t^2 + 2tT'}{2n}\right) - \frac{\eps n^{1/3}}{2A}\hspace{0.15cm}\forall t\in [T-T']\bigg).
\end{multline}

Note that the increments $D_i$ satisfy
\begin{align*}
\mathbb{P}\left(D_i=d-2\right)= p = \frac{1+\lambda n^{-1/3}}{d-1} = 1-\P\left(D_i=-1\right).
\end{align*}
As we said earlier, we would like to approximate $\sum_{i=1}^t D_i$ with Brownian motion, following the strategy in \cite{de_ambroggio_roberts:near_critical_ER}. However, we first need to turn the $D_i$ into random variables whose distribution does not depend on $n$, and moreover we want $D_i$ to have mean zero and unit variance. This is accomplished in the following lemma, whose proof is postponed to Section \ref{aux_sec}. Note that this lemma is very similar to Lemma \ref{DeltatoDlem}, but because we now run our random walk for a longer time $T-T'$ rather than $T'$, we see a factor relating to $\lambda$ appear. With \eqref{wwfsimp} in mind, we define
\[f(t) = f_{n,A,\eps,d}(t) = \frac{(d-2)}{d(d-1)}\left(\frac{t^2 + 2tT'}{2n}\right) - \frac{\eps n^{1/3}}{2A}\]
and let $T'' = T-T' = \lfloor(d-1) An^{2/3}\rfloor + 1 - \lfloor n^{2/3}/A^2\rfloor$.

\begin{lem}\label{secondchangeofmeasure}
	Let $(D'_i)_{i=0}^{T''}$ be a sequence of i.i.d.~random variables with $\P(D'_i=\sqrt{d-2}) = \frac{1}{d-1} = 1-\P(D'_i=-1/\sqrt{d-2})$. There exists a constant $c=c(d)>0$ depending on $d$ such that, for any $\ell>0$ and any large enough $n\in\N$,
	\begin{align*}
	&\mathbb{P}\bigg(\sum_{i=1}^t D_i > f(t)\hspace{0.15cm}\forall t\in [T'']\bigg)\\
	&\ge c e^{\frac{\lambda A^2(d-1)}{2d}-\frac{\lambda^2 A(d-1)}{2(d-2)} - 2|\lambda| n^{-1/3}\ell} \mathbb{P}\bigg(\sum_{i=1}^t D'_i > \frac{f(t)}{\sqrt{d-2}}\hspace{0.15cm}\forall t\in [T''],\, \sum_{i=1}^{T''} D'_i \le \frac{f(T'')+\ell}{\sqrt{d-2}}\bigg).
	\end{align*}
\end{lem}

We will prove this result in Section \ref{aux_sec}. For now we proceed with the proof of Proposition \ref{propsecondbit}, noting that $\E[D'_i]=0$ and $\text{Var}(D'_i)=1$. We will use the following powerful approximation result. The precise phrasing used below is taken from Chatterjee \cite{chatterjee:strong_embeddings}.

\begin{thm}[Koml\'os, Major, Tusn\'ady \cite{komlos_major_tusnady:approximation_partial_sums}]\label{embedding}
	Let $(\xi_i)_{i\geq 1}$ be a sequence of i.i.d. random variables with $\mathbb{E}[\xi_1]=0$ and $\mathbb{E}[\xi_1^{2}]=1$. Suppose that there exists $\theta>0$ such that $\mathbb{E}\left[e^{\theta |\xi_1|}\right]<\infty$. Then for every $N\in \mathbb{N}$ it is possible to construct a version of $(\xi_i)_{i=0}^N$ and a standard Brownian Motion $(B_s)_{s\in[0,N]}$ on the same probability space such that, for every $x\geq 0$,
	\begin{equation*}
	\mathbb{P}\left(\max_{k\leq N}\left|\sum_{i=1}^k \xi_i-B_k\right|>H\log N + x\right)\leq a e^{-b x}
	\end{equation*}
	where $H$, $a$ and $b>0$ do not depend on $N$.
\end{thm}

We next use Theorem \ref{embedding} with $\xi_i = D'_i$ to bound from below the probability appearing in Lemma \ref{secondchangeofmeasure}. Taking $\ell = n^{1/3}\sqrt{d-2}/A$ and $x=\frac{\eps n^{1/3}}{8A\sqrt{d-2}}$, and noting then that for large $n$ we have
\[H\log T'' + x \le \frac{\eps n^{1/3}}{4A\sqrt{d-2}} \,\,\,\,\text{ and }\,\,\,\, \frac{\ell}{\sqrt{d-2}} - H\log T'' - x \ge \frac{n^{1/3}}{2A},\]
we deduce that
\begin{align}
&\mathbb{P}\bigg(\sum_{i=1}^t D'_i > \frac{f(t)}{\sqrt{d-2}}\hspace{0.15cm}\forall t\in [T''],\, \sum_{i=1}^{T''} D'_i \le \frac{f(T'')+\ell}{\sqrt{d-2}}\bigg)\nonumber\\
&\ge \P\bigg(B_t > \frac{f(t)}{\sqrt{d-2}} + H\log T'' + x\hspace{0.15cm}\forall t\in [T''],\, B_{T''} \le \frac{f(T'')+\ell}{\sqrt{d-2}} - H\log T'' - x\bigg) - ae^{-bx}\nonumber\\
&\ge \P\bigg(B_t > \frac{f(t)}{\sqrt{d-2}} + \frac{\eps n^{1/3}}{4A\sqrt{d-2}} \hspace{0.15cm}\forall t\in [T''],\, B_{T''} \le \frac{f(T'')}{\sqrt{d-2}} + \frac{n^{1/3}}{2A}\bigg) - ae^{-b\eps n^{1/3}/8A}.\label{applyKMT}
\end{align}

Our final lemma in this section bounds the probability on the right-hand side above.

\begin{lem}\label{mainlemmasecondbit}
	For any $\eps\in(0,1)$ there exists a finite constant $c=c(d,\eps)>0$ depending on $d$ and $\eps$ such that, if $n$ is large enough,
	\begin{align*}
	\P\bigg(B_t > \frac{f(t)}{\sqrt{d-2}} + \frac{\eps n^{1/3}}{4A\sqrt{d-2}} \hspace{0.15cm}\forall t\in [T''],\, B_{T''} \le \frac{f(T'')}{\sqrt{d-2}} + \frac{n^{1/3}}{2A}\bigg) \geq \frac{c}{A^{3/2}}e^{-\frac{A^3(d-1)(d-2)}{8d^2}}.
	\end{align*}
\end{lem}

We will prove this in Section \ref{aux_sec}; for now we use it to finish our main proof for this section.

\begin{proof}[Proof of Proposition \ref{propsecondbit}]
Putting the steps above together, we combine \eqref{wwfsimp} with Lemma \ref{secondchangeofmeasure} (again with $\ell = n^{1/3}\sqrt{d-2}/A$) and apply \eqref{applyKMT} to obtain
\begin{align*}
&\mathbb{P}\bigg(\sum_{i=T'+1}^t D_i >q(t)-\eps\sqrt{T'}\hspace{0.15cm}\forall t\in [T]\setminus[T']\bigg)\\
&\hspace{7mm}\ge c e^{\frac{\lambda A^2(d-1)}{2d}-\frac{\lambda^2 A(d-1)}{2(d-2)}}\\
&\hspace{15mm}\cdot\P\bigg(B_t > \frac{f(t)}{\sqrt{d-2}} + \frac{\eps n^{1/3}}{4A} \hspace{0.15cm}\forall t\in [T''],\, B_{T''} \le \frac{f(T'')}{\sqrt{d-2}} + \frac{n^{1/3}}{2A}\bigg) - ae^{-b\eps n^{1/3}/8A}.
\end{align*}
Lemma \ref{mainlemmasecondbit} then says that this is at least
\[\frac{c}{A^{3/2}} e^{-\frac{A^3(d-1)(d-2)}{8d^2} + \frac{\lambda A^2(d-1)}{2d}-\frac{\lambda^2 A(d-1)}{2(d-2)}},\]
completing the proof of Proposition \ref{propsecondbit}.
\end{proof}

\subsection{Auxiliary results for Proposition \ref{propsecondbit}: proofs of Lemmas \ref{secondchangeofmeasure} and \ref{mainlemmasecondbit}}\label{aux_sec}

We begin with Lemma \ref{secondchangeofmeasure}, which requires us to bound from below the probability
\[\mathbb{P}\bigg(\sum_{i=1}^t D_i > f(t)\hspace{0.15cm}\forall t\in [T'']\bigg),\]
in terms of i.i.d.~random variables $(D'_i)_{i=1}^{T''}$ satisfying $\P(D'_i=\sqrt{d-2}) = \frac{1}{d-1} = 1-\P(D'_i=-1/\sqrt{d-2})$. We recall that $D_i = (d-1)\ind_{R_i} - 1$. The definition of the function $f$ is unimportant for this lemma.

\begin{proof}[Proof of Lemma \ref{secondchangeofmeasure}]
	We note that if $\lambda=0$ then there is nothing to prove; we may simply let $D'_i = D_i/(d-2)^{1/2}$. If $\lambda\neq 0$, however, then $D_i$ has a small drift, which we use a change of measure to remove. Let
	\[\gamma = \frac{1}{d-1}\log\left(\frac{1-p}{p(d-2)}\right)\]
	and define a new probability measure $\widehat{\mathbb{P}}$, with expectation operator $\widehat\E$, through
	\begin{align*}
	\frac{\d\widehat{\mathbb{P}}}{\d\mathbb{P}}\bigg|_{\F_{T''}} \coloneqq \frac{e^{\gamma \sum_{i=1}^{T''}D_i}}{\mathbb{E}\left[e^{\gamma \sum_{i=1}^{T''}D_i}\right]}.
	\end{align*}
	One may check that, by our choice of $\gamma$, under $\widehat P$ the sequence $(D_i)_{i=1}^{T''}$ is i.i.d.~with
	\[\widehat\P(D_i=d-2) = \frac{1}{d-1} = 1-\widehat\P(D_i=-1).\]
	In particular we have
	\begin{align*}
	\widehat\E[D_1]=0 \,\,\,\,\text{ and }\,\,\,\,\widehat\E[D_1^2] = d-2.
	\end{align*}

	Now, for any $\ell>0$ we have
	\begin{align}
	\nonumber &\mathbb{P}\bigg(\sum_{i=1}^t D_i > f(t)\hspace{0.15cm}\forall t\in [T'']\bigg)\\
\nonumber	&\ge \mathbb{P}\bigg(\sum_{i=1}^t D_i > f(t)\hspace{0.15cm}\forall t\in [T''],\, \sum_{i=1}^{T''}D_i \leq f(T'')+\ell\bigg)\\
\nonumber	&=\widehat{\mathbb{E}}\left[e^{-\gamma \sum_{i=1}^{T''}D_i}\ind_{\left\{\sum_{i=1}^t D_i > f(t)\hspace{0.15cm}\forall t\in [T''],\, \sum_{i=1}^{T''}D_i \leq f(T'')+\ell\right\}}\right]\mathbb{E}\left[e^{\gamma \sum_{i=1}^{T''}\xi_i}\right]\\
	& \geq  e^{-\gamma f(T'') - |\gamma|\ell}\mathbb{E}\left[e^{\gamma \sum_{i=1}^{T''}D_i}\right]\widehat{\mathbb{P}}\bigg(\sum_{i=1}^t D_i > f(t)\hspace{0.15cm}\forall t\in [T''],\, \sum_{i=1}^{T''}D_i \leq f(T'')+\ell\bigg).\label{probwithterm}
	\end{align}
	Some elementary computations reveal that
	\begin{equation}\label{gammadef}
	\gamma = -\frac{\lambda n^{-1/3}}{d-2} + \frac{(d-3)\lambda^2 n^{-2/3}}{2(d-2)^2} + O(\lambda^3 n^{-1})
	\end{equation}
	and
	\[f(T'') = \frac{(d-1)(d-2)}{2d}A^2 n^{1/3} + O(n^{1/3}/A).\]
	Thus
	\begin{equation}\label{gammaf}
	\gamma f(T'') = -\frac{(d-1)}{2d}A^2\lambda + O(1).
	\end{equation}
	Further simple algebra shows that
	\[\E[e^{\gamma D_1}] = e^{-\gamma}\left(1-\frac{\lambda n^{-1/3}}{d-2}\right),\]
	and since $1-x = \exp(-x-x^2/2+O(x^3))$, using \eqref{gammadef} we have
	\begin{align*}
	\E[e^{\gamma D_1}] &= \exp\left(-\gamma - \frac{\lambda n^{-1/3}}{d-2} - \frac{\lambda^2 n^{-2/3}}{2(d-2)^2} + O(\lambda^3 n^{-1})\right)\\
	&= \exp\left(-\frac{\lambda^2 n^{-2/3}}{2(d-2)} + O(\lambda^3 n^{-1})\right).
	\end{align*}
	Thus, using the fact that $|\lambda| = O(A)$,
	\begin{align*}
	\mathbb{E}\left[e^{\gamma \sum_{i=1}^{T''}D_i}\right] &= \E\left[e^{\gamma D_1}\right]^{T''} = \exp\left(-\frac{\lambda^2 n^{-2/3} T''}{2(d-2)} + O(\lambda^3 n^{-1}T'')\right)\\
	&= \exp\left(-\frac{(d-1)\lambda^2 A}{2(d-2)} + O(1)\right).
	\end{align*}
	Substituting this and \eqref{gammaf} into \eqref{probwithterm} gives that
	\begin{multline*}
	\mathbb{P}\bigg(\sum_{i=1}^t D_i > f(t)\hspace{0.15cm}\forall t\in [T'']\bigg)\\
	\ge c e^{\frac{(d-1)}{2d}A^2\lambda - \frac{(d-1)\lambda^2 A}{2(d-2)} - |\gamma|\ell} \widehat{\mathbb{P}}\bigg(\sum_{i=1}^t D_i > f(t)\hspace{0.15cm}\forall t\in [T''],\, \sum_{i=1}^{T''}D_i \leq f(T'')+\ell\bigg).
	\end{multline*}
	Note also that $|\gamma|\le 2|\lambda| n^{-1/3}$ when $n$ is large. Finally we observe that $\big(\frac{D_i}{\sqrt{d-2}}\big)_{i=1}^{T''}$ under $\widehat{\P}$ has the same distribution as $(D'_i)_{i=1}^{T''}$ under $\P$ from the statement of the lemma. This completes the proof.
\end{proof}

	We now begin our preparations for the proof of Lemma \ref{mainlemmasecondbit}. Recall that
	\[f(t) = \frac{(d-2)}{d(d-1)}\left(\frac{t^2 + 2tT'}{2n}\right) - \frac{\eps n^{1/3}}{2A}\]
	and to reduce the notation in what follows, let
	\[\phi(t) = \frac{f(t)}{\sqrt{d-2}} + \frac{\eps n^{1/3}}{4A\sqrt{d-2}} = \frac{\sqrt{d-2}}{d(d-1)}\left(\frac{t^2 + 2tT'}{2n}\right) - \frac{\eps n^{1/3}}{4A\sqrt{d-2}}.\]
	We can then bound the probability in the statement of the lemma as follows.
	\begin{align}
	P &\coloneqq \P\bigg(B_t > \frac{f(t)}{\sqrt{d-2}} + \frac{\eps n^{1/3}}{4A\sqrt{d-2}} \hspace{0.15cm}\forall t\in [T''],\, B_{T''} \le \frac{f(T'')}{\sqrt{d-2}} + \frac{n^{1/3}}{2A}\bigg)\nonumber\\
	&\ge \P\left(B_t > \phi(t) \,\, \forall t\in[T''],\, B_{T''} \le \phi(T'') + n^{1/3}/4A\right)\nonumber\\
	&\ge \P\left(B_s > \phi(s) \,\, \forall s\in[0,T''],\, B_{T''} \le \phi(T'') + n^{1/3}/4A \right)\label{Pdefmainlem},
	\end{align}
	where we note that in the last line we have moved from a discrete set of times $t\in[T'']$ to a continuous interval $s\in[0,T'']$.
	
	Following closely the argument developed in \cite{de_ambroggio_roberts:near_critical_ER}, we approximate the curve $\phi(s)$ with two straight lines defined, for $s\in [0,T''/2]$, by
	\begin{align*}
	\ell_1(s) &= \phi(0) + \Big(\frac{\phi(T''/2)-\phi(0)}{T''/2}\Big)s\\
	&=-\frac{\eps n^{1/3}}{4A\sqrt{d-2}}+ \frac{\sqrt{d-2}}{d(d-1)}\left(\frac{T+3T'}{4n}\right)s
	\end{align*}
	and
	\begin{align*}
	\ell_2(s) &= \phi(T''/2) + \Big(\frac{\phi(T'')-\phi(T''/2)}{T''/2}\Big)s \\
	&= \frac{\sqrt{d-2}}{d(d-1)}\frac{(T+3T')(T-T')}{8n} - \frac{\eps n^{1/3}}{4A\sqrt{d-2}} + \frac{\sqrt{d-2}}{d(d-1)}\left(\frac{3T''+4T'}{4n}\right)s.
	\end{align*}
	Also define
	\[I_1 = \left[\frac{\phi(T'')}{2} + \frac{n^{1/3}}{8A} - A^{1/2}n^{1/3},\, \frac{\phi(T'')}{2} + \frac{n^{1/3}}{8A}\right]\]
	and
	\[I_2 = \left[\phi(T'') + \frac{n^{1/3}}{8A},\, \phi(T'') + \frac{n^{1/3}}{4A}\right].\]
	Note that for all large enough $n$, the intervals $I_1$ and $I_2$ both fall entirely above the curve $\phi(s)$. (See Figure \ref{graphfig} for reference.)

\begin{figure}[h]
	\def\svgwidth{160mm}
 		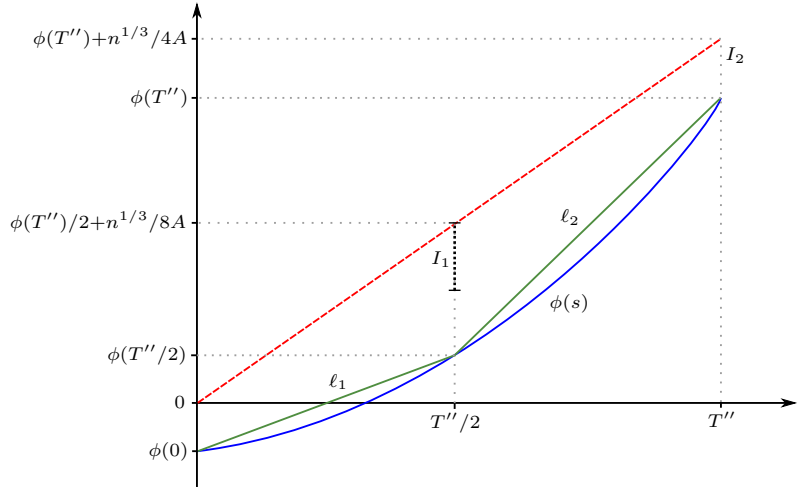
 		\vspace{-7mm}
 		\caption{We want our Brownian motion to stay above the blue curve, and the two green lines $\ell_1$ and $\ell_2$ show linear approximations to this curve on the two half-intervals $[0,T''/2]$ and $[T''/2,T'']$. The dashed red line shows roughly where we expect our Brownian motion to be, given that it stays above the curve. This is a caricature of the true picture, and not to scale.}
\label{graphfig}
 	\end{figure}
	
	Since $\phi$ is convex, the straight lines $\ell_1$ and $\ell_2$ fall above the curve and therefore we can bound
	\begin{align*}
	&\P\left(B_s > \phi(s) \,\, \forall s\in[0,T''],\, B_{T''} \le \phi(T'') + n^{1/3}/4A\right)\\
	&\hspace{0.2cm}\geq \P\left(B_s > \ell_1(s)\;\; \forall s\in\big[0,\tfrac{T''}{2}\big],\, B_s > \ell_2\big(s-\tfrac{T''}{2}\big) \;\;\forall s\in\big[\tfrac{T''}{2},T''\big],\, B_{T''} \le \ell_2(T'') + \tfrac{n^{1/3}}{4A}\right).
	\end{align*}
	Since we are proving a lower bound, we may also insist that at times $T''/2$ and $T''$ our Brownian motion falls within the intervals $I_1$ and $I_2$ respectively; putting this together with \eqref{Pdefmainlem}, we obtain that
	\begin{multline}\label{BrownianStep2}
	P \ge \int_{I_1} \P_0\left(B_s > \ell_1(s)\;\; \forall s\in\big[0,\tfrac{T''}{2}\big],\, B_{T''/2} \in \d w\right)\\
	\cdot \P_w\left(B_s > \ell_2(s)\;\; \forall s\in\big[0,\tfrac{T''}{2}\big],\, B_{T''/2} \in I_2\right),
	\end{multline}
	where here $\P_w$ denotes a probability measure under which our Brownian motion starts from $w$ rather than $0$.
	
	To complete our proof we need to bound the probabilities that appear within the last integral. Our next lemma, whose proof simply involves applying a Girsanov transform to remove the drift and then applying the reflection principle, gives us a general formula that we will then apply to gain the desired bounds. For the details of the proof we refer to \cite{de_ambroggio_roberts:near_critical_ER}.
	
	\begin{lem}[{\cite[Lemma 4.12]{de_ambroggio_roberts:near_critical_ER}}]\label{reflectionlem}
		For any $\mu,y\in\mathbb{R}$, $t>0$, $x>y$ and $z>y+\mu t$,
		\[\P_x(B_s > y + \mu s \;\;\forall s\le t,\, B_t \in \d z) = \frac{1}{\sqrt{2\pi t}} \exp\Big(-\frac{(z-x)^2}{2t}\Big)\big(1-e^{2(x-y)(\mu t+y-z)/t}\big)\d z.\]
	\end{lem}
	We now use Lemma \ref{reflectionlem} to obtain a lower bound for the probability that our Brownian motion stays above the line $l_1(s)$ and finishes near $w\in I_1$ at time $T''/2$, i.e.~for the first probability in \eqref{BrownianStep2}.
	
	\begin{cor}\label{BrownianCor1}
		There exists a constant $c=c(d,\eps)>0$ such that for any $w\in I_1$,
		\[\P_0\big(B_s > \ell_1(s)\;\; \forall s\in[0,T''/2],\, B_{T''/2} \in \d w\big) \ge \frac{c}{\sqrt {T''}} e^{-w^2/T''}\d w.\]
	\end{cor}
	
	\begin{proof}
		We apply Lemma \ref{reflectionlem} with $x=0$, $z=w\in I_1$, $t=T''/2$,
		\begin{equation*}
		y=\phi(0) = -\frac{\eps n^{1/3}}{4A\sqrt{d-2}} \,\,\,\,\text{ and }\,\,\,\,\mu = \frac{\phi(T''/2)-\phi(0)}{T''/2} = \frac{\sqrt{d-2}}{d(d-1)}\left(\frac{T+3T'}{4n}\right).
		\end{equation*}
		Note that, with these parameters,
		\begin{align*}
		\mu t + y - z &\le \big(\phi(T''/2) - \phi(0)\big) + \phi(0) - \big(\phi(T'')/2 + n^{1/3}/8A - A^{1/2}n^{1/3}\big)\\
		&\sim \phi(T''/2) - \phi(T'')/2\\
		&\sim -\frac{\sqrt{d-2}(d-1)}{d}\frac{A^2 n^{1/3}}{8} < 0.
		\end{align*}
		We may therefore apply Lemma \ref{reflectionlem}, and noting that at most
		\begin{align*}
		2(x-y)(\mu t+y-z)/t &\sim 2\left(\frac{\eps n^{1/3}}{4A\sqrt{d-2}}\right)\left(-\frac{\sqrt{d-2}(d-1)}{d}\frac{A^2 n^{1/3}}{8}\right)\left(\frac{2}{(d-1)An^{2/3}}\right)\\
		&= - \eps/8d
		\end{align*}
		the last factor in Lemma \ref{reflectionlem} reduces to a positive constant and the result follows.
	\end{proof}
	
	Next we bound from below the second probability that appears in the integral (\ref{BrownianStep2}), again by means of Lemma \ref{reflectionlem}.
	
	\begin{cor}\label{BrownianCor2}
		There exists a constant $c=c(d,\eps)>0$ such that for any $w\in I_1$ and $n$ and $A$ sufficiently large,
		\[\P_w\big(B_s > \ell_2(s)\;\; \forall s\in[0,T''/2],\, B_{T''/2} \in I_2\big) \ge \frac{c}{\sqrt{T''}} \int_{I_2} e^{-(z-w)^2/T''} \d z.\]
	\end{cor}	
	
	\begin{proof}
		We now apply Lemma \ref{reflectionlem} with $x=w\in I_1$, $y=\ell_2(0) = \phi(T''/2)$, $t=T''/2$ and
		\[\mu = \frac{\phi(T'')-\phi(T''/2)}{T''/2}.\]
		We also have $z\in I_2 = [\phi(T'')+ n^{1/3}/8A,\, \phi(T'')+n^{1/3}/4A]$; thus
		\[\mu t + y - z \le \big(\phi(T'') - \phi(T''/2)\big) + \phi(T''/2) - \big(\phi(T'')+n^{1/3}/8A) = -n^{1/3}/8A.\]
		Simple estimates show that
		\[x-y \sim \frac{\sqrt{d-2}}{d(d-1)}\frac{T^2}{4} - \frac{\sqrt{d-2}}{d(d-1)}\frac{T^2}{8} \sim \frac{\sqrt{d-2}(d-1)}{d}\frac{A^2 n^{1/3}}{8},\]
		and therefore
		\[(x-y)(\mu t + y - z)/t \le -\frac{\sqrt{d-2}}{32d} + o(1).\]
		We then deduce the result from Lemma \ref{reflectionlem}.
%
	\end{proof}

	The two corollaries above, combined with \eqref{BrownianStep2}, are the ingredients needed to complete our proof of Lemma \ref{mainlemmasecondbit}.
	
	\begin{proof}[Proof of Lemma \ref{mainlemmasecondbit}]
	Substituting Corollaries \ref{BrownianCor1} and \ref{BrownianCor2} into \eqref{BrownianStep2}, we obtain that
	\[P \ge \frac{c}{T''} \int_{I_2} \int_{I_1}  e^{-w^2/T'' - (z-w)^2/T''} \d w \, \d z.\]
	Let $\Phi = \phi(T'')+n^{1/3}/4A$, so that
	\[I_1 = [\Phi/2 - A^{1/2}n^{1/3},\,\Phi/2] \,\,\,\,\text{ and }\,\,\,\, I_2 = [\Phi-n^{1/3}/8A, \Phi].\]
	Using the substitutions $u = w - \Phi/2$ and $v = z - \Phi$, we have
	\[P \ge \frac{c}{T''} \int_{-n^{1/3}/8A}^{0} \int_{-A^{1/2}n^{1/3}}^0  e^{-(u+\Phi/2)^2/T'' - (v-u+\Phi/2)^2/T''} \d u \, \d v\]
	which, after multiplying out the quadratic terms in the exponent, becomes
	\[P \ge \frac{c}{T''} \int_{-n^{1/3}/8A}^{0} \int_{-A^{1/2}n^{1/3}}^0  e^{-2u^2/T'' + 2uv/T'' - \Phi^2/2T'' - v^2/T'' - v\Phi/T''} \d u \, \d v.\]
	Since $u,v\le 0$, we have $2uv/T''\ge 0$ and, removing this term, we may otherwise separate the two integrals, giving
	\[P \ge \frac{c}{T''} e^{-\Phi^2/(2T'')} \left(\int_{-A^{1/2}n^{1/3}}^0  e^{-2u^2/T''} \d u\right)\left(\int_{-n^{1/3}/8A}^{0} e^{- v^2/T'' - v\Phi/T''} \d v\right).\]
	Since $A^{1/2}n^{1/3} = O(\sqrt{T''})$, we have
	\[\int_{-A^{1/2}n^{1/3}}^0  e^{-2u^2/T''} \d u \ge c A^{1/2}n^{1/3} \ge c\sqrt{ T''},\]
	and since $n^{1/3}/8A = o(\sqrt{T''})$, we have
	\[\int_{-n^{1/3}/8A}^{0} e^{- v^2/T'' - v\Phi/T''} \d v \sim \int_{-n^{1/3}/8A}^{0} e^{- v\Phi/T''} \d v = \frac{T''}{\Phi}\left(\exp\left(\frac{n^{1/3}\Phi}{8A T''}\right)-1\right).\]
	We deduce that
	\begin{equation}\label{BrownianStep3}
	P \ge \frac{c\sqrt{T''}}{\Phi}  e^{-\Phi^2/(2T'')} \left(\exp\left(\frac{n^{1/3}\Phi}{8A T''}\right)-1\right).
	\end{equation}
	Finally, we note that
	\begin{align*}
	\Phi = \phi(T'') + \frac{n^{1/3}}{4A} &=  \frac{\sqrt{d-2}}{d(d-1)}\left(\frac{T^2 - (T')^2}{2n}\right) + \frac{n^{1/3}}{4A}\\
	&= \frac{\sqrt{d-2}(d-1)A^2n^{1/3}}{2d} + O\left(\frac{n^{1/3}}{A}\right)
	\end{align*}
	so that
	\[\Phi^2 = \frac{(d-2)(d-1)^2 A^4 n^{2/3}}{4d^2} + O(An^{2/3});\]
	and also
	\[T'' = (d-1)An^{2/3} + O(n^{2/3}/A^2).\]
	Substituting these estimates into \eqref{BrownianStep3} gives
	\[ P \ge \frac{c}{A^{3/2}} \exp\left(-\frac{(d-2)(d-1)A^3}{8d^2}\right),\]
	and the proof is complete.
\end{proof}

\section*{Acknowledgements}
Both authors would like to thank the Royal Society for their generous funding, of a PhD scholarship for UDA and a University Research Fellowship for MR.

\bibliographystyle{plain}
\def\cprime{$'$}

\end{document}